\documentclass{amsart}

\usepackage{amsmath,amssymb,amsthm, mathrsfs, mathtools, bm, amsfonts}
\usepackage{mathabx}
\usepackage{amscd,mathtools}
\usepackage[utf8]{inputenc}
\usepackage[english]{babel}
\usepackage{cite}
\usepackage{color}
\usepackage[pagebackref,colorlinks,citecolor=blue,linkcolor=red]{hyperref}

\renewcommand*{\backrefalt}[4]{\ifcase #1 (Not cited).\or (Cited p.~#2).\else (Cited pp.~#2).\fi} % Cited on...
\usepackage{float}
\usepackage{tikz}
\usepackage{lmodern}
\usetikzlibrary{decorations.pathmorphing}

\usepackage{multicol}
\tikzset{snake it/.style={decorate, decoration=snake}}
\usetikzlibrary{automata}
\usetikzlibrary{cd}
\usepackage[margin=3cm]{geometry}

\usepackage{comment}
\usepackage{mathtools}

\usepackage{csquotes}

\usepackage{amsthm}

\usepackage{tikz}
\usetikzlibrary{calc}
\usetikzlibrary{arrows}
\usetikzlibrary{decorations.pathreplacing}
\usetikzlibrary{intersections}

\usepackage{hyperref}
\usepackage{tikz}
\usepackage{caption, subcaption}
\usepackage{tikz-cd}
\usetikzlibrary{calc,decorations.pathreplacing}
\usepackage{tikz}
\usetikzlibrary{calc}
\usetikzlibrary{arrows}
\usetikzlibrary{decorations.pathreplacing}
\usetikzlibrary{intersections}

 \newtheorem{theorem}{Theorem}[section]
  \newtheorem{proposition}[theorem]{Proposition}

  \newtheorem{corollary}[theorem]{Corollary}
  \newtheorem{lemma}[theorem]{Lemma}

  \newtheorem{axiom}[theorem]{Axiom}

  \theoremstyle{definition}
  \newtheorem{definition}[theorem]{Definition}
  \newtheorem{claim}[theorem]{Claim}
  
  \newtheorem*{claim*}{Claim}

    \newtheorem{notation}[theorem]{Notation}

  \newtheorem*{question*}{Question}
  \newtheorem*{answer*}{Answer}
  \newtheorem*{application*}{Application}

  \theoremstyle{remark}
  \newtheorem{remark}[theorem]{Remark}
  \newtheorem*{remark*}{Remark}

\usepackage{amssymb}

\DeclarePairedDelimiterX{\Norm}[1]{\lVert}{\rVert}{#1}
  
\theoremstyle{definition}

  \newcommand{\hh}{{\sf h}}

  \newcommand{\mm}{{\sf m}}

 \newcommand{\calA}{\mathcal{A}}
  \newcommand{\calB}{\mathcal{B}}
  \newcommand{\calC}{\mathcal{C}}
  \newcommand{\calD}{\mathcal{D}}

  \newcommand{\calG}{\mathcal{G}}
  \newcommand{\calH}{\mathcal{H}}

  \newcommand{\calN}{\mathcal{N}}

  \newcommand{\calQ}{\mathcal{Q}}

  \newcommand{\calU}{\mathcal{U}}
    \newcommand{\calV}{\mathcal{V}}
  \newcommand{\calW}{\mathcal{W}}
  \newcommand{\calX}{\mathcal{X}}
    \newcommand{\calY}{\mathcal{Y}}
  \newcommand{\calZ}{\mathcal{Z}}

    \newcommand{\hT}{\widehat{T}}
    \newcommand{\hx}{\hat{x}}
     \newcommand{\ha}{\hat{a}}
     \newcommand{\bT}{\mathbb{T}}

      \newcommand{\hb}{\hat{b}}
          \newcommand{\hc}{\hat{c}}
       
    \newcommand{\hy}{\hat{y}}
    \newcommand{\hd}{\hat{\delta}}

    \newcommand{\hO}{\widehat{\Omega}}
    
    \newcommand{\hpsi}{\hat{\psi}}
    \newcommand{\hPsi}{\widehat{\Psi}}
        
    \newcommand{\hf}{\hat{f}}
    
    \newcommand{\BM}{\calB_{\min}}
    
    \renewcommand{\hh}{\mathbf{h}}
    
    \newcommand{\hpi}{\widehat{\pi}}

    \newcommand{\ES}{E_{\mathfrak S}}

    \newcommand{\Rel}{\mathrm{Rel}}
  
  \newcommand{\dist}{\mathbf{d}}

\newcommand{\diam}{\mathrm{diam}}
\newcommand{\hull}{\mathrm{hull}}

\newcommand{\gate}{\mathfrak{g}}

\newcommand{\nest}{\sqsubset}

\newtheorem{thmi}{Theorem}
\newtheorem{cori}[thmi]{Corollary}

\begin{document}

\title[Bicombing MCG and Teich via stable cubical intervals]{Bicombing the mapping class group and Teichm\"uller space via stable cubical intervals}

%authors

 \author   {Matthew Gentry Durham}
 \address{Department of Mathematics and Statistics, CUNY Hunter College, New York, NY }
 \email{matthew.durham@hunter.edu}

\begin{abstract}
In this mostly expository article, we provide a new account of our proof with Minsky and Sisto that mapping class groups and Teichm\"uller spaces admit bicombings.  More generally, we explain how the hierarchical hull of a pair of points in any colorable hierarchically hyperbolic space is quasi-isometric to a finite CAT(0) cube complex of bounded dimension, with the added property that perturbing the pair of points results in a uniformly bounded change to the cubical structure.  Our approach is simplified and new in many aspects.

\end{abstract}

\maketitle

\section{Introduction} \label{sec:intro}

The recent infusion of techniques from cube complexes into the study of mapping class groups of finite-type surfaces has had dramatic consequences, including toward our ability to control top-dimensional quasi-isometrically embedded flats \cite{HHS_quasi}, build bicombings \cite{DMS_bary, HHP}, and understand their geometry at infinity \cite{DZ22, Ham_boundary, Dur_infcube}, among others \cite{ABD, Petyt_quasicube, DZ22, zalloum2025effective}.

The central result facilitating most of the above work was the \emph{cubical model theorem} of Behrstock--Hagen--Sisto \cite{HHS_quasi}.  It says, roughly, that the coarse convex hull of a finite set of points in the mapping class group is quasi-isometrically approximated by a CAT(0) cube complex, which is a higher dimensional analogue of Gromov's tree approximation theorem for hyperbolic spaces.  The cubical model theorem was later generalized by Bowditch \cite{Bowditch_hulls} (to certain coarse median spaces) and the author \cite{Dur_infcube} (to allow for points at infinity).  All versions hold in the broader context of hierarchically hyperbolic spaces \cite{HHS_I}, a framework for studying a wide variety of spaces with blended negative and positive curvature, e.g. Teichm\"uller spaces \cite{Brock, Rafi:hypteich, Dur:augmented}, that built off the seminal work of Masur--Minsky \cite{MM99,MM00} on the mapping class group; see Section \ref{sec:HHS}. 

The cubical approximation theorem allows one to build any number of cubically-defined (and hierarchically well-behaved) paths between a pair of points by importing them from the cubical model of their hull.  This observation was at the core of two recent proofs that the mapping class group is \emph{bicombable}\footnote{Roughly, a bicombing on a metric space is a transitive family of quasi-geodesic paths that are stable under small perturbations of the endpoints; see Definition \ref{defn:bicombing}.}.  The two proofs, due to Haettel--Hoda--Petyt \cite{HHP} and Minsky, Sisto, and the author \cite{DMS_bary}, are very different but both depend crucially on this cubical model machinery.  See also a recent elegant proof of bicombability using wallspace techniques due to Petyt--Zalloum \cite{PZ_walls}.

The purpose of this paper is to explain a new proof of the two-point version of the main technical result of \cite{DMS_bary} via the machinery in \cite{Dur_infcube}, which we state here informally (see Theorem \ref{thm:stable cubes}):

\begin{thmi}\label{thmi:main}
The hierarchical hull of a pair of points in a colorable HHS can be stably approximated by a CAT(0) cube complex.
\end{thmi}

With this tool in hand, bicombability is a straight-forward application of cubical techniques essentially via the arguments in  \cite[Sections 5 and 6]{DMS_bary}:

\begin{cori}\label{cori:semihyp}
Let $S$ be a finite type surface.  Then its mapping class group and Teichm\"uller space with the Teichm\"uller metric admit bicombings by hierarchy paths.
\end{cori}

\begin{remark}
One can upgrade both of these statements to ensure coarse equivariance in the presence of an appropriate group action.  In particular, in \cite{DMS_bary} we showed that the bicombings in Corollary \ref{cori:semihyp} are coarsely equivariant with respect to the mapping class group action, and that the mapping class group is \emph{semi-hyperbolic} in the sense of \cite{AB95}.  This equivariance upgrade is relatively straight-forward once one has the bicombing, but we have omitted it since it adds yet another layer complexity to the statements and proofs in an already technical construction.  We point the reader to \cite[Remark 2]{DMS_bary} for a discussion of how equivariance works for this construction. 
\end{remark}

% More generally, these results hold in the context of \emph{hierarchically hyperbolic spaces} (HHSes), a framework for studying a wide variety of spaces with blended negative and positive curvature that built off the seminal work of Masur--Minsky \cite{MM99,MM00} on the mapping class group; see Section \ref{sec:HHS}.  In the HHS setting, the guiding philosophy is to study objects by projecting to a family of hyperbolic spaces and then use techniques from hyperbolic geometry.  The cubical model theorem is utilized similarly---one exports problems from the mapping class group to a family cube complexes and employs cubical techniques.

We have several motivations for writing this paper.  First, we wanted to give a simplified account of the cubical model results in \cite{DMS_bary, DMS_FJ} in less generality.  Roughly speaking, the cubical model for a finite subset $F$ of an HHS from \cite{HHS_quasi} is produced from a wallspace on a family of trees obtained by projecting $F$ to each of the hyperbolic spaces in the hierarchy.  The proof of bicombability in \cite{DMS_bary} proceeds by using hyperbolic geometry techniques---culminating in the Stable Tree Theorem \cite[Theorem 3.2]{DMS_bary}---to produce a family of trees which is coarsely stable under bounded perturbations of $F$.  In this paper, we simplify a number of aspects of this initial portion of the argument by limiting ourselves to modeling only a pair of points instead of arbitrary collections of points and rays.

The second motivation is to explain how to combine the stability techniques from \cite{DMS_bary} with the new cubical model construction from \cite{Dur_infcube}\footnote{This construction has many extra benefits as compared to the construction from \cite{HHS_quasi}.  We point interested readers to the in-depth discussion in the introduction of \cite{Dur_infcube}.}.  This involves a new, more elementary proof of the key Stable Trees Theorem for intervals; see Section \ref{sec:stable intervals} below.  Notably, with Minsky and Sisto, we recently \cite{DMS_FJ} implemented a significant expansion of the techniques from \cite{DMS_bary} utilizing the framework of \cite{Dur_infcube}\footnote{The first version of this expository article appeared before \cite{DMS_FJ}, but we have subsequently adapted it to account for that paper.}.  This allowed us to build an asymptotically CAT(0) metric \cite{Kar_asymp} on any colorable HHS, construct a $\calZ$-boundary (in the sense of \cite{bestvina1996local, Dra:BM_formula}) for any colorable HHG (see also \cite{hamenstadt2025z} for mapping class groups), and establish the Farrell--Jones Conjecture \cite{farrell1993isomorphism} for many HHGs, including extra-large type Artin groups.  This paper can be seen as explaining the simplest possible case of the main technical construction in \cite{DMS_FJ}.

Unsurprisingly, combining the work in \cite{DMS_bary} and \cite{Dur_infcube} involves an exposition of both, though we have mostly focused on where they meet.  We summarize the main features of this (mostly) expository paper as follows:

\begin{itemize}
\item We give an exposition of the basic cubical model construction from \cite{Dur_infcube}, which is significantly simplified from the much more general setting considered therein, (i.e., which allows for modeling points at infinity).

\item We explain how to combine the stability techniques from \cite{DMS_bary} with the cubical model construction of \cite{Dur_infcube}.  In particular, this involves defining a refined notion of a \emph{stable decomposition} (Definition \ref{defn:stable decomp}), and giving a new proof of the Stable Trees Theorem \ref{thm:stable intervals} for intervals.  We note that these stable decompositions also appear in \cite{DMS_FJ}.

\item Using the above, we prove bicombability of mapping class groups (and HHSes) via the stable cubulation techniques from \cite{DMS_bary}, while restricting ourselves to the simplified setting of cubical models for pairs of points.  Notably, the bicombing construction from \cite{HHP} also depends on the cubical models (from \cite{HHS_quasi}), so one can also see this paper as supplying an exposition for some of the underlying technology of their construction.

\item By focusing on the bicombing result (Corollary \ref{cori:semihyp}), we are able significantly simplifly the endgame of the path construction.  In particular, we utilize Niblo--Reeves paths \cite{NibloReeves} in the cubical model instead of the more elaborate version of Niblo--Reeves contractions from our work with Minsky and Sisto \cite[Section 5]{DMS_bary}.
\end{itemize}

\subsection{Background reading}

Our main focus in this article is explaining how to adapt the stable cubulation techniques from \cite{DMS_bary} to the  cubical model setup from \cite{Dur_infcube}, which we note is done in much more general setting in \cite{DMS_FJ}.  As such, the reader will benefit from a variety of contextual information and motivation in and around cube complexes and HHSes.  We provide some suggested reading in the list below:
\begin{itemize}
	\item  See Sisto's notes \cite{Sisto_HHS} and the background in Casals-Ruiz--Hagen--Kazachkov \cite{CRHK} for an introduction to HHSes.  While we do give a proof of the harder bound of the Distance Formula in Corollary \ref{cor:DF lower bound}, we expect the reader to be familiar with the general framework.
	\item  See Farb--Margalit \cite{FM_primer} for an introduction to mapping class groups, and \cite[Section 11]{HHS_II} and there references therein (particular, Masur--Minsky's work \cite{MM99, MM00}) for its hierarchically hyperbolic structure.  Similarly, see Rafi \cite{Rafi:combo, Rafi:hypteich} and the author \cite{Dur:augmented} for the hierarchical structure on Teichm\"uller space with the Teichm\"uller metric.
	\item For an introduction to bicombings and semihyperbolicity, see Alonso--Bridson \cite{AB95}.
	\item For background on cube complexes, see Sageev's notes \cite{Sageev:pcmi} and Schwer's book \cite{schwer2324cat}.
	\item  For a discussion of cubical techniques in HHSes and related objects, see Zalloum's expository article\cite{Zalloum_expo}.  
\end{itemize}

\subsection{Structure of the paper}

The paper begins with a high-level discussion of the endgame for the bicombing construction using cubical approximation techniques.  In particular, in Section \ref{sec:endgame}, we will extract the key facts that we prove about colorable hierarchically hyperbolic spaces, namely the Stable Cubes Theorem \ref{thm:stable cubes} and Stable Moves Theorem \ref{thm:stable moves}, to build a bicombing on any metric space possessing the properties described in the former.  This includes the much simplified construction of the bicombing paths themselves via Niblo--Reeves paths in Subsection \ref{sec:NR stability}.

With this setup as motivation, we then proceed to describe the cubical model construction for HHSes.  Sections \ref{sec:interval systems} through \ref{sec:cubulation} explain the interval version of the cubical model theorem from \cite{Dur_infcube}.  The heart of the paper is in Sections \ref{sec:controlling domains} and \ref{sec:stable intervals}, where we explain how to build the stabilized family of intervals which serve as input, as well as in Section \ref{sec:stable cubes}, where we see how to plug this stabilized input into the cubulation machine.
 
\subsection{Acknowledgements}

This work was supported in part by NSF grant DMS-1906487 and NSF CAREER DMS-2441982.  The main thanks go to Yair Minsky and Alessandro Sisto, as our papers \cite{DMS_bary, DMS_FJ} are at the center of this article.  Thanks to Thomas Koberda for his efforts in coorganizing the first Riverside Workshop on Geometric Group Theory\footnote{\url{https://sites.google.com/view/riggthew-2023/home}} in 2023 (and its subsequent versions), which was supported by NSF grant DMS-2234299.  Thanks also to Harry Petyt for giving many detailed comments which greatly improved the paper.  This paper was written for the conference proceedings and I appreciated Thomas' patience and encouragement while completing it.  Thanks also to the other speakers---Anne Lonjou, Marissa Kawehi Loving, and Jean-Pierre Mutanguha---for their outstanding minicourses.  Finally, thanks to the attendees of RivGGT23 where my minicourse covered some of the material in this paper, and the UCR graduate students who attended my topics course on the construction in \cite{Dur_infcube}.

\section{Beginning at the end} \label{sec:endgame}

The purpose of this section is to prove the main consequence of the technical results we will spend the rest of the paper establishing.  That is, we will begin by explaining how the Stable Cubical Intervals Theorem \ref{thmi:main}  (Theorem \ref{thm:stable cubes}) is enough to prove bicombability of the mapping class groups and colorably hierarchically hyperbolic spaces, more generally.

We begin with a brief primer on CAT(0) cube complexes and a discussion of Niblo--Reeves paths and their stability under hyperplane deletions.  In Subsection \ref{subsec:stable cubes, early}, we give a precise version of our stable cubulations result along with its associated path-building construction.  In Subsection \ref{subsec:semihyp}, we then define and discuss bicombings.  In Subsection \ref{subsec:building bicombing}, we explain how to put these parts together to give a bicombing for any metric space satisfying some version of our stable cubulations setup.

\subsection{Wall spaces, cube complexes, and hyperplane deletions}\label{subsec:CCC background}

%@@@@Adapt to median setup

This subsection includes a quick set of definitions and references for readers unfamiliar with CAT(0) cube complexes.

Our approach here follows most closely Roller's approach \cite{Roller} to Sageev's cubulation machine \cite{Sageev:cubulation}, which uses the notion of wall-space as introduced by Haglund-Paulin \cite{HP_wallspace}.  See Sageev's notes \cite{Sageev:pcmi}, as well as Hruska-Wise \cite{HW_CCC}, for excellent expositions on this subject.  We note that we will later take an approach using medians in Section \ref{sec:cubulation}.  

A \emph{cube complex} is a CW complex obtained by gluing a collection of unit Euclidean cubes of various dimensions along their faces by a collection of (Euclidean) isometries.  A cube complex $\Sigma$ is \emph{non-positively curved} when it satisfies Gromov's link condition, and every non-positively curved simply connected cube complex admits a unique CAT(0) metric.

A \emph{hyperplane} in a CAT(0) cube complex $\Sigma$ is a connected subspace whose intersection with each cube $\sigma = [0,1]^{n}$ is either $\emptyset$ or the subspace obtained by restricting exactly one coordinate to $\frac{1}{2}$.  In particular, every hyperplane partitions $\Sigma$.  This motivates the definition of a half-spaces and pocsets, following Roller \cite{Roller}.

\begin{definition}[Pocset]\label{defn:pocset}
Let $(\Sigma,<)$ be a partially-ordered set.  A \emph{pocset} structure on $\Sigma$ has the additional structure of an order-reversing involution $A \mapsto A^*$ satisfying
\begin{enumerate}
\item $A \neq A^*$ and $A, A^*$ are $<$-incomparable;
\item $A < B \implies B^* < A^*$.
\end{enumerate}
\begin{itemize}
\item Given $A<B$, the \emph{interval} between $A$ and $B$ is $[A,B] = \{C|A<C<B\}$.
\item The pocset $(\Sigma,<)$ is \emph{locally finite} if every interval is finite.
\item Given $A,B \in \Sigma$, we say $A$ is \emph{transverse} to $B$ if none of $A<B, A<B^*, A^*<B, A^*<B^*$ holds.  The pocset $(\Sigma,<)$ has \emph{finite width} if the maximal size of a pairwise transverse subset is finite.

\end{itemize}
\end{definition}

In a CAT(0) cube complex, every vertex is contained in exactly one half-space associated to each hyperplane.  The following notion abstracts the combinatorial structure of such a family of half-spaces:

\begin{definition}[Ultrafilters and the DCC]\label{defn:ultrafilter}
Given a locally finite pocset $(\Sigma, <)$, an \emph{ultrafilter} $\alpha$ on $\Sigma$ is a subset of $\Sigma$ satisfying
\begin{enumerate}
\item (Choice) Exactly one element of each pair $\{A,A^*\}$ is an element of $\alpha$;
\item (Consistency) $A \in \alpha$ and $A< B \implies B \in \alpha$.
\end{enumerate}
\begin{itemize}
\item An ultrafilter $\alpha$ satisfies the \emph{descending chain condition} (DCC) if every descending chain of elements terminates.
\end{itemize}
\end{definition}

We can now define the dual cube complex to a locally finite, discrete, finite-width pocset:

\begin{definition}[Dual cube complex]\label{defn:dual CCC}
Given a locally finite, discrete, finite-width pocset $(\Sigma,<)$, we define a CW complex $X(\Sigma)$ as follows:
\begin{enumerate}
\item The $0$-cells $X^{(0)}$ of $X$ are the DCC ultrafilters on $(\Sigma, <)$.
\item Two $0$-cells $\alpha, \beta$ are connected by an edge in $X^{(1)}$ if and only if $|\alpha \triangle \beta| = 2$, i.e. $\alpha = (\beta - \{A\}) \cup \{A^*\}$ for some $A \in \Sigma$.
\item Define $X^{(n)}$ inductively by adding an $n$-cube whenever the boundary of one appears in the $X^{(n-1)}$-skeleton.
\end{enumerate}
\end{definition}

The following is a theorem of Roller \cite{Roller}, building on Sageev \cite{Sageev:cubulation}:

\begin{theorem}[Sageev's machine]\label{thm:Sageev}
Given a locally finite, finite-width pocset $(\Sigma,<)$, the CW complex $X(\Sigma)$ is a finite dimensional CAT(0) cube complex, where the dimension is equal to the width of $(\Sigma, <)$.
\begin{itemize}
\item We call $X(\Sigma)$ the \emph{dual cube complex} to $(\Sigma, <)$.
\end{itemize}
\end{theorem}

We note that the dual cube complex to the pocset coming from the hyperplanes of a CAT(0) cube complex is (isomorphic to) that cube complex.

The following lemma is useful for obtaining isomorphisms between cube complexes whose wall-spaces we can directly compare.  Its proof is an exercise in the basics of CAT(0) cube complexes which we leave to the reader.

\begin{lemma}\label{lem:halfspaces_bijection}
 Let $\calW$, $\calW'$ be wallspaces, and let $\iota:\calH_{\calW} \to \calH_{\calW'}$ be a bijection of their half-spaces, which preserves
 complements and disjointness.  Then $\iota$ induces a cubical isomorphism
$h:\calD(\calW)\to \calD(\calW')$
between the corresponding dual CAT(0) cube complexes.
\end{lemma}

Finally, in our stable cubulations theorem, we will want to realize maps between certain cube complexes as a sort of collapsing of a controlled piece of their combinatorial data.  The right notion here is of a \emph{hyperplane deletion}.  The motivating example is where one has a wall-space $\calW(\calH)$ on some set of walls $\calH$, and one removes some collection of the walls $\calH - \calG$, to obtain a \emph{restriction map} $\calD(\calW(\calH)) \to \calD(\calW(\calH- \calG))$ between their dual cube complexes in the sense of \cite{CapraceSageev}.  This map collapses the carriers of the hyperplanes in $\calG$ (recall that a carrier of a hyperplane $\hh$ is naturally isomorphic to $\hh \times [0,1]$).

\subsection{Stable cubes and stable moves}\label{subsec:stable cubes, early}

We will prove Corollary \ref{cori:semihyp} in an ostensibly more general framework than that of (colorable) HHSes\footnote{We actually suspect that coarse median spaces satisfying some version of the property in Definition \ref{defn:stable cubes} will be HHSes, though it is unclear to us what exactly the right property is, e.g. perhaps a factor system in the sense of \cite{HHS_I}.}.  That is, we only need a metric space which satisfies the following definitions.  The first is effectively a weak two-point version of the cubical model property for HHSes from \cite{HHS_quasi}, proven in Theorem \ref{thm:Q CCC} in this paper.

\begin{definition}[Weak local quasi-cubicality]\label{defn:lqc pair}

Let $X$ be a metric space.  A $B_0$-\emph{cubical model} for a pair $a,b \in X$ is a CAT(0) cube complex $\calQ$ of dimension at most $B_0$ satisfying the following:
\begin{enumerate}
\item There is a $(B_0,B_0)$-quasi-isometric embedding $\hO:\calQ \to X$ (the \emph{cubical model map}),
\item There are two \emph{distinguished vertices} $\ha,\hb \in \calQ$ so that any $\ell^{\infty}$-geodesic $\gamma \subset \calQ$ between $\ha,\hb$ is sent to a $(B_0,B_0)$-quasi-geodesic $\hO(\gamma) \subset X$ between $a,b$.
\end{enumerate}
\begin{itemize}
\item We say $X$ is $B_0$-\emph{weakly locally quasi-cubical} (wLQC) if there exists $B_0>0$ so that any pair $a,b\in X$ admits a $B_0$-cubical model.
\end{itemize}
\end{definition}

 While CAT(0) cube complexes have many standard metrics (all of which are quasi-isometric), the convenient one for us is the $\ell^{\infty}$-metric, which, for vertices in a single unit cube, is just the $\sup$-metric. 

The next definition is extracted from Theorem \ref{thm:stable cubes} below.  This property holds for colorable hierarchically hyperbolic spaces, e.g., the mapping class group and Teichm\"uller space; see \cite[Theorem 4.1]{DMS_bary} for the original, more general version, and \cite[Theorem 9.1]{DMS_FJ} for the state-of-the-art.  Part of our purpose here for extracting it from the HHS setting is to highlight the structure of the endgame.

\begin{definition}[Stable weak local quasicubicality]\label{defn:stable cubes}
We say that a metric space $X$ is \emph{$B_0$-stably wLQC} if there is a $B_0$ so that the following holds:  

\begin{itemize}
\item For any $a,b, a',b' \in X$ with $d_X(a,a') \leq 1$ and $d_X(b,b')\leq 1$, there exist $B_0$-cubical models $\calQ, \calQ_0$ and $\calQ',\calQ'_0$ for $a,b$ and $a',b'$ so that the following diagram commutes up to additive error $B_0$:
\begin{equation}\label{Phi diagram, early}
  \begin{tikzcd}
   \calQ \arrow[ddrr,"\hO", bend left=40] \arrow[dr,"\Delta_0 \hspace{.075in}" left] &  \\
    &\calQ_{0} \arrow[dr,"\hO_0 \hspace{.2in}" below]\arrow[dd, "I_0"] \\
    & & X\\
    & \calQ'_{0} \arrow[ur,"\hO'_0"] \\
    \calQ'\arrow[uurr,"\hspace{.2in} \vspace{.1in} \hO'" below, bend right=40] \arrow[ur,"\Delta'_0"] & \\
  \end{tikzcd}
  \end{equation}
 \item In the above,
  \begin{enumerate}
  \item  $\hO, \hO_0, \hO', \hO'_0$ are cubical model maps and hence $(B_0,B_0)$-quasi-isometric embeddings,
  \item the map $I_0$ is a cubical isomorphism,
  \item the maps $\Delta_0, \Delta'_0$ are induced by deleting at most $B_0$-many hyperplanes, and
  \item $\Delta_0, \Delta'_0$ take the distinguished vertices associated to $a,b$ and $a',b'$ in $\calQ,\calQ'$ to those in $\calQ_0,\calQ'_0$, respectively.
  \end{enumerate}
  \end{itemize}
\end{definition}

Requirement (4) says that the way that the cube complexes $\calQ,\calQ_0$ model $a,b$ align, and similarly for $\calQ', \calQ'_0$ and $a',b'$.  Note that if $X$ satisfies Definition \ref{defn:stable cubes} then it satisfies Definition \ref{defn:lqc pair}.  

\begin{remark}
Theorems \ref{thm:Q CCC} and \ref{thm:stable cubes} actually provide a lot more information about these cubical models.  In particular, if $\calX$ is an HHS and $a,b \in \calX$, then the cube complexes $\calQ$ and $\calQ_0$ above model the \emph{hierarchical hull} of $a,b$ (Definition \ref{defn:hier hull}), which is a coarsely convex subspace which encodes the hierarchical geometry between $a,b$.  Moreover, the combinatorial geometry of the cube complex encodes essentially all of this hierarchical data.  See the introduction of \cite{Dur_infcube} for a detailed discussion.
\end{remark}

\subsection{Stable moves}\label{sec:NR stability}

In this subsection, we establish the key path-building portion of our bicombing construction, namely the Stable Moves Theorem \ref{thm:stable moves}.  This involves proving that a well-studied family of paths in a cube complex due to Niblo--Reeves \cite{NibloReeves} behaves well under hyperplane deletion.  These paths were at the center of their construction of a biautomatic structure on an CAT(0) cubical group, and so it is quite agreeable that they play a central role in our construction of a bicombing for our sort of quasi-cubical spaces.  While we have maintained some of the structure and terminology from our work with Minsky and Sisto in \cite[Section 5]{DMS_bary}, our limited setting allows for a much simpler path construction.

Let $X$ be a finite CAT(0) cube complex with $X = \hull_X(a,b)$ the cubical hull\footnote{The  assumption that $X = \hull_X(a,b)$ is not necessary, but is convenient for our purposes.} of $a,b \in X^0$.  Set $\calH = \calH(a,b)$ to be the set of hyperplanes in $X$ separating $a,b$, so that $X = \calX(\calH)$ is the dual cube complex to $\calH$.  Recall that we can think of vertices of $X$ as coherent orientations on hyperplanes in $\calH$.

\begin{definition}
Let $p \in X^0$.  A hyperplane $\hh \in \calH$ is \emph{transitional}, denoted $\hh \in T(p,\calH)$, if
\begin{enumerate}
\item $\hh$ separates $p$ from $b$, and
\item $p$ is \emph{adjacent} to $\hh$, i.e. $\hh$ crosses some cube of which $p$ is a vertex.
\end{enumerate}
\end{definition}

We note that there is a natural partial order $\prec$ on $\calH$, where $\hh \prec \hh'$ when $\hh$ separates $\hh'$ from $a$, or equivalently $\hh'$ separates $\hh$ from $b$.  In our special setting with $X = \hull_X(a,b)$, every vertex $p \in X^0$ occurs on some $\ell^1$-geodesic from $a$ to $b$.  The property of transitionality for $\hh \in T(p,\calH)$ is encoding that $\hh$ is minimal in the partial order on $\calH(p,b) \subset \calH$.

The following lemma says that minimal elements in the partial order must cross:

\begin{lemma}\label{lem:trans cross}
If $\hh_1,\hh_2 \in T(p,\calH)$, then $\hh_1$ crosses $\hh_2$.
\end{lemma}

\begin{proof}
If $\hh_1$ and $\hh_2$ do not cross, then, without loss of generality, there is a half-space of $\hh_2$ containing $a, \hh_1$.  On the other hand, $p$ is adjacent to $\hh_2$.  This implies that $p,a$ are in different half-spaces of $\hh_1$, which is a contradiction.
\end{proof}

With this lemma, we can define our \emph{Niblo--Reeves} paths $p_0=a, p_1, \dots, p_n = b$ as follows \cite{NibloReeves}:

\begin{itemize}
\item Let $p_0 = a$.
\item By Lemma \ref{lem:trans cross}, $p_0$ is the corner of a unique cube $C_0$ for which the set of hyperplanes that cross $C_0$ is precisely $T(p,\calH)$.  Set $p_1$ to be the diagonally opposite corner of $C_0$.  Equivalently, $p_1$ is obtained from $p_0$ by flipping orientations on the hyperplanes in $T(p,\calH)$.
\item Define $p_i$ for $i \geq 1$ recursively with $p_i$ in place of $p_0$.
\end{itemize}

Thus, at each step $p_i \mapsto p_{i+1}$, the path crosses the set of $\prec$-minimal hyperplanes in $\calH(p_i,b)$.  The following proposition states the basic properties of the Niblo--Reeves paths:

\begin{proposition}\label{prop:NR basics}
Continuing with the above notation, the following hold:
\begin{enumerate}
\item There exists some $n$ so that $p_n=b$.
\item For all $0 \leq i \leq n$, we have $d^{\infty}_X(p_{i}, p_{i+1}) =1$.
\item Any path obtained by concatenating $\ell^1$-geodesics between $p_i, p_{i+1}$ for $0 \leq i < n$ is an $\ell^1$-geodesic.
\end{enumerate}
\end{proposition}

\begin{proof}
For (1), observe that the sequence $p_i$ only crosses hyperplanes separating $a$ from $b$.  Each step $p_i \mapsto p_{i+1}$ of the sequence crosses the collection of $\prec$-minimal hyperplanes in $\calH(p_i, b)\subset \calH$.  Now (1) is equivalent to the claim that $\calH = \bigcup_i \calH(p_i,b)$, which is to say that this process exhausts the set of hyperplanes separating $a$ and $b$, which is clear because $\calH$ is finite.

Item (2) is by construction, since $p_i$ and $p_{i+1}$ are opposite corners of a cube.

Finally, item (3) follows from the fact that $\calH(p_i, b) \subset \calH(p_{i+1},b)$ with $\calH(p_{i+1},b) - \calH(p_i, b) = \calH(p_i,p_{i+1})$.  Hence if $\hh \in T(p_i,b)$, then $\hh$ does not separate $p_i$ from $p_j$ for $j<i$.  This says that the sequence $\{p_i\}$ does not cross any hyperplane twice, and it follows that the concatenation of the $\ell^1$-geodesic segments will not either, making it an $\ell^1$-geodesic.
\end{proof}

The following is a useful observation about this construction:

\begin{proposition}\label{prop:NR iso}
Let $\Phi:\calQ \to \calQ'$ be an isomorphism of cube complexes.  If $\gamma$ is the Niblo--Reeves path between $a,b \in \calQ^0$, then $\Phi(\gamma)$ is the Niblo--Reeves path between $\Phi(a), \Phi(b) \in \left(\calQ'\right)^{0}$.
\end{proposition}

We can now state the main result of this subsection.

\begin{theorem}[Stable moves]\label{thm:stable moves}
Suppose that $\hh \in \calH$ is hyperplane, and let $\Delta_{\hh}:\calD(\calH) \to \calD(\calH - \{\hh\})$ denote the hyperplane collapse map.  Set $\Delta_{\hh}(a) = a'$ and $\Delta_{\hh}(b) =b'$.  Let $p_0,\dots,p_n$ denote the Niblo--Reeves path in $X = \calD(\calH)$  between $a,b$ and $p'_0, \dots, p'_m$ the path between $a',b'$ in $\calD(\calH- \{\hh\})$.  Then the following hold:
\begin{enumerate}
\item We have $n \leq m+1$.
\item For all $1 \leq i \leq n$, we have $d^{\infty}_{\calD(\calH - \{\hh\})}(\Delta_{\hh}(p_i), p'_i) \leq 1$.
\end{enumerate}
\end{theorem}

In other words, Niblo--Reeves paths are stable under hyperplane deletion maps.

The following notation will be useful for us:

\begin{definition}
Let $\hh' \in \calH$.  We set $t_{\hh'}(a, \calH)$, called the \emph{transition index} of $\hh'$, to be the index $i$ at which $\hh' \in \calH(p_i, p_{i+1})$.  Similarly, if $\hh' \in \calH - \{\hh\}$, we define $t_{\hh'}(a', \calH - \{\hh\})$ to be the index $i$ so that $\hh' \in \calH(p'_i, p'_{i+1})$.
\end{definition} 

\begin{lemma}\label{lem:move control}
If $\hh' \in \calH - \{\hh\}$, then $t_{\hh'}(a,\calH) = t_{\hh'}(a', \calH - \{\hh\}) + \delta$ for $\delta \in \{0,1\}$.
\end{lemma}

\begin{proof}
This is essentially \cite[Lemma 5.12]{DMS_bary}, though the proof is basically immediate in our setting.  Note that the transition index $t_{\hh} = d_{\infty}(a, \hh) + \frac{1}{2}$.  Since the $\ell^{\infty}$ distance between $a'$ and $\hh'$ measures the maximal length of a chain of hyperplanes separating them (plus $\frac{1}{2}$), that distance can be at most $1$ more than $d_{\infty}(a,\hh)$.  This completes the proof.
\end{proof}

\begin{proof}[Proof of the Stable Moves Theorem \ref{thm:stable moves}]
Suppose $\hh_1, \hh_2$ separate $\Delta_{\hh}(p_i)$ and $p'_i$ in $X' = \calD(\calH - \{\hh\})$ but do not cross, since we are done if all such hyperplanes cross.

Without loss of generality, we may assume that $\hh_1$ separates $\Delta_{\hh}(p_i)$ from $\hh_2$.  Let $j = t_{\hh_1}(a',\calH - \{\hh\})$.  Then $j < t_{\hh_2}(a',\calH - \{\hh\}) < i$, with the latter inequality following because $p'_{i+1}$ has already crossed $\hh_2$ (and $\hh_1$).  Thus $j + 1 < i$.

But then $t_{\hh_1}(a, \calH) \leq t_{\hh_1}(a', \calH - \{\hh\}) + 1 = j+1 < i$, which says that $\hh_1$ separates $p_i$ from $a$, which is a contradiction.
\end{proof}

\subsection{Bicombings}\label{subsec:semihyp}

Our ultimate goal is to prove that mapping class groups and Teichm\"uller spaces admits bicombings (Theorem \ref{thmi:main}).  Roughly, a bicombing on a space $X$ is a family of uniform quasi-geodesic paths that is transitive (i.e., connects any pair of points of $X$), and for which paths in the family between pairs of endpoints that are close to each other uniformly fellow-travel in a parameterized fashion.   The notion of bicombing was first introduced by Thurston in the context of algorithmic problems in groups.  It is a straightforward but informative exercise to prove that geodesics in hyperbolic and CAT(0) spaces satisfy this property, and this is precisely what the definition is intended to generalize.

\smallskip
The following definition is essentially from \cite{AB95}, as restated in \cite[Definition 6.4]{DMS_bary}:  

\begin{definition}[Bicombings]\label{defn:bicombing}
A \emph{discrete, bounded, quasi-geodesic bicombing} of a metric space $X$ consists of a family of discrete paths $\{\Gamma_{x,y}\}_{x,y \in X}$ and a constant $\zeta>0$ satisfying the following:
\begin{enumerate}
\item \emph{Quasi-geodesic}: For any $x,y \in X$ with $d = d_{X}(x,y)$, there exists $n_{x,y} \leq d\zeta + \zeta$ so that the path $\Gamma_{x,y}:\{0, \dots, n_{x,y}\} \rightarrow X$ is a $(\zeta,\zeta)$-quasi-isometric embedding with $\Gamma_{x,y}(0)=x$ and $\Gamma_{x,y}(n_{x,y}) = y$; and
\item \emph{Fellow-traveling}: If $x',y' \in X$ and $d_X(x,x'), d_X(y,y') \leq 1$, then for all $t \in \{0, \dots, \max\{n_{x,y}, n_{x',y'}\}\}$, we have
$$d_X(\Gamma_{x,y}(t), \Gamma_{x',y'}(t)) \leq \zeta.$$
\end{enumerate}

For our purposes, the fellow-traveling condition in item (2) is the key challenge.  At this point, there are many constructions of uniform quasigeodesics in mapping class groups.

\end{definition}

\subsection{Building the bicombing}\label{subsec:building bicombing}

We are now ready to define the bicombing.  Let $X$ be $B_0$-stably wLQC (Definition \ref{defn:stable cubes}).  For any $a,b \in X$, let $\calQ$ be a CAT(0) cube complex provided by Definition \ref{defn:stable cubes} and $\ha,\hb\in \calQ$ the distinguished vertices associated to $a,b$.

We can define a family of paths $\{\Gamma_{a,b}\}_{a,b \in X}$ by taking $\Gamma_{a,b}$ to be the image under the quasi-isometry $\hO:\calQ \to X$ of the Niblo--Reeves path between $\ha,\hb$.  By Definition \ref{defn:lqc pair}, this is a uniform $(B_0,B_0)$-quasi-geodesic between $a,b$.

\begin{theorem}\label{thm:bicombing}
The family of paths $\{\Gamma_{a,b}\}_{a,b \in X}$ forms a discrete bounded quasi-geodesic bicombing on $X$.
\end{theorem}

\begin{figure}
    \centering
    \includegraphics[width=.8\textwidth]{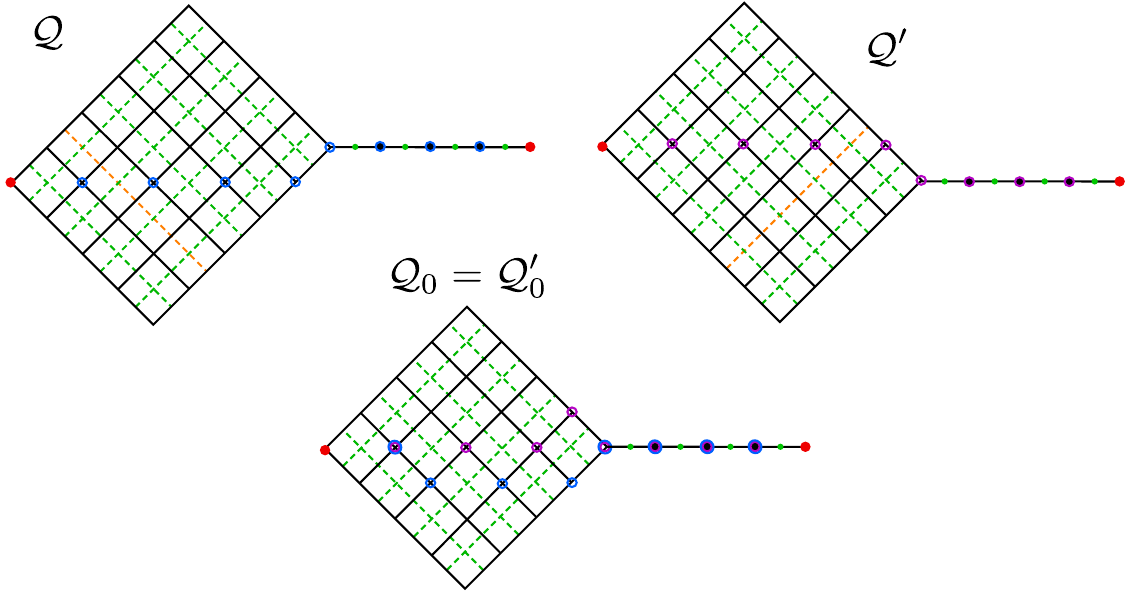}
    \caption{Stable Moves: A simple example of how Niblo--Reeves paths behave under hyperplane collapses.  In each cube complex, the red vertices correspond to $a,b$.  The orange hyperplanes are collapsed when passing to $\calQ_0 = \calQ'_0$.  The Niblo--Reeves path in $\calQ_0$ between $a,b$ travels in a straight line between the red vertices.  The Niblo--Reeves paths in $\calQ$ (blue) and $\calQ'$ (purple) descend to paths which are close to it the Niblo--Reeves path in $\calQ_0$, and hence close to each other.  This example is a modified version of \cite[Figure 25]{DMS_bary}.}
    \label{fig:bicoming_example}
\end{figure}

\begin{proof}
We need to confirm the properties in Definition \ref{defn:bicombing}.  The paths are clearly discrete and are $(B_0,B_0)$-quasi-geodesics by assumption.  So we need to show that they satisfy the fellow-traveling property in item (2) of that definition.  This is a direct application of the properties in Definition \ref{defn:stable cubes} and Theorem \ref{thm:stable moves}.

Let $\ha,\hb \in \calQ$ and $\ha_0,\hb_0 \in \calQ_0$ be the distinguished vertices associated to $a,b$, and the same for $\ha',\hb' \in \calQ'$ and $\ha'_0,\hb'_0 \in \calQ'_0$ with $a',b'$, respectively.  Note that $\Delta_0(\ha) = \ha_0$, etc, by item (4) of Definition \ref{defn:stable cubes}.  Let $\Lambda, \Lambda_0, \Lambda',\Lambda'_0$ be the Niblo--Reeves paths in $\calQ,\calQ_0,\calQ',\calQ'_0$ connecting the respective pairs $(\ha,\hb), (\ha_0,\hb_0), (\ha',\hb'), (\ha'_0,\hb'_0)$ in their respective cube complexes $\calQ,\calQ_0,\calQ',\calQ'_0$.  So in our notation, $\Gamma_{a,b} = \hO(\Lambda)$ and $\Gamma_{a',b'} = \hO'(\Lambda')$.  We have the following estimate
\begin{eqnarray}
d_X(\Gamma_{a,b}(i), \Gamma_{a',b'}(i)) &=& d_X(\hO(\Lambda(i)), \hO'(\Lambda'(i)))\\
&\leq& d_X(\hO_0 \circ \Delta_0(\Lambda(i)), \hO'_0 \circ \Delta'_0 (\Lambda'(i))) + 2B_0\\
&\leq& d_X(\hO_0(\Lambda_0(i)), \hO'_0(\Lambda'_0(i))) + 4B_0\\
&\leq& d_X(\hO'_0 \circ I_0(\Lambda_0(i)), \hO'_0(\Lambda'_0(i))) + 5B_0\\
&=& d_X(\hO'_0(\Lambda'_0(i)), \hO'_0(\Lambda'_0(i)))\\
&=& 5B_0.
\end{eqnarray}
In the above, (3) follows from commutativity of the top and bottom triangles in diagram \ref{Phi diagram, early}, (4) follows from Theorem \ref{thm:stable moves}, (5) follows from commutativity the middle triangle of diagram \ref{Phi diagram, early}, and (6) follows from the fact that the cubical isomorphism $I_0:\calQ_0 \to \calQ'_0$ identifies the paths $\Lambda_0$ and $\Lambda'_0$, i.e. $\Lambda_0(i) = \Lambda'_0(i)$ for each $i$ (Proposition \ref{prop:NR iso}).  This completes the proof.

\end{proof}

With this in hand, we can prove our bicombing result:

\begin{proof}[Proof of Corollary \ref{cori:semihyp}]
Mapping class groups and Teichm\"uller spaces of finite type surfaces are colorable HHSes (see Section \ref{sec:HHS} for references), so they are stably weakly locally quasi-cubical by Theorem \ref{thm:stable cubes}.  Hence Theorem \ref{thm:bicombing} provides a bicombing by pushing forward the Niblo--Reeves path in an appropriate cubical model to the respective space.  Any such $\ell^{\infty}$-geodesic can be completed to an $\ell^1$-geodesic by connecting vertices via $\ell^1$-geodesic segments as in Proposition \ref{prop:NR basics}.  But now Proposition \ref{prop:cubical hp} says that the images of these paths are hierarchy paths (Definition \ref{defn:hp}).  This completes the proof.
\end{proof}

\begin{remark}\label{rem:no medians}
A reader familiar with both \cite{HHS_quasi} and \cite{DMS_bary} will note that one usually uses the fact that the cubical model maps are \emph{quasi-median} maps to get this last hierarchy path property.  However, we have avoided discussions of median structures in this paper because they are unnecessary for our purposes.  See \cite[Section 15]{Dur_infcube} for how our cubical model construction works with medians.
\end{remark}

\section{A brief primer on HHSes}\label{sec:HHS}

Hierarchically hyperbolic spaces were introduced by Behrstock--Hagen--Sisto in \cite{HHS_I, HHS_II}, building off of work of Masur--Minsky \cite{MM99,MM00} on the mapping class group, among others \cite{Brock, Rafi:combo, Rafi:hypteich, Dur:augmented, EMR:rank}.

For our purposes, we will introduce the basic foundational concepts that we need.  It turns out that the cubical model construction is a foundational-level construction, as one can use it to prove key facts like the Distance Formula \ref{thm:DF}; see Bowditch \cite{Bowditch_hulls} and \cite{Dur_infcube}.

The following is a complete list (at the time of writing) of the spaces known to be hierarchically hyperbolic.  With the exception of some cubical examples, all of them are colorable (Definition \ref{defn:colorable}), and so our main results Theorem \ref{thmi:main} and Corollary \ref{cori:semihyp} apply:

\begin{itemize}
\item Hyperbolic spaces.
\item Mapping class groups \cite{MM99, MM00, Aougab_hyp, HHS_I} and Teichm\"uller spaces, with both the Teichm\"uller \cite{Rafi:combo, Dur:augmented} and Weil-Petersson metrics \cite{Brock}.
\item Many \cite{HS:cubical}, but not all \cite{shepherd2025cubulation}, CAT(0) cubical groups,  including all right-angled Artin and Coxeter groups \cite{HHS_I}.
\item Fundamental groups of closed 3-manifolds without Nil or Sol summands \cite{HHS_II, hagen2022equivariant}.
\item Quotients of mapping class groups by large powers of Dehn twists \cite{behrstock2024combinatorial} and pseudo-Anosovs \cite{HHS:asdim}.
\item Surface group extensions of Veech groups \cite{DDLS1, DDLS2, bongiovanni2024extensions} and multicurve stabilizers \cite{russell2021extensions} of mapping class groups.
\item  The genus-2 handlebody group \cite{chesser2022stable}.
\item Artin groups of extra-large type \cite{MMS:Artin}.
\item Many free-by-cyclic groups \cite{bongiovanni2025characterizing}.
\item Admissible curve graphs \cite{calderon2024hierarchical}.
\item The class of HHSes is closed under products and appropriate combinations \cite{HHS_II}, including graph products \cite{BR:graph_prods} and graph braid groups  \cite{berlyne2021hierarchical}.
\end{itemize}

We note that there are non-colorable cubical HHSes \cite{Hagen:non-col} and non-HHS cubical groups \cite{shepherd2025cubulation}.  CAT(0) cube complexes are obviously locally modeled by cube complexes and, moreover, are bicombable \cite{NibloReeves}, so they are not of interest in this paper in any event.

\subsection{Basic setup of an HHS}\label{subsec:basic HHS}

Roughly, a \emph{hierarchically hyperbolic} structure is a pair $(\calX, \mathfrak S)$ where $\calX$ is a uniformly quasigeodesic metric space and $\mathfrak S$ is an index set for a family of uniformly hyperbolic spaces $\{\calC(U)\}_{U \in \mathfrak S}$.  The ambient space $\calX$ admits a family of uniformly coarsely Lipschitz projections $\pi_U:\calX \to \calC(U)$, and those projections, along with certain relations on the family $\mathfrak S$, largely control the coarse geometry of $\calX$.

The following extracts some of the properties we need from the axiomatic setup \cite{HHS_II}, see \cite{Sisto_HHS} for a discussion.

\begin{enumerate}

\item For each $U \in \mathfrak S$ there exist an $\ES$-hyperbolic space $\calC(U
)$ and a $\ES$-coarsely surjective $(\ES,\ES)$-coarsely Lipschitz map $\pi_U: \calX \rightarrow \calC(U),$ with $\ES$ independent of $U$.

    \item The set $\mathfrak S$ has the following relations:
    \begin{itemize}
        \item It is equipped with a partial order $\nest$ called \emph{nesting}. If $\mathfrak S \neq \emptyset,$ the set $\mathfrak S$ contains a unique $\nest$-maximal element $S.$
        
        \item It has a symmetric, anti-reflexive relation $\perp$ called \emph{orthogonality}. If $U \nest V$, then $U,V$ are not related via $\perp$. Furthermore, domains have \emph{orthogonal containers}: if $U \in \mathfrak S$ and there is a domain orthogonal to $U$ then there is a domain $W$ such that $V \nest W$ whenever $V \perp U.$
        
        \item For any distinct $ U, V \in \mathfrak S$ unrelated via $\perp, \nest$, then we say $U,V$ are \emph{transverse} and write $U  \pitchfork V.$ 
     
        \end{itemize}
          
        \item There exists an integer called the \emph{complexity} of $\calX$ and denoted $\xi(\mathfrak S)$ such that whenever $U_1,U_2,\cdots U_n$ is a collection of pairwise non-transverse domains, then $n \leq \xi(\mathfrak S).$
        
        \item If $U  \sqsubset V$ or $U  \pitchfork V,$ then there exists a set $\rho^U_V$ of diameter at most $\ES$ in $\calC(V)$.
        
        \item If $U  \sqsubset V$, there is a map $\rho^V_U:\calC(V) \rightarrow \calC(U)$, called the \emph{relative projection} of $U$ to $V$ which satisfies the following: For $x \in \calX$ with $d_{ V}(\pi_V(x),\rho^U_V)>\ES,$ we have $\pi_U(x) \underset{\ES}{\asymp} \rho^V_U(\pi_V(x)).$

\end{enumerate}

For two points $x,y \in \calX,$ it is standard to use $d_U(x,y)$ to denote $d_{\calC(U)}(\pi_U(x), \pi_U(y)),$ and similarly for subsets of $\calX.$ For a subset $A \subseteq \calX,$ we will also use $\diam_U(A)$ to denote the diameter of the set $\pi_U(A) \subset \calC(U)$.

\subsection{Basic HHS facts}

We now collect some useful facts about HHSes from various places. \medskip

The following definition from \cite{Beh:thesis, BKMM:qirigid, HHS_II}, will be important for us.  The idea here is that we want to combine all of the projections $\pi_U:\calX \to \calC(U)$ for each $U \in \mathfrak S$ into a global projection $\Pi:\calX \to \prod_U \calC(U)$.  The following definition uses the relations on $\mathfrak S$ to constrain tuples in $\prod_U \calC(U)$.

\begin{definition}[Consistent tuple]\label{defn:consistency}
Let $\kappa\geq 0$ and let $(b_U)\in\prod_{U\in\mathfrak S}2^{\calC(U)}$ be a tuple such that for each $U\in\mathfrak S$, 
the $U$--coordinate  $b_U$ has diameter $\leq\kappa$.  Then $(b_U)$ is \textbf{\em $\kappa$--consistent} if for all 
$V,W\in \mathfrak S$, we have $$\min\{\dist_V(b_V,\rho^W_V),\dist_W(b_W,\rho^V_W)\}\leq\kappa$$ whenever $V\pitchfork W$ and 
$$\min\{\dist_W(x,\rho^V_W),\diam_V(b_V\cup\rho^W_V)\}\leq\kappa$$ whenever $V\nest W$.
\begin{itemize}
\item We refer to the above inequalities as the \emph{consistency inequalities}.
\end{itemize}
\end{definition}

\medskip

The following is a key fact for us.  It says that the image of $\Pi:\calX \to \prod_U \calC(U)$ consists of uniformly consistent tuples:

\begin{lemma}\label{lem:base consistency}
There exists $\theta = \theta(\mathfrak S)>0$ so that if $x \in \calX$, then the tuple $(\pi_U(x)) \in \prod_{U \in \mathfrak S} \calC(U)$ is $\theta$-consistent.
\end{lemma}

On the other hand, we would like to know whether all consistent tuples in $\prod_U \calC(U)$ essentially come from $\calX$.  That is the content of the following Realization Theorem.  In particular, says that any tuple satisfying Definition \ref{defn:consistency} can be coarsely realized via a point in $\calX$.  It is \cite[Theorem 4.3]{BKMM:qirigid} for mapping class groups, \cite[Theorem 3.2]{EMR:rank} for Teichm\"uller space with the Teichm\"uller metric, and \cite[Theorem 3.1]{HHS_II} for the general HHS case:

\begin{theorem}[Realization]\label{thm:realization}
For each $\kappa \geq 1$ there exists $\theta_e, \theta_u \geq 0$ such that the following holds.  Let $\mathbf{b} = (b_U)_{U \in \mathfrak S} \in \prod_{U \in \mathfrak S} 2^{\calC(U)}$ be $\kappa$-consistent.

Then there exists $x \in X$ so that $d_U(b_U, \pi_W(x)) \leq \theta_e$ for all $U \in \mathfrak S$.  Moreover, $x$ is \emph{coarsely unique} in the sense that 
$$\diam_{\calX} \{x \in \calX| d_W(b_W, \pi_W(x)) \leq \theta_e \textrm{ for all } U \in \mathfrak S\} < \theta_u.$$
\end{theorem}

\begin{remark}
Notably, the proof of Theorem \ref{thm:realization} in \cite{HHS_II} does not involve the hyperbolicity of the associated spaces $\{\calC(U)\}_{U \in \mathfrak S}$, and actually follows from a weak form which is axiomatized into the HHS setup.  In this sense, realization is a purely hierarchical fact which plays a similar role to the existence of a coarse median in Bowditch's analysis \cite{Bowditch_hulls}.

\end{remark}

Perhaps the main technique in HHSes is taking an object in the ambient space $\calX$ and projecting it to all of the hyperbolic spaces, doing hyperbolic geometry, and then reassembling the pieces via hierarchical tools like the Realization Theorem \ref{thm:realization}.  Toward that end, and recognizing that HHSes are fundamentally coarse spaces, we will usually want to restrict ourselves to the \emph{relevant} collection of hyperbolic spaces to which an object has a sufficiently large projection.

\begin{definition}[Relevant domains] \label{defn:relevant}
Given a subset $A \subset \calX$ and $K > 0$, we say that $U \in \mathfrak S$ is $K$-\emph{relevant} for $A$ if $\diam_U(A) \geq K$.  We denote the set of $K$-relevant domains for $A$ by $\Rel_K(A)$.
\end{definition}

The following Passing-Up lemma is a foundational fact about HHSes; in fact, it can be taken to be an axiom \cite[Proposition 4.15]{Dur_infcube}.  Roughly speaking, the passing-up lemma says that if $x,y \in \calX$ have a large number of lower complexity relevant domains, then there must be some higher complexity relevant domain containing at least one of them, thereby allowing us to ``pass up'' the $\nest$-lattice.

\begin{lemma}[Passing-up]\label{lem:passing-up}
Let $\calX$ be an HHS with constant $\ES$.  For every $D>0$ there is an integer $PU=PU(D)>0$ such that if $V \in \mathfrak S$ and $x,y \in \calX$ satisfy $d_{U_i}(x,y)>\ES$ for a collection of domains $\{U_i\}_{i=1}^{PU}$ with each $U_i \in \mathfrak S_V$, then there exists $W \in \mathfrak S_V$ and $i$ with $U_i \sqsubsetneq W$ such that $d_W(x,y)>D.$
\end{lemma}

As an example of how this lemma is foundational, we observe the following corollary.  The original proof is embedded in the proof of Theorem \ref{thm:realization} in \cite[Theorem 3.1]{HHS_II}, but not explicitly stated.  We give a proof for completeness:

\begin{corollary}\label{cor:rel sets are finite}
For any $x,y \in \calX$, the set $\Rel_{\ES}(x,y)$ is finite.
\end{corollary}

\begin{proof}
Let $C$ be the coarse Lipschitz constant for all of the projections $\pi_U$ for $U \in \mathfrak S$, and set $D > C d_{\calX}(x,y) + C$.  If $\# \Rel_{\ES}(x,y) > P(D)$, then the Passing-Up Lemma \ref{lem:passing-up} provides a domain $W \in \mathfrak S$ with $d_W(x,y)>D$, which is impossible.
\end{proof}

Another fundamental fact about HHSes is the following distance formula, which originated in work of Masur-Minsky \cite{MM00}, and found various iterations elsewhere \cite{Brock, Rafi:combo, Dur:augmented, kim2014geometry} before being generalized to all HHSes in \cite{HHS_II}.  See \cite[Theorem 4.5]{HHS_II} for a proof from the now-standard HHS setup, and also \cite[Theorem 1.2]{Bowditch_hulls} and \cite[Corollary C]{Dur_infcube} from proofs using cubical model techniques, which we also recover here below in Corollary \ref{cor:DF lower bound}.

\begin{theorem}[Distance formula]\label{thm:DF}
There exists $K_0 = K_0(\mathfrak S)>0$ so that if $K>K_0$ and $x,y \in \calX$, then
$$d_{\calX}(x,y) \asymp \sum_{U \in \mathfrak S} [d_U(x,y)]_K$$
where the coarse equality $\asymp$ depends only on $\mathfrak S$ and $K$.
\end{theorem}

Philosophically (for our purposes), the distance formula says that if we want to understand how to move between two points in an HHS, it should be enough to focus on the set of relevant domains between them. 

The next definition is of a special kind of path in an HHS.  The proof of their existence is always tied up in the proof of the more difficult bound of the distance formula, as is the case in this paper in Subsection \ref{subsec:HP and DF}.  Otherwise, we will not need this notion.

\begin{definition}\label{defn:hp}
A path $\gamma:I \to \calX$ in an HHS $(\calX, \mathfrak S)$ is a $L$-\emph{hierarchy path} if its projection $\pi_U(\gamma) \subset \calC(U)$ is an unparameterized $(L,L)$-quasi-geodesic for each $U \in \mathfrak S$.
\end{definition}

We end this subsection with a basic observation, which generalizes the observation that if two subsurfaces $U,V$ are not transverse, then their boundary curves must be disjoint up to isotopy.

\begin{lemma}\label{lem:rho control}
Suppose that $U,V,W \in \Rel_{K_0}(a,b)$ for $a,b \in \calX$.
\begin{enumerate}
\item If $U, V \nest W$ and $U,V$ are not transverse, then $d_W(\rho^U_W, \rho^V_W) < \ES$.
\item If $U \pitchfork V$ and $W \nest V$ with $W,U$ not orthogonal, then $d_U(\rho^V_W, \rho^W_U)<\ES$.
\end{enumerate}
\end{lemma}

Finally, we end with a key HHS axiom, originally proven for mapping class groups in \cite[Theorem 3.1]{MM_II}:

\begin{axiom}[Bounded geodesic image]\label{ax:BGIA}  If $U  \sqsubset V$ and $\gamma \in \calC(V)$ is a geodesic with $d_{V}(\rho^U_V,\gamma)>\ES$, then diam$(\rho^V_U(\gamma))\leq \ES.$ Furthermore, for any $x,y \in \calX$ and any geodesic $\gamma$ connecting $\pi_V(x),\pi_V(y)$, if $d_{V}(\rho^U_V,\gamma)>\ES$, then $d(\pi_U(x), \pi_U(y)) \leq \ES.$
\end{axiom} 

\subsection{Hierarchical hulls and related basic facts}

In our HHS setup, the basic object we will want to consider is the following hierarchical notion of a coarse convex hull.  These \emph{hierarchical hulls} are built out of consistent tuples of points whose coordinates are contained in the hyperbolic hulls in each domain.  More specifically, given a pair of points $x,y \in \calX$, for each, $U \in \mathfrak S$, we set $H_U = \hull_U(\pi_U(\{x,y\}))$, where $\hull_U$ denotes the hyperbolic hull in $\calC(U)$.

\begin{definition}[Hierarchical hull]\label{defn:hier hull}
Given a constant $\theta>0$, we define the \emph{$\theta$-hull} of $x,y$ to be all $z \in \calX$ so that $d_U(z, H_U)<\theta$ for all $U \in \mathfrak S$.
\begin{itemize}
\item We denote the $\theta$-hull of $x,y$ by $H_{\theta} = \hull_{\theta}(x,y)$.
\end{itemize}
\end{definition}

One of the key facts about hierarchical hulls is that they are \emph{hierarchically quasiconvex} \cite[Lemma 6.2]{HHS_II}, see \cite{russell2023convexity} for more on convexity in HHSes.  We do not need this definition for our purposes, but we do need the following consequence:

\begin{lemma}\label{lem:gate retract}
For any $\theta>0$, there exists is a $(J,J)$-coarse Lipschitz retract $\gate_{H_{\theta}}:\calX \to H_{\theta}$, where $J = J(\theta, \mathfrak S)>0$ .
\end{lemma}

In the above lemma, the map $\gate_{H_{\theta}}: \calX \to H_{\theta}$ is a hierarchical analogue of a cubical gate map.  When $z \in \calX$, the point $\gate_{H_{\theta}}(z) \in H_{\theta}$ is defined by coordinate-wise closest-point projection of $\pi_U(z)$ to $H_U$ and then applying the Realization Theorem \ref{thm:realization}; see \cite[Proposition 6.3]{HHS_II} and the surrounding text for details.

\subsection{Consequences of the Passing-up Lemma \ref{lem:passing-up}}

The following statements are useful for the sort of distance formula related hierarchical accounting that we need to do the proof of the cubical model theorem.  The first says that any given relevant domain for a pair $x,y \in \calX$ can only nest into boundedly-many other relevant domains for $x,y$.  It first appeared in \cite[Lemma 2.15]{DMS_bary}; see \cite[Lemma 11.13]{CRHK} for a different proof and \cite[Lemma 4.8]{Dur_infcube} for a version that allows $x,y$ to be at infinity.

\begin{lemma}[Covering]\label{lem:covering}
For any HHS $(\calX, \mathfrak S)$, there exists $N = N(\mathfrak S)>0$ so that the following holds.  For any $x,y \in \calX$ and all $U \in \Rel_{50\ES}(x,y)$, there are at most $N$ domains $V \in \Rel_{50\ES}(x,y)$ so that $U \nest V$.
\end{lemma}

Another consequence of Lemma \ref{lem:passing-up} is the following distance formula type statement, which is a less general version of \cite[Proposition 4.14]{Dur_infcube}.  Roughly, it says that the number of medium-sized relevant domains for a pair $x,y$ can be accounted for (i.e., in terms of the domain projections) by large-sized domains higher up the hierarchy in a precise fashion.  

\begin{proposition}[Bounding large containers] \label{prop:bounding containers}
Let $K_1 >50\ES$, $K_2 \geq K_1+47\ES$, and $x,y \in \calX$ with $\calU_1 = \Rel_{K_1}(x,y), \calU_2 = \Rel_{K_2}(x,y),$ and $\calV = \calU_1 - \calU_2$.  Then the following hold:
\begin{enumerate}
    \item For all but $PU(K_2)-1$ domains $V \in \calV$, there exists $W \in \calU_2$ with $V \nest W$.
    \item For any $W \in \calU_2$, if $\calW = \{V \in \calV| V \nest W\}$ then we have $\#\calW \prec d_W(x,y)$ with constants depending only on $\mathfrak S$ and $K_2$.
\end{enumerate}
\end{proposition}

The above proposition is useful for proving the following slight generalization of the Distance Formula \ref{thm:DF}.  It's useful for circumstances, such as in Section \ref{subsec:Q qie}, where we want to make distance formula-type arguments where the summands are only coarsely equal to the standard distance formula terms.  Probably the key example of this is when we want to replace $d_U(x,y)$ with $d_{T_U}(x,y)$, where each $T_U$ is a uniform quasi-geodesic in $\calC(U)$ between $\pi_U(x)$ and $\pi_U(y)$.  This originally appeared in \cite[Lemma 9.6]{Min_ELC} for mapping class groups, and see \cite[Theorem 2.9]{BHMS:kill_twists} (which depends on the distance formula) and \cite[Corollary 4.17]{Dur_infcube} (which does not) for proofs in the HHS setting.

\begin{corollary}\label{cor:DF replace}
Let $x,y \in \calX$.  Suppose that $\sigma:\mathbb R_+ \to \mathbb R_+$ is a function satisfying 
$$\sigma(d_U(x,y)) \asymp d_U(x,y)$$
for each $U \in \mathfrak S$, where the constants in $\asymp$ are independent of $U$.  Then for any $K_1>50\ES$, we have
$$\sum_{U \in \calU} [d_U(x,y)]_{K_1} \asymp \sum_{U \in \calU} [\sigma(d_U(x,y))]_{K_1}$$
where the constants in $\asymp$ depend only on $\mathfrak S$, $K_1$, and $\sigma$.
\end{corollary}

\section{From hierarchical hulls to collapsed interval systems} \label{sec:interval systems}

In the next three sections, we will give an overview of the cubical model construction for the hierarchical hull of a pair of points in an HHS from \cite{Dur_infcube}.  The original construction is due to Behrstock-Hagen-Sisto \cite{HHS_quasi}, and see also Bowditch \cite{Bowditch_hulls} for a version in the more general setting of coarse median spaces satisfying HHS-like axioms.

The main goal of this section is to explain how to take a pair of points $a,b \in \calX$ in an HHS $(\calX, \mathfrak S)$ and produce a family of intervals which encode the relevant hierarchical data in a fairly efficient manner.  Roughly, the result of this process is a collection of intervals $\{\hT_U\}_{U \in \calU}$ equipped with a family of relative projections satisfying certain HHS-like properties, where the index set $\calU = \Rel_K(a,b)$ is the set of $K$-relevant domains for $a,b$ (Definition \ref{defn:relevant}).  For any such collection of intervals, one can define a notion of $0$-consistency (an exact form of Definition \ref{defn:consistency}), which defines a special $0$-consistent subset $\calQ$ of the product $\prod_{U \in \calU} \hT_U$.  In Section \ref{sec:hPsi qi} below we will explain why this set $\calQ$ is quasi-isometric to the hierarchical hull of $a,b$ in $\calX$.  Finally in Section \ref{sec:cubulation}, we will axiomatize the properties any such family of intervals satisfies, explain how it admits a natural wallspace structure, and show that the associated $0$-consistent set $\calQ$ is isometric to the cubical dual of this wallspace.

\subsection{Basic setup of the section}

For the rest of this section, we fix a constant $K\geq K_0$, where $K_0 = K_0(\mathfrak S)>0$ is the constant from the Distance Formula \ref{thm:DF}.  In this section, we will be able to allow $K$ to be as large as necessary.

Fix also $a, b\in \calX$.  Let $\calU = \Rel_K(a,b) = \{U \in \mathfrak S| d_U(a,b)>K\}$ denote their $K$-relevant set.  To each $U \in \calU$, let $\phi_U:T_U\to \calC(U)$ denote a $(C,C)$-quasi-isometric embedding of an interval between $a_U$ and $b_U$, so that $\phi_U(a_U) = \pi_U(a)$ and $\phi_U(b_U) = \pi_U(b)$.  Here we can assume that $C=C(\mathfrak S)>0$ depends only on the ambient HHS.  Note that we are only taking a quasi-isometric embedding, instead of a geodesic, because that is the level of generality that we will require later.

We also set some terminology for $a_U,b_U$, since they will play a distinguished role below. 

\begin{definition}[Marked points on intervals] \label{defn:marked}
For each $U\in \calU$ the points $a_U, b_U \in T_U$ are called \emph{marked points}.
\end{definition}

We note that in the basic construction explained in this section, marked points will always been the endpoints of the intervals $T_U$.  This will also be true of our collapsed stable trees defined in Subsection \ref{subsec:collapsed stable intervals}, though not necessarily of the stable intervals themselves.

The discussion in this section closely follows \cite[Sections 6 and 7]{Dur_infcube}. Given the technical nature of this part, we will mostly state definitions and important properties, providing sketches where convenient but leaving the details for the reader to verify.

\subsection{Reduced interval systems} \label{subsec:reduced interval systems}

Our first goal is to push the hierarchical data associated to $a,b$ onto the family of intervals $\{T_U\}_{U \in \calU}$.  This involves producing a notion of relative projection between the intervals and verifying some HHS-like properties for them.

These \emph{interval projections} are defined simply by closest-point projecting the relative projections of $(\calX,\mathfrak S)$ to the images of $\phi_U(T_U) \subset \calC(U)$ then pulling back to $T_U$ via $\phi^{-1}_U$.  We write $\delta^V_U$, rather than $\rho^V_U$, for the relative interval projections among the $T_U$.  These interval projections are interval versions of the ``tree projections'' described in \cite[Section 6]{Dur_infcube}.

The following lemma, which summarizes the basic properties that we need, says that the $\delta^V_U$ behave in essentially the same way as the $\rho^V_U$.  Its proof is an exercise in the definitions and the basic properties of an HHS; see \cite[Lemma 6.13]{Dur_infcube}.

\begin{lemma}\label{lem:interval control}
For any $C>0$ there exists $R_C = R_C(C,\mathfrak S)>0$ so that the following holds:

\begin{enumerate}
    \item If $U \pitchfork V \in \calU$, then both $\delta^V_U \subset T_U$ and $\delta^U_V \subset T_V$ are sets of diameter at most $R_C$ which are $R_C$-close to a marked point of $T_U,T_V$, respectively.  Moreover, $d_U(\phi_U(\delta^V_U), \rho^V_U)<R_C$, and similarly for $\delta^U_V$.
    \item If $U \nest V \in \calU$, then $\delta^U_V \subset T_V$ has diameter bounded by $R_C$.  Moreover, $d_V(\phi_V(\delta^U_V), \rho^U_V)<R_C$.
     \item If $U, V \nest W \in \calU$ with $U,V$ not transverse, then $\diam_{T_W}(\delta^U_W \cup \delta^V_W) < R_C$.
    \item If $U \pitchfork V$ and $W \nest V$ and $W,U$ not orthogonal, then $\diam_{T_U}(\delta^V_U \cup \delta^W_U)<R_C$.
    \item(Bounded geodesic image) If $U \nest V \in \calU$, then any component $C \subset T_V - \calN_{R_C}(\delta^U_V)$ satisfies $\diam_{T_U}(\delta^V_U(C))<R_C$.
  
\end{enumerate}

\end{lemma}

It will be useful for us to abstract the properties obtained for the intervals $T_U$ in the last subsection.  We are doing this so that we can later plug in the intervals produced by the Stable Interval Theorem \ref{thm:stable intervals} below directly into this setup.

\begin{definition}[Interval system]\label{defn:interval axioms}
Given $R>0$, an $R$-\emph{interval system} for $a,b \in \calX$ and $\calU = \Rel_K(a,b)$ is the following set of data:
\begin{enumerate}
\item For each $U \in \calU$, we have an interval $T_U$ and a $(R,R)$-quasi-isometric embedding $\phi_U:T_U \to \calC(U)$ so that $d^{Haus}_U(\phi_U(T_U), \hull_U(a,b))<R$.
\item For each $U \in\calU$, there exists a marked point $a_U \in T_U$ with $d_{\calC(U)}(\phi_U(a_U),\pi_U(a))<R$, and similarly for $b$.
\item There are $\delta^V_U$ sets which satisfy the conclusions of Lemma \ref{lem:interval control}.

\end{enumerate}
\end{definition}

\begin{remark}
When stabilizing the construction of the cubical models in Section \ref{sec:stable cubes}, we will want to use the stable intervals provided by Theorem \ref{thm:stable intervals} below.  For these, we will induce an interval system structure by a slightly more careful construction than just working directly with Gromov trees, e.g. item (2) in Definition \ref{defn:interval axioms} is immediate in that setting.  But the upshot is the same: the rest of the material in this section also holds for that setup.
\end{remark}

\subsection{Interval consistency}

The goal of this subsection is to define a map $\Psi:\hull_{\calX}(a,b)\to \prod_{U \in \calU} T_U$ and show that its image satisfies the appropriate analogue of Definition \ref{defn:consistency} in this setting.  Toward that end, for any $\alpha>0$, we set $\calZ_{\alpha} \subset \prod_{U \in \calU} T_U$ to be the set of $\alpha$-consistent tuples.

We can now define our map $\Psi:H \to \prod_{U \in \calU} T_U$ for $H = \hull_{\calX}(a,b)$.  Let $x \in H$ and for each $U \in \calU$, let $x_U$ denote a choice of a point in $\psi_U(x) = \phi_{U}^{-1}(p_U(\pi_U(x))) \subset T_U$, where $p_U:\calC(U)\to \phi_U(T_U)$ is closest point projection.  For each $x \in H$, this defines a tuple $(\psi_U(x)) = (x_U) \in \prod_{U \in \calU}T_U$.  That is, the coordinate-wise maps $\psi_U:H \to T_U$ combine to a global map $\Psi:H \to \prod_U T_U$.

\begin{lemma}\label{lem:interval consistency}
For any $R>0$, there exists $\alpha_0 = \alpha_0(\mathfrak S,R)>0$ so that if $\{T_U\}_{U \in \calU}$ is an $R$-interval system for $a,b \in \calX$ and $x \in \hull_{\calX}(a,b)$, then $\Psi(x) = (x_U)$ is $\alpha_0$-consistent.
\end{lemma}

\begin{proof}

Let $x \in \hull_{\calX}(a,b)$.  The corresponding tuple $(\pi_U(x_U))$ is $O(\theta, \mathfrak S)$-consistent by Lemma \ref{lem:base consistency}.  Moreover, by definition of $H = \hull_{\calX}(a,b)$ (Definition \ref{defn:hier hull}), we have $d_U(\pi_U(x_U), H_U) < \theta$, where $H_U = \hull_{\calC(U)}(a,b)$ and $\theta = \theta(\mathfrak S)>0$.  Hence $d_U(\pi_U(x_U), \phi_U(T_U))<\theta'$ for some $\theta' = \theta'(\mathfrak S, \theta)>0$.  At this point, a straight-forward argument using the properties in Lemma \ref{lem:interval control} gives the required consistency constant $\alpha_0 = \alpha_0(\mathfrak S, R)$.  We leave the details to the reader.
\end{proof}

\subsection{Interval thickenings}\label{subsec:thickenings}

Our next goal is to define a way of collapsing the relative projection data associated to intervals $T_U$, i.e. the $\delta^V_U$, in an interval system to remove coarseness from the setup.  The idea is that Definition \ref{defn:interval axioms} roughly encodes the projection data up to a controlled error.  In slightly more detail, observe that each interval $T_U$ contains the relative projections $\delta^V_U$ for $V \nest U$.  We would like to collect close-by relative projections into single intervals called \emph{cluster components}, so that we can later collapse them.  One complication is that the $\delta^V_U$ need not be subintervals of $T_U$, so we need a couple of constants---one to determine what it means for relative projections to be ``close'', and the other to guarantee that cluster components which are not close satisfy a minimum distance bound.

We will use the following slightly more abstract setup, as it will be useful later during our stabilization process:

 \begin{definition}[Thickenings]\label{defn:thickening}
Let $T$ be an interval with a decomposition $T = A \cup B$ into collections of segments.  Given $r_1,r_2>0$, the $(r_1,r_2)$-\emph{thickening} of $T$ \emph{along $B$} defines a decomposition $T = \bT_e \cup \bT_c$ into \emph{edge} $\bT_e$ and \emph{cluster} $\bT_c$ components as follows:
\begin{enumerate}
\item First, take the $r_1$-neighborhoods in $T$ of the components of $B$.  Call the collection of these neighborhood $\calB$, and their complement $T - \calB = \calA$.
\item Second, connect any two subintervals in $\calB$ which come within $r_2$ of each other in $T$ by the geodesic between them.  Call the resulting collection of subintervals $\bT_c$ and set $\bT_e = T- \bT_c$.
\end{enumerate}
\end{definition}

In our current context, we can take $B$ to be the collection of $\delta^V_U \subset T_U$ segments for $V \nest U$ along with the marked points $a_U,b_U$ on $T_U$ (Definition \ref{defn:marked}).  Thus a thickening of $T_U$ along this collection $B$ expands the $\delta^V_U$ segments and $a_U,b_U$ and combines nearby expansions.  For each $U \in \calU$, we denote the resulting decomposition of the thickened interval by $T_U = \bT^e_{U} \cup \bT^c_{U}$.  When we perform an $(r_1,r_2)$-thickening to each interval in an interval system, we refer to the resulting collection of thickened intervals as a $(r_1,r_2)$-\emph{thickening} of the interval system.

\begin{remark}
In our context, we will be able to control the constants $r_1,r_2>0$ in terms of the ambient HHS, and, importantly, we will be able to choose our largeness constant $K$ to be independent of them.
\end{remark}

\subsection{Collapsed intervals and $0$-consistency} \label{subsec:collapsing}

Let $\{T_U\}_{U \in \calU}$ be an $R$-interval system for $a,b \in \calX$.  Fix $r_i = r_i(R, \mathfrak S)>0$ for $i=1,2$ to be determined soon.  For each $U \in \calU$, let $T_U = \bT^e_U \cup \bT^c_U$ denote its $(r_1,r_2)$-thickening, as in Definition \ref{defn:thickening}.

Let $q_U:T_U \to \hT_U$ denote the map which collapses each component of $\bT^c_U$ to a point, each of which we call a \emph{collapsed point}.  The resulting object $\hT_U$ is an interval.  Note that by construction $a_U, b_U \in \bT^c_U$ for each $U \in \calU$.  Hence each of $a_U,b_U$ determines a \emph{marked point} $\ha_U,\hb_U \in \hT_U$, and these marked points $\ha_U,\hb_U$ will be the endpoints of the (collapsed) interval $\hT_U$.  Moreover, the decomposition $T_U = \bT^e_U \cup \bT^c_U$ descends to a decomposition $\hT_U = \hT^e_U \cup \hT^c_U$, where the components of $\hT^c_U$ are the collapsed points corresponding to the components of $\bT^c_U$, and the components of $\hT^e_U$ are isometric images of the components of $\bT^e_U$.

Combining these $q_U$ maps with our coordinate-wise projections $\psi_U:H \to T_U$, we get coordinate-wise maps $\hpsi_U:H \to \hT_U$ to the collapsed intervals, and hence combining these gives a global map $\hPsi:H \to \prod_U \hT_U$.

We can define \emph{collapsed relative projections} $\hd^V_U$ in the obvious way, by composing with the quotient maps $q_U$.  Similarly, consistency in this setting is the appropriate analogue of Definition \ref{defn:consistency}, except now all coarseness has been removed:

\begin{definition} \label{defn:Q defined}
A tuple $(\hx_U) \in \prod_{U \in \calU} \hT_U$ is $0$-\emph{consistent} if the following hold:
\begin{enumerate}
\item For any $U \pitchfork V \in \calU$, either $\hx_U = \hd^V_U$ or $\hx_V = \hd^U_V$.
\item For any $U \nest V \in \calU$, either $\hx_V = \hd^U_V$ or $\hx_U \in \hd^V_U(\hx_V) \subset \hT_U$.
\end{enumerate}
\begin{itemize}
\item We let $\calQ$ denote set of $0$-consistent tuples in $\prod_{U \in \calU} \hT_U$.
\end{itemize}
\end{definition}

We can prove that $\hPsi$ has image in $\calQ$:

\begin{proposition}\label{prop:hPsi defined}
For any $R>0$,  there exists $R_1 = R_1(R,\mathfrak S)>0$ so that for any $r_1 > R_1$ and $r_2>0$, if $\{\hT_U\}_{U \in \calU}$ are the collapsed intervals associated to an $(r_1,r_2)$-thickening of an $R$-interval system for $a,b \in \calX$ and $x \in \hull_{\calX}(a,b)$, then $\hPsi(x) = (\hx_U) \in \calQ$.
\end{proposition}

\begin{proof}
We need to prove $0$-consistency of $\hPsi(x)$ as in Definition \ref{defn:Q defined}.

First suppose that $U \pitchfork V \in \calU$.  Since $\Psi(x) = (x_V)$ is $\alpha_0$-consistent by Lemma \ref{lem:interval consistency} where $\alpha_0=\alpha_0(R,\theta,\mathfrak S)>0$, we have, without loss of generality, that $d_{T_U}(x_U, \delta^V_U)< \alpha_0$.  Hence $q_U(x_U) = \hd^V_U$ as long as $r_1 > \alpha_0$.  

Next suppose that $U \nest V \in \calU$.  If $d_{T_V}(x_V, \delta^U_V)< \alpha_0 < r_1$, then $q_V(x_V) = \hd^U_V$, and we are done.  Otherwise, we have $\diam_{T_U}(x_U \cup \delta^V_U(x_V))<\alpha_0$ while also $d_{T_V}(x_V, \delta^U_V)>\alpha_0$.  By taking $\alpha_0 > R$, we have that $\diam_{T_U}(\delta^V_U(x_V)\cup f_U)< R$ for some $f \in \{a,b\}$ by the BGI property in item (4) of Definition \ref{defn:interval axioms}.  In particular, setting $r_1>\alpha_0 + R$, then both $x_U$ and $\delta^V_U(x_V)$ are within $\alpha_0 + R< r_1$ of $f_U$, and hence $q_U(x_U) = q_V(\delta^V_U(x_V))$.  Finally, observe that $q^{-1}_V(q_V(x_V))$ is contained in the component of $T_V - \mathcal{N}_{R}(\delta^U_V)$ containing $x_V$, and hence $q_V(\delta^V_U(x_V)) = \hd^V_U(q_V(x_V))$.  This completes the proof.
\end{proof}

\subsection{Controlling the collapsed intervals}

In this subsection, we record the following technical lemma, which is in part a collapsed version of Lemma \ref{lem:interval control}, but also contains the conclusion of Proposition \ref{prop:hPsi defined}.  The statement and its proof are basically \cite[Lemma 7.13]{Dur_infcube}.

\begin{lemma}[Collapsed interval control]\label{lem:collapsed interval control}
Let $R>0$ and $r_1 > R_1$ for $R_1 = R_1(R,\theta, \mathfrak S)>0$ as in Proposition \ref{prop:hPsi defined}.  There exists $R_2 = R_2(R,\mathfrak S)>0$ so that for any $r_2>R_2$, the following hold for any $(r_1,r_2)$-thickening of an $R$-interval system for $a,b \in \calX$: 
\begin{enumerate}
    \item For any $x \in H = \hull_{\calX}(a,b)$, the tuple $\hPsi(x) \in \calQ$ is $0$-consistent.
    
    \item If $U \pitchfork V \in \calU$, then both $\hd^V_U \in \hT_U$ and $\hd^U_V \in \hT_V$ are in $\{\ha_U, \hb_U\}$.
    
    \item If $U \nest V \in \calU$, then $\hd^U_V$ coincides with a collapsed point.
    
    \item For any $U \in \calU$ and any $x,y \in H$, we have $d_U(x,y) \asymp d_{T_U}(x_u,y_U)\geq d_{\hT_U}(\hx_U,\hy_U)$, with the constants in $\asymp$ depending only on $\mathfrak S$ and $R$.

    \item If $x,y \in H$ and $U \in \calU$ and $\hpsi_U(x), \hpsi_U(y)$ are distinct collapsed points, then $d_U(x,y) > 50\ES$.
    
    \item If $U \in \calU$ is $\nest_{\calU}$-minimal, then $d_{\hT_U}(\ha_U,\hb_U) \succ K - 2r_1$, with the constants in $\succ$ depending only on $\mathfrak S$ and where $K = K(\mathfrak S)>0$ is chosen independent of the setup.
    
    \item (Bounded geodesic image) Suppose $V \nest U$, $C \subset \hT_U - \hd^V_U$ is a component, then $\hd^U_V(C)$ is a point.  In particular, if $f \in \{a,b\}$ with $\hf_U$ contained in $C$, then $\hd^U_V(C) = \hf_V.$

 \item If $\hx \in \calQ$, then $\{U \in \calU| \hx_U \notin \{\ha_U,\hb_U\}\}$ is pairwise non-transverse, and hence has cardinality bounded in terms of $\mathfrak S$.

    \item If $\hx \in \calQ$, then $E(\hx) = \{U \in \calU| \hx_U \in \hT^e_U\}$ is pairwise orthogonal.

\end{enumerate}
\end{lemma}

\begin{remark}
Item (5) is notably the first time that the constant $r_2$ from Definition \ref{defn:thickening} plays a substantial role.  This is important later in proving that the map $\hPsi:H \to \calQ$ between  $0$-consistent set $\calQ$ and the hierarchical hull $H = \hull_{\calX}(a,b)$ is a quasi-isometry.  Another notable statement is item (6), which is about $\nest_{\calU}$-minimal domains, namely domains $U$ for which no $V \in \calU$ has $V \nest U$.  Here is where our largeness constant $K$ plays a crucial role.  In particular, since nothing nests into a $\nest_{\calU}$-minimal domain $U$, the map $q_U: T_U \to \hT_U$ only collapses the $r_1$-neighborhood of each endpoint, and hence the diameter of $\hT_U$ is coarsely bounded below by $K$.  This is also important for confirming that $\hPsi:H \to \calQ$ is a quasi-isometry; see Section \ref{sec:hPsi qi} below.
\end{remark}

\begin{proof}
Item (1) is  Proposition \ref{prop:hPsi defined}.  Item (2)follows from item (2) of Definition \ref{defn:interval axioms} by choosing $r_1>2R$.  Similarly, item (3) follows from Definition \ref{defn:interval axioms} (namely, item (1) of Lemma \ref{lem:interval control}) by choosing $r_1 > 2R$.  

Item (4) follows directly from the construction and item (3) of Definition \ref{defn:interval axioms}.

Item (4) is an immediate consequence of the definition of the hull $H$ (Definition \ref{defn:hier hull}) and the construction in this subsection.  In particular, by assumption $\phi_U:T_U \to \calC(U)$ is an $(R,R)$-quasi-geodesic between $\pi_U(a), \pi_U(b)$, giving the coarse equality, while the inequality follows because $q_U:T_U\to \hT_U$ collapses subintervals and therefore is distance non-increasing.

For item (5), if $\hpsi_U(x),\hpsi_U(y)$ are in distinct clusters, then $d_{T_U}(x,y) > r_2$, and hence $d_U(x,y) \succ r_2$ by item (4).  Hence choosing $r_2$ large enough gives the desired lower bound.

For item (6): When $U \in \calU$ is $\nest_{\calU}$-minimal with respect to $\calU$, then $q_U:T_U \to \hT_U$ only collapses the $r_1$-neighborhoods of $a_U,b_U$ in $T_U$.  Since $d_U(a,b)>K$ by assumption, it follows that $d_{T_U}(a_U,b_U) \succ K$, while the map $q_U:T_U \to \hT_U$ only collapses the $r_1$-neighborhoods of $a_U,b_U$.  Thus provided that $K -2r_1 > r_2$ (which we can arrange), the claim follows.

For item (7): By item (1), if $\hf_U$ is a marked point, then $0$-consistency of $(\hf_V) = \hPsi(f)$ implies that $\hd^U_V(\hf_U) = \hf_V$ because $d_{T_U}(\delta^V_U(f_U),f_V)<\alpha_0 < r_1$, where we note that $\alpha_0 = \alpha_0(R, \mathfrak S)>0$ by Lemma \ref{lem:interval consistency}.  On the other hand, since $C \cap \hd^V_U = \emptyset$, we have $d_{T_U}(q^{-1}_U(C), \delta^V_U)>R$.  Hence item (4) of Definition \ref{defn:interval axioms} says that $\diam_{T_V}(\delta^U_V(C))<R$, while also $\delta^U_V(f_U) \subset \delta^V_U(C)$.  Hence $d_{T_V}(\delta^U_V(C), f_V)<\alpha_0 + R<r_1$, where we guarantee the last inequality by increasing $r_1$ as necessary.  It follows then that $\hd^U_V(C) = \hf_V$, as claimed.

For item (8): Suppose $U, V$ are transverse and $\hx_U \notin  \{\ha_U, \hb_U\}$.  Since $\hd^V_U \in \{\ha_U,\hb_U\}$ by item (2) and hence $\hx_U \neq \hd^V_U$, it follows from $0$-consistency of $\hx$ that $\hx_V = \hd^U_V$.  Thus $\hx_V \in \{\ha_V,\hb_V\}$ and so the set in the statement of item (8) is pairwise non-transverse, and hence has bounded complexity by item (2) in Subsection \ref{subsec:basic HHS}.

Finally, for item (9): Recall that $\hT^e_U = \hT_U - \hT^c_U$, where $\hT^c_U$ are the collapsed points.  Hence $E(\hx)$ is the set of domains where $\hx_U$ is not a collapsed point.  Supposing that $U,V$ are neither orthogonal nor (by item (8)) transverse.  When $V \nest U$ and $\hx_U = \hd^V_U$, we are done by items (1) and (3).  When $V \nest U$ and $\hx_U \neq \hd^V_U$, item (7) implies there is some $f \in \{a,b\}$ so that $\hd^U_V(\hx_U) = \hf_V$, while $0$-consistency says that $\hx_V = \hd^U_V(\hx_U)$, giving $\hx_V = \hf_V$.  Hence if $U$ and $V$ are not orthogonal, then one of $\hx_U$ or $\hx_V$ must be a collapsed point, proving this item.  This completes the proof.

\end{proof}

\section{$\hPsi:H \to \calQ$ is a quasi-isometry} \label{sec:hPsi qi}

In this section, we explain why the map $\hPsi:H \to \calQ$ (Proposition \ref{prop:hPsi defined}) between the hierarchical hull $H = \hull_{\calX}(a,b)$ of a pair of points $a,b \in \calX$ and the $0$-consistent set associated to any sufficiently wide $(r_1,r_2)$-thickening of an $R$-interval system for $a,b$ is a quasi-isometry.  There are two main parts to this:

\begin{enumerate}
\item Proving that $\hPsi:H \to \calQ$ is a quasi-isometric embedding, and
\item Proving that $\hPsi$ is coarsely surjective.  This is accomplished by defining a coarse inverse $\hO:\calQ \to H$, so that $\hPsi \circ \hO$ is coarsely the identity on $\calQ$.
\end{enumerate}

As we shall see, these two results have fairly different flavors---(1) is essentially purely hierarchical, while (2) has distinctly cubical elements.  Item (1) essentially says that one can encode the distance formula in $\calX$ as the $\ell^1$-distance in a product of intervals.  Its proof involves mostly distance formula-type arguments and we present it in full.  On the other hand, the proof of item (2), which is contained in \cite[Sections 9--11]{Dur_infcube}, is rather involved and many of the details of the argument are not crucial to our expository purposes in this article.  Thus, in this section, we will mostly explain the basics of the construction, including the points where our simpler setting simplifies things considerably.

The one exception here is the Tree Trimming Theorem \ref{thm:tree trimming}, which is a distance-formula like tool that we need both here and for our stability results in Section \ref{sec:stable cubes}.  We will discuss it in detail and present its proof in Appendix \ref{app:TT}.

For the rest of this section, fix $a,b \in \calX$ and let $H = \hull_{\calX}(a,b)$ to be their hierarchical hull.  Fix a largeness constant $K$, which we will be able to control relative to all other involved constants, and set $\calU = \Rel_K(a,b)$.  Finally, fix $R>0$ and an $R$-interval system $\{T_U\}_{U \in \calU}$ for $a,b$.  Let $\{\hT_U\}_{U \in \calU}$ denote the family of collapsed intervals corresponding to an $(r_1,r_2)$-thickening, where $r_i = r_i(R,\theta,\mathfrak S)>0$ are provided by Proposition \ref{prop:hPsi defined} and Lemma \ref{lem:collapsed interval control}.

\subsection{$\hPsi:H \to \calQ$ is a quasi-isometric embedding}\label{subsec:Q qie}

The content of this subsection is roughly contained in \cite[Section 8]{Dur_infcube}.  

Before we state the result, we set some additional notation: For numbers $A,B$, the notation $[A]_B$ means $0$ when $A<B$ and $A$ when $A\geq B$.  For $x \in H$ and $U \in \calU$, we will write
\begin{itemize}
\item $x = \pi_U(x) \in \calC(U)$;
\item $x_U = \psi_U(x) \in T_U$;
\item $\hx_U = \hpsi_U(x) \in \hT_U$.
\end{itemize}

\begin{theorem}\label{thm:Q distance estimate}
There exists $K_1 = K_1(\mathfrak S)>0$ so that for any $K_2 \geq K_1$ and any $x, y \in H$, we have that
$$d_{\calX}(a,b) \asymp \sum_{U \in \mathfrak S} [d_U(x,y)]_{K_2} \asymp \sum_{U \in \calU} d_{\hT_U}(\hx_U,\hy_U)$$
where the constants in $\asymp$ only depend on $\mathfrak S$ and the choice of $K_2$.
\end{theorem}

We note that the left coarse equality of the above statement is simply the Distance Formula \ref{thm:DF}, hence the content is the right coarse equality.  In fact, one can obtain the (more difficult) lower-bound in the distance formula by combining the above theorem with the construction of hierarchy paths via cubical models, as we do below in Corollary \ref{cor:DF lower bound}.

Before beginning the discussion of the proof properly, we make a couple of simplifying observations, which are all basically distance formula-type manipulations.

First, note that if $x,y \in H$ and $V \in \Rel_K(x,y)$ but $V \notin \calU$, then $d_V(x,y) < K + \epsilon$ for some $\epsilon = \epsilon(\mathfrak S)>0$.  This is because $\pi_V(x), \pi_V(y) \in \calN_{\theta}(\hull_V(a,b))$.  Hence by taking $K_2 \geq K_1 > K + \epsilon$, we may increase the threshold to remove any extraneous domains.  That is:
$$\sum_{U \in \mathfrak S} [d_U(x,y)]_{K_2} = \sum_{U \in \calU} [d_U(x,y)]_{K_2}.$$

We next observe that for any $x,y \in H$, item (4) of Lemma \ref{lem:collapsed interval control} implies that $d_U(x,y) \asymp d_{T_U}(x_U, y_U)$, with constants depending only on $\mathfrak S$.  Thus we can use Corollary \ref{cor:DF replace} to replace the terms $d_U(x,y)$ with $d_{T_U}(x_U,y_U)$ in the sum, namely:

$$\sum_{U \in \calU} [d_U(x,y)]_{K_2} \asymp \sum_{U \in \calU} [d_{T_U}(x_U,y_U)]_{K_2}.$$

The two bounds in Theorem \ref{thm:Q distance estimate} are proven separately.  The upper bound only requires the basic details from the construction of the $\hT_U$ and some distance formula tools.  We have included the full proof of the proposition for clarity but also because it is used in proving the existence of hierarchy paths and the upper bound of the distance formula; moreover, the argument is similar to the proof of Corollary \ref{cor:DF replace}, which actually implies the proposition.

\begin{proposition}[Upper bound]\label{prop:upper bound}
There exists $K_1= K_1(\mathfrak S)>0$ so that for any $K_2 >K_1$ and any $x, y \in H$ we have
\begin{equation}\label{eqn:upper}
    \sum_{U \in \calU} [d_U(x,y)]_{K_2} \succ \sum_{U \in \calU} d_{\hT_U}(\hx_U,\hy_U)
\end{equation}
with the constants in $\succ$ depending only on $\mathfrak S$ and $K_2$.
\end{proposition}

\begin{proof}

As above, let $K_1 \geq K + \epsilon$ and set $K_2 \geq K_1 + 100\ES$.  Let $\calU_i = \Rel_{K_i}(x,y)$, so that $\calU_2 \subset \calU_1 \subset \calU$.  If $U \in \calU$ satisfies $\hx_U = \hy_U$, then the corresponding term disappears on the right-hand side.  Hence it suffices to account for the domains in $\calV = \{V \in \calU | d_V(x,y) < K_2 \textrm{ and } \hx_U \neq \hy_U\}$, which appear on the right-hand side (RHS) but not the left-hand side (LHS) of equation \eqref{eqn:upper}.

For this, item (9) of Lemma \ref{lem:collapsed interval control} says that $\hx_U$ is not contained in some cluster point of $\hT_U$ for only boundedly-many $U$, and the same for $\hy_U$.  (In fact, we can take $\hx_U = \ha_U$ and $\hy_U = \hb_U$, be this will not simplify things.) This means that, up to ignoring a bounded collection of domains $\calV' \subset \calV$, if $V \in \calV - \calV'$ then $\hx_V$ and $\hy_V$ are at distinct cluster points, and so $d_V(x,y) > 50\ES$ by item (5) of Lemma \ref{lem:collapsed interval control}, i.e., $\calV-\calV' \subset \Rel_{50\ES}(x,y)$, while also $d_{\hT_V}(\hx_V,\hy_V) \geq r_2 \geq 1$.

One finishes the proof with an application of Proposition \ref{prop:bounding containers} as follows.  Set $\calU_i = \Rel_{K_i}(x,y)$ and observe that $\calV - \calV' \subset \calU_1 - \calU_2$.  Proposition \ref{prop:bounding containers} says that up to ignoring a subset $\calV'' \subset \calV - \calV'$ with $\#\calV'' \leq P(K_2)$, for each $V \in \calV_0 = \calV - (\calV' \cup \calV'')$ there exists some domain $W \in \calU_2$ determining a term on the RHS with $V \nest W$.  Moreover, it says that the set of domains $\calW = \{V \in \calV| V\nest W\}$ for which each such $W$ accounts in this way satisfies $\#\calW \leq (1+A)\cdot d_W(x,y) + A$, for some $A = A(K_2, \mathfrak S)>0$.

Combining these observations leads to the computation
\begin{eqnarray*}
\sum_{U \in \calU_2} d_U(x,y) &=& \sum_{U \in \calU_2} d_U(x,y) + \sum_{V \in \calV - \calV'} d_V(x,y)+ \sum_{Z \in \calV'} d_Z(x,y)\\
&\leq& \sum_{U \in \calU_2} ((1+A)\cdot d_U(x,y)+A) + K_2 \cdot P(K_2)\\
&\leq& (1+2A)\cdot \sum_{U \in \calU_2} d_U(x,y) + K_2 \cdot P(K_2)
\end{eqnarray*}
where the last inequality follows from the fact that $d_U(x,y) > K_2$ for each $U \in \calU_2$.  This completes the proof. 

\end{proof}

\begin{figure}[h]
    \centering
    \includegraphics[width=.8\textwidth]{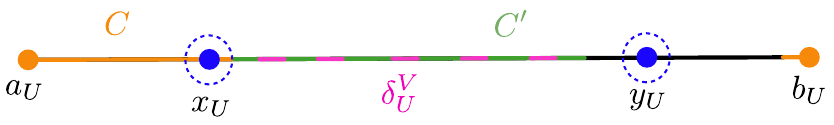}
    \caption{Proof of Proposition \ref{prop:lower bound}: The subcluster $C'$ isolates the portion of $C$ which intersects $[x_U,y_U]$ but avoids the $(\alpha_0+R)$-neighborhoods of $x_U$ and $y_U$.  The (pink) net of $\delta$-sets from $\nest_{\calU}$-minimal domains $V \nest U$ with $\delta^V_U \subset C'$ coarsely encodes the length of $C'$, with the BGIA property guaranteeing that $\hx_V = \ha_V$ and $\hy_V = \hb_V$ and hence $d_{\hT_V}(\hx_V,\hy_V) \geq K - B^2 - 2r_1$ for each such $V$.}
    \label{fig:subcluster}
\end{figure}

\begin{proposition}[Lower bound]\label{prop:lower bound}
There exists $K_1= K_1(\mathfrak S)>0$ so that for any $K_2 >K_1$ and any $x, y \in H$ we have
\begin{equation}\label{eqn:lower}
    \sum_{U \in \calU} [d_U(x,y)]_{K_2} \prec \sum_{U \in \calU} d_{\hT_U}(\hx_U,\hy_U)
\end{equation}
with the constants in $\prec$ depending only on $\mathfrak S$ and $K_2$.
\end{proposition}

\begin{proof}

Observe that since $d_U(x,y) \prec d_{T_U}(x_U,y_U)$ for each $U$, it suffices by Corollary \ref{cor:DF replace} to show that 
$$\sum_{U \in \calU} [d_{T_U}(x_U,y_U)]_{K_2} \prec \sum_{U \in \calU} d_{\hT_U}(\hx_U,\hy_U).$$

For this bound, there are possibly domains $U \in \calU$ where $d_{T_U}(x_U,y_U) > K_2$ but there is a discrepancy $d_{T_U}(x_U,y_U) > d_{\hT_U}(\hx_U,\hy_U)$.  Let $\calW$ denote the collection of such domains.

Since $\hT_U$ is obtained from $T_U$ by collapsing clusters, we need to analyze the clusters which separate $x_U$ from $y_U$, i.e. those clusters $C_1, \dots, C_n$ which intersect, in that order, the geodesic $[x_U,y_U]$ in $T_U$.  The proof is mainly a BGIA \ref{ax:BGIA} type argument along with an application of Proposition \ref{prop:bounding containers} and the Covering Lemma \ref{lem:covering}.

Continuing with the above notation, for each $i$, let $C'_i = C_i \cap [x_U,y_U]_{T_U} - \left(B_{\alpha_0 + R}(x_U) \cup B_{\alpha_0 + R}(y_U)\right)$, where $\alpha_0 = \alpha(R, \mathfrak S)>0$ is the interval consistency constant provided by Lemma \ref{lem:interval consistency}; see Figure \ref{fig:subcluster}.

Observe that for each $i$, if we set $\calV_i = \{V \in \calU|V \nest U \text{ and } \delta^V_U \subset C_i\}$ then $\{\delta^V_U| V \in \calV_i\}$ is $(r_2 + 2r_1)$-dense in $C_i$.  Hence we can find a subcollection $\calV_i'$ so that
\begin{itemize}
    \item Each $V \in \calV_i'$ is $\nest_{\calU}$-minimal;
    \item For each $V \in \calV'$, we have $\delta^V_U \subset C_i'$;
    \item $\{\delta^V_U| V \in \calV_i\}$ is a $(2r_2+4r_1, r_2)$-net on $C_i'$ (see Definition \ref{defn:net} below).
\end{itemize}

Note that the increase in the net constants in the third item is required to guarantee the second item.  In particular, we have $\#\calV_i \succ \diam_{T_U}(C_i') \geq \diam_{T_U}(C_i \cap [x_U,y_U])- 2(\alpha_0+R)$, where $\#\calV_i$ denotes the cardinality of $\calV_i$.

Let $V \in \calV_i'$.  Then, by construction, $d_{T_U}(\delta^V_U, x_U)> \alpha + R$ and $d_{T_U}(\delta^V_U, y_U)>\alpha_0 + R$.  Hence, up to switching the roles of $a$ and $b$, the BGI property (5) of Lemma \ref{lem:interval control} says that $d_{T_V}(\delta^U_V(a_V),\delta^U_V(x_V))<R$ and $d_{T_V}(\delta^U_V(b_V),\delta^U_V(y_V))<R$. 

Lemma \ref{lem:interval consistency} says that the tuples $(a_U), (b_U), (x_U), (y_U)$ are all $\alpha_0$-consistent.  Thus since $d_{T_U}(a_U, \delta^V_U)>R + \alpha_0$ and similarly for $b_U,x_U,y_U$, $\alpha_0$-consistency forces that $d_{T_V}(a_V, \delta^U_V(a_U)) < \alpha_0$, and similarly for $b_V,x_V,y_V$.  As such, we get that $d_{T_V}(x_V,a_V) < \alpha_0 + R$ and $d_{T_V}(y_V,b_V) < \alpha_0 + R$, and so that $d_{T_V}(x_V,y_V) \geq d_{T_V}(a_V,b_V) - 2\alpha_0 - 2R$, from which we can conclude that $\ha_V = \hx_V$ and $\hb_V = \hy_V$ by requiring $r_1 > \alpha_0 + R$.

On the other hand, since each $V \in \calV_i'$ is $\nest_{\calU}$-minimal and $d_V(a,b) > K$, item (4) of Lemma \ref{lem:collapsed interval control} says that $d_{\hT_V}(\hat{a}_V,\hat{b}_V) \succ K$.  Hence we obtain that $d_{\hT_V}(\hx_V,\hy_V)  = d_{\hT_V}(\ha_V,\hb_V) \succ K$.

Thus for each $U \in \calW$, we have
\begin{eqnarray*}
d_{T_U}(x_U,y_U) & = & d_{\hT_U}(\hx_U,\hy_U) + \sum_{i=1}^n \diam_{T_U}(C_i \cap [x_U,y_U]_{T_U})\\
&\prec& d_{\hT_U}(\hx_U,\hy_U) + \sum_{i=1}^n \#\calV_i'\\
&\prec& d_{\hT_U}(\hx_U,\hy_U) + \sum_{i=1}^n \left(\sum_{V \in \calV_i'} d_{\hT_V}(\hx_V,\hy_V)\right).
\end{eqnarray*}

Finally, the Covering Lemma \ref{lem:covering} says that any domain $V \in \calU$ can nest into only boundedly-many domains in $\calU$, with bounds determined by $\mathfrak S$.  Hence, each $\nest_{\calU}$-minimal domain $V \in \calU$ must account for at most boundedly-many $U \in \calW$ in the above way.  Thus, using an argument similar to the end of the proof of Proposition \ref{prop:upper bound}, we can account for the distance lost by collapsing clusters in each domain in $\calW$ with these large $\nest_{\calU}$-minimal domains by increasing the multiplicative constant for the estimate in equation \eqref{eqn:upper} by a bounded amount.  This completes the proof.
\end{proof}

We used the following definition in the above proof:

\begin{definition}[Net]\label{defn:net}
Given constants $a,A\geq 0$, a $(a,A)$-\emph{net} on a subspace $Y \subset X$ of a metric space is a collection $Z \subset Y$ of points so that:
\begin{enumerate}
\item $Z$ is $a$-coarsely dense in $Y$, i.e. for every $y \in Y$ there is some $z \in Z$ with $d(z,y)<a$, and
\item For any $z,z' \in Z$, we have $d_X(z,z') \geq A$.
\end{enumerate}
\end{definition}

In particular, an $(a,A)$-net $Z$ on $Y$ satisfies $\#Z \asymp \diam(Y)$ with the constants depending only on $a,A$.

\subsection{The map $\hO:\calQ \to H$ and honing clusters}\label{subsec:hO defined}

With Theorem \ref{thm:Q distance estimate} in hand, in order to prove that $\hPsi:H \to \calQ$ is a quasi-isometry, we need to explain why $\hPsi$ is coarsely surjective.  The proof of this fact is quite involved and beyond the scope of an expository paper.  As such, we will mostly present key ideas and not a full detailed proof because the proof is not greatly simplified versus the most general version, see \cite[Sections 9--11]{Dur_infcube}.

In this subsection, we will construct a map $\Omega:\calQ \to \prod_{U \in \calU} T_U$, and sketch why its image is $\beta$-consistent for some $\beta = \beta(\mathfrak S, R, r_1,r_2)>0$.   Upgrading this to a map $\hO:\calQ \to H$ involves standard HHS techniques.  In Subsection \ref{subsec:tree trimming}, we will prove our Tree Trimming Theorem \ref{thm:tree trimming}, which is a useful distance formula-like tool that will allow us modify our collapsed interval setup while coarsely preserving the geometry of $\calQ$.  This tool is a key ingredient in Subsection \ref{subsec:upgrade to HHS}, where we sketch why $\hO$ is a coarse inverse to $\hPsi$, namely that $d_{\calQ}(\hPsi \circ \hO(\hx), \hx)$ is bounded independently of $\hx \in \calQ$.

We can now set about defining the map $\Omega:\calQ \to \prod_{U \in \calU} T_U$, which is defined coordinate-wise.  To put things more precisely, we want to define $\Omega$ so that:

\begin{itemize}
    \item For each $\hx \in \calQ$, the tuple $\Omega(\hx)$ is $\omega$-consistent for some $\omega = \omega(\mathfrak S)>0$, and
    \item when $\hx  = (\hx_U)\in \calQ$ and $(x_U) = \Omega(\hx)$, then $q_U(x_U) = \hx_U$ for each $U \in \calU$.
\end{itemize}

Unsurprisingly, the cluster subintervals, which are collapsed under each $q_U:T_U \to \hT_U$, play a key role.  In particular, these cluster subintervals can have arbitrarily large diameter, and so if a coordinate $\hx_U\in \hT_U$ lies in a cluster point of $\hT_U$, then its preimage $q_U^{-1}(\hx_U)$ is that whole cluster.  Choosing a point within that cluster then requires utilizing the coordinates for $V \nest U$ with $\hd^V_U = \hx_U$.  We think of this process as \emph{honing} the cluster.

\smallskip

\textbf{\underline{Honing clusters}}: Let $\hx = (\hx_U) \in \calQ$.  When $U\in \calU$ has $\hx_U \in \hT^e_U$, then the definition is simple: $x_U = q_U^{-1}(\hx_U)$, which is a unique point.  For the other domains, we need some more setup.

\begin{remark}
    The reader should be warned here that one has to deal with the case where all coordinates $\hx_U$ are contained in a collapsed point.  When this happens, there is not an obvious choice of collection of minimal domains to work with, and in order to determine coordinates back in the intervals $T_U$, one must use all of the hierarchical data at hand. 
\end{remark}

Suppose that $U \in \calU$ and $\hx_U \notin \hT^e_U$.  Let $C$ denote the cluster in $T_U$ so that $q_U(C) = \hx_U$.  The first step is to associate a marked point $l_V \in \{a,b\}$ to each $\nest_{\calU}$-minimal domain $V \nest U$ with $\delta^V_U \subset C$.  We then use these $l_V$ to hone the clusters of the larger $U \in \calU$ in which such $V$ are involved.

Suppose then that $U \in \calU$ is $\nest_{\calU}$-minimal.  Since $\hT_U$ has no internal cluster points, $\hat{x}_U$ is either a marked point---either $\ha_U$ or $\hb_U$---or a point in $\hT^e_U$.  If the former, let $l_U$ be the corresponding marked point---i.e. either $a,b$.  If the latter, set $l_U = \emptyset$.  We note that $\ha_U \neq \hb_U$ by construction when $U$ is $\nest_{\calU}$-minimal.

Now suppose that $U \in \calU$ is not $\nest_{\calU}$-minimal.  We have that either $\hat{x}_U \in \hT^e_U$, in which case we set $l_U = \emptyset$ and $x_U = q_U^{-1}(\hat{x}_U)$.  Otherwise, $\hat{x}_U$ is contained in a collapsed cluster $q_U(C)$ for the cluster $C \subset T_U$.  Let $\calB_{\min}(C)$ denote the set of $\nest$-minimal domains $V\nest U$ so that $\delta^V_U \subset C$.  Observe that if $\BM(C) = \emptyset$, then $\diam_{T_U}(C)$ is bounded in terms of $r_1,r_2$ and hence in terms of $\mathfrak S$.

We now explain the honing process for domains in $\BM(C)$, which will associate to $\hx_U$ a bounded diameter subset of the cluster $C$ in $T_U$ (Lemma \ref{lem:bdd int total}).

Let $V \in \BM(C)$.  Then we set 
$$C^V_U = \hull_C(C \cap (\calN_{r_2}(\delta^V_U) \cup l_V)),$$
which is a subinterval of $C$, see Figure \ref{fig:C^V_U}.  Finally, let $B_U(\hx) = \bigcap_{V \in \BM(C)} C^V_U$ be the total intersection of the $C^V_U$.

\begin{figure}
    \centering
    \includegraphics[width=.8\textwidth]{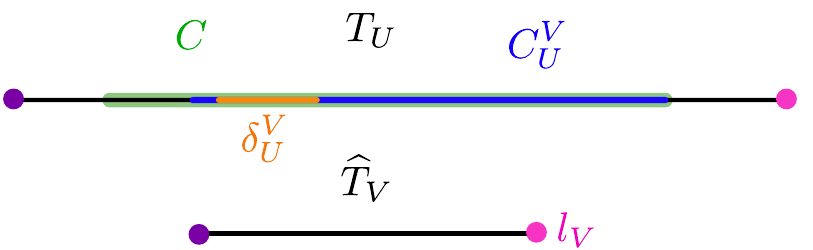}
    \caption{The cluster $C$ (green) in $T_U$ is identified with $\hx_U$ in $\hT_U$.  Since $V \in \BM(C)$, its shadow $\delta^V_U$ (orange) avoids all branches of $T_U$.  The marked points and leaves $l_V$ (pink) chosen by $\hx_V$ determine the half-subtree $C^V_U$ (blue) of $C$ chosen by $V$.} 
    \label{fig:C^V_U}
\end{figure}

In Lemma \ref{lem:bdd int pair} below, we prove that if $V, W \in \BM(C)$, then $C^V_U \cap C^W_U \neq \emptyset$.  Since each $C^V_U$ is a subinterval, it follows from the Helly property for trees that $B_U(x)$ is nonempty.  Finally, in Lemma \ref{lem:bdd int total} we prove that this intersection has diameter bounded in terms of $\mathfrak S$.

Thus we can define a tuple $\Omega(\hx) = (x_U)_{U \in \calU}$ as follows:
\begin{itemize}
    \item For $U \in \calU$ with $U \in E(\hx)$, we set $x_U = q_U^{-1}(\hx_U)$.
    \item For $U \in \calU$ with $U \notin E(\hx)$ and $U \nest_{\calU}$-minimal, we let $x_U$ be any point in $q_U^{-1}(\hx_U)$.
    \item Otherwise, we let $x_U \in B_U(\hx)$.
\end{itemize}

The following is immediate from the construction:

\begin{lemma}\label{lem:B_U in cluster}
For every $\hx \in \calQ$ and $U \in \calU$, we have $q_U(x_U) = \hx_U$.
\end{lemma}

In Proposition \ref{prop:Q-consistent} below, we prove that $(x_U)$ is a $\omega$-consistent tuple for $\omega = \omega(\mathfrak S)>0$.

\begin{remark}[Cubical flavor]
The process of cluster honing is one of the few places where one sees cubical-like features appearing in a non-explicitly cubical part of the construction.  In particular, it is crucial that the two endpoints $\ha_V, \hb_V$ for every $\nest_{\calU}$-minimal $V$ are distinct in $\hT_V$.  Hence $\nest$-minimal domains determine a choice of endpoint between $a,b$, like a wall choosing a half-space.

When $U$ is not $\nest_{\calU}$-minimal and $\hd^V_U = \hx_U$, then the coordinate $\hx_V$ either chooses neither endpoint of $\hT_V$ (if $\hx_V \in \hT^e_V$), or chooses one of $\ha_V,\hb_V$.  This choice then determines a sort of ``half-interval'' $C^V_U$ of $C$, where we can think of $\delta^V_U$ as functioning like a hyperplane here.  In Lemmas \ref{lem:bdd int pair} and \ref{lem:bdd int total} below, we see that the intersection over all such ``half-intervals'' $B_U = \bigcap_{V \in \BM(C)} C^V_U$ is nonempty and has bounded diameter, and this ultra-filter-like behavior is forced by consistency of the tuple $\hx$.
\end{remark}

\begin{remark}[A warning]
    The reader should be warned that one has to deal with the case of $\hx \in \calQ$ where $\hx_U$ is a cluster or endpoint for every coordinate $U$.  While, as in the previous remark, such points determine unambiguous coordinates in $T_U$ for all $\nest$-minimal domains $U \in \calU$, there is no way to leverage this data alone into choosing coordinates in $T_V$ for domains $V$ higher up the $\nest$-lattice.  One must utilize the full set of relations between all of the domains.
\end{remark}

\subsection{Bounding intersections}

In this subsection, we prove the supporting statements we need to define the map $\Omega:\calQ \to \prod_{U \in \calU} T_U$.  Again fix $\hx \in \calQ$.

The first says that the $C^V_U$ for $V \in \BM(C)$ have pairwise nonempty intersection:

\begin{lemma}\label{lem:bdd int pair}
If $V, W \in \BM(C)$, then $C^V_U \cap C^W_U \neq \emptyset$.  Moreover, if $d_{T_U}(\delta^V_U, \delta^W_U)>r_2$, then at least one of $\delta^V_U \subset C^W_U$ or $\delta^W_U \subset C^V_U$ holds.
\end{lemma}

\begin{proof}[Sketch of proof]

We provide a sketch of the proof, see \cite[Lemma 9.9]{Dur_infcube}.  We shall observe item (1) during the course of the proof of the sketch.

By definition of $C^V_U = \hull_C((\calN_{r_2}(\delta^V_U) \cup l_V)\cap C)$, we may assume that $\calN_{r_2}(\delta^V_U) \cap \calN_{r_2}(\delta^W_U) = \emptyset$, and hence that $V \pitchfork W$ by item (3) of Lemma \ref{lem:interval control}.  The idea now is that $\delta^V_U, \delta^W_U$ must occur separately along the cluster $C$.  Without loss of generality, we can assume that $\delta^V_U$ separates $a_U$ from $\delta^W_U$.  One can then use the consistency of $(\hx_U)$ and Lemma \ref{lem:interval control} and \ref{lem:collapsed interval control} to show that at least one of $l_V = b$ or $l_W = a$, which would force $C^V_U \cap C^W_U \neq \emptyset$, as required; see Figure \ref{fig:cluster_int}.

\begin{figure}
    \centering
    \includegraphics[width=.8\textwidth]{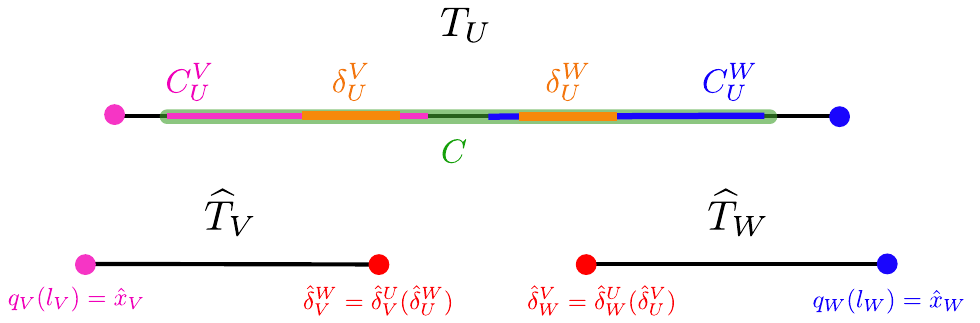}
    \caption{The impossible situation in which $V\pitchfork W \in \BM(C)$ choose non-intersecting sides $C^V_U$ (pink) and $C^W_U$ (blue) of $C$ (green).  This forces inconsistency of $\hx_V$ and $\hx_W$.} 
    \label{fig:cluster_int}
\end{figure}
\end{proof}

As discussed above, we can define $B_U = \bigcap_{V \in \BM(C)}C^V_U$ and it must be non-empty by the Helly property for trees as long as $\BM(C) \neq \emptyset$.

\begin{remark}\label{rem:Helly}
The Helly property for trees says that if one has a pairwise-intersecting collection of subtrees then the collection has nonempty total intersection.  In the general version of our current construction from \cite{Dur_infcube}, the trees and subtrees in question need not be finite diameter.  This is one place where the arguments in our current context are simpler than that considered in \cite{Dur_infcube}.
\end{remark}

The next lemma bounds the diameter of $B_U$:

\begin{lemma}\label{lem:bdd int total}
There exists $B_1 = B_1(\mathfrak S,r_2)>0$ such that for all $U \in \calU$, we have 
$\diam_{T_U}(B_U) < B_1$.
\end{lemma}

\begin{proof}

If $B_U$ is too wide (relative to $r_2$), then we can find $V, W \in \BM(C)$ with $\delta^V_U, \delta^W_U \subset B_U$ and $r'< d_{T_U}(\delta^V_U, \delta^W_U) < 2r'$ for $r' = r'(r_2)>0$, since the $\delta$-sets for domains in $\BM(C)$ coarsely cover $C$ with coarseness depending only on $r_2$.  By Lemma \ref{lem:bdd int pair}, we must have at least one of $\delta^V_U \subset C^W_U$ or $\delta^W_U \subset C^V_U$, or possibly both.  In fact, both must hold, for if $\delta^V_U \notin C^W_U$, then $\delta^V_U \notin B_U$, which is impossible.  But then $C^V_U \cap C^W_U \subset \hull_C(\calN_{r_2}(\delta^V_U) \cup \calN_{r_2}(\delta^W_U))$, which has diameter bounded by $2r' + 2r_2$ since $d_{T_U}(\delta^V_U, \delta^W_U) < 2r'$.  Since $B_U \subset C^V_U \cap C^W_U$, this completes the proof.
\end{proof}

We are now in place to state the consistency result:

\begin{proposition}\label{prop:Q-consistent}
There exist $\alpha_{\omega}=\alpha_{\omega}(\mathfrak S,r_1,r_2,R)>0$ and a map $\Omega: \calQ \to \calY$ so that for any $\hx \in \calQ$, the tuple $\Omega(\hx)$ is $\alpha_{\omega}$-consistent.  Moreover, if $\Omega(\hx) = (x_U)$, then $q_U(x_U) = \hx_U$ for each $U \in \calU$.
\end{proposition}

The proof of this proposition is a somewhat elaborate exercise in the definition of $\Omega$ and the construction of the clusters, in which one utilizes many standard HHS-type arguments.  Unlike other proofs in this paper, the proof of this fact in our current setting is not substantially different from proof in the general case in \cite[Subsection 9.5]{Dur_infcube}, beyond the need for working with $\nest$-minimal \emph{bipartite domains} as defined in \cite[Definition 9.2]{Dur_infcube} instead of all $\nest$-minimal domains.  Thus instead of repeating it here, we leave it to the interested reader to check the details.

\subsection{Tree trimming}\label{subsec:tree trimming}

In this subsection, we introduce a useful tool for modifying collapsed interval systems and their associated $0$-consistent sets.  These ``tree-trimming'' techniques will allow us to perform controlled collapses of subintervals across the various collapsed intervals $\hT_U$ while coarsely preserving the associated $\calQ$-distance.  We think of them as a more flexible version of the distance formula, as they allow us to modify distance terms instead of just throwing them out by increasing a threshold.  For our purposes, they will be used in two different ways: We will need them in the proof of coarse surjectivity of $\hPsi:H \to \calQ$, and they will also be useful later in Subsection \ref{subsec:tt cube} to modify the cubical structure on $\calQ$.

The main idea is simple: Given a family of subintervals $\calB$ across of the family of intervals $\hT_U$, we want to collapse these subintervals $\Delta_U:\hT_U \to \hT_{U,0}$ coordinate-wise and then combine them into a global collapsing map $\Delta:\calQ \to \calQ'$, where $\calQ'$ is the $0$-consistent set of $\prod_{U \in \calU} \hT_{U,0}$, defined the same way as in Definition \ref{defn:Q defined}.  Formally, we use the following definition:

\begin{definition}[Admissible trimming setup]\label{defn:trimming setup}
Given $B>0$, a $B$-\emph{admissible trimming setup} is a collection of \emph{trimming intervals} $\calB$ in the various $\hT_{U,e}$ satisfying the following:
\begin{itemize}
\item For each component  $E \in \pi_0(\hT_{U,e})$, there are at most $B$-many components $D \in \calB$ with $D \subset E$.
\item For each component $D \in \calB$ satisfies $|D|<B$.
\end{itemize}
\end{definition}

Given a $B$-admissible trimming setup $\calB$ on $\{\hT_U\}_{U \in\calU}$, for each $U \in \calU$ let $\Delta_U:\hT_U \to \hT_{U,0}$ denote the map which collapses each trimming interval in $\hT_U$ to a point.  These allow us to induce projection data on the $\hT_{U,0}$ as follows: For $V \nest U$ and $V \pitchfork U$, we set $(\hd^V_U)_0 = \Delta_U(\hd^V_U)$, and if $U \nest V$, we set $(\hd^V_U)_0 = \Delta_U \circ \hd^V_U \circ \Delta^{-1}_V$.  We let $\Delta: \prod_{\calU} \hT_U \to \prod_{\calU} \hT_{U,0}$ denote the product map.

\begin{theorem}[Tree Trimming]\label{thm:tree trimming}
There exists $B_{tt} = B_{tt}(\mathfrak S, K)>0$ so that for any $B < B_{tt}$, there exists $L =L(\mathfrak S, B)>0$ so that the following hold:

\begin{enumerate}
\item We have $\Delta(\calQ) \subset \calQ'$.
\item There is a map $\Xi:\calQ' \to \calQ$ so that $\Delta \circ \Xi = id_{\calQ'}$.  In particular, $\Delta(\calQ) = \calQ'$.
\item The map $\Delta:\calQ \to \calQ'$ is an $(L,L)$-quasi-isometry.
\end{enumerate}
\end{theorem}

While the statement is quite believable in light of the Distance Formula \ref{thm:DF}, the proof of this theorem is quite involved.  While these ideas were first developed in greater generality in \cite[Section 10]{Dur_infcube}---where it also works when the underlying data involves more than two points and also hierarchy rays---the proof these is only slightly more involved than that required in our setting.  Moreover, many of the technicalities from dealing with rays are dealt with in separate arguments, so the interested reader is encouraged to see a full proof there.

Instead, we will present a short sketch below, which will hopefully provide enough intuition to make the statement believable.  We have included a proof of the stated version of Theorem \ref{thm:tree trimming} in Appendix \ref{app:TT}.

\begin{proof}[Sketch of proof]
The fact that $\Delta(\calQ) \subset \calQ'$, namely that $\Delta$ preserves $0$-consistency, is immediate from the fact that the $\Delta_U$ just collapse subintervals, so coordinates which were collapsed projections remain collapsed projections, showing item (1).

Item (2) is the most involved part.  For each $U \in \calU$, let $C_U = \Delta^{-1}_U(\hy_U)$.  The goal is to pick, for each $U \in \calU$, a point $\hx_U \in C_U$ and then prove $0$-consistency of the tuple $\hx = (\hx_U)$.  Note that $\Delta(\hx) =\hy$ for such a tuple, by construction.

The proof is similar in spirit to the cluster-honing procedure from Subsection \ref{subsec:hO defined}.  The main difference is that the collapsing subintervals are not automatically associated to a collection of domains, and this is where the work lies.  The argument is inductive on level in the $\nest$-lattice and requires a careful analysis of the construction and number of hierarchal techniques.

Finally, for item (3), since $\Delta|_{\calQ}$ is surjective and distance non-increasing, it suffices to prove that $d_{\calQ} \prec d_{\calQ'}$, and for this it suffices to bound the distances in each $U \in \calU$.

Given $\hx, \hy \in \calQ$ with $\hx' = \Delta(\hx)$ and $\hy' = \Delta(\hy)$, the concern is that somehow one will have arbitrarily shrunk distances during the trimming process.  However, each edge component of $\hT_U$ through which the geodesic $[\hx_U,\hy_U]$ passes can only contain boundedly-many subintervals collapsed.  However, all but at most two of such edge components have endpoints which are cluster points labeled by $\hd^V_U$ for $V$ $\nest$-minimal.  Since for such $V$, the map $\Delta_V:\hT_V \to \hT'_V$ collapses subintervals of total diameter at most $B^2$, we can use these $\nest$-minimal domains to account for all information lost up the hierarchy.

Making this precise takes care and utilizes Proposition \ref{prop:bounding containers}.  The interested reader can see \cite[Section 10]{Dur_infcube}. 
\end{proof}

\subsection{Converting to simplicial intervals}

One application of the Tree Trimming Theorem \ref{thm:tree trimming} is that we can replace the family of collapsed intervals $\hT_U$ with a family of simplicial intervals, where the cluster and marked points lie at vertices.  The idea here is that one can trim each edge component $E \subset \hT^e_U$ to have a nonnegative integer length simply by collapsing some small segment of $E$ of length less than $1$.  The family of such segments is an $1$-admissible trimming setup, and this provides us a family of intervals $\{\hT'_U\}_{U \in \calU}$ where $\Delta_U:\hT_U\to \hT'_U$ collapses these small segments, and $\Delta:\calQ \to \calQ'$ combines these to a map between their $0$-consistent sets.

\begin{corollary}\label{cor:simplicial version}
Suppose that $\{\hT_U\}_{U \in \calU}$ are the collapsed intervals associated to an $(r_1,r_2)$-thickening of an $R$-interval system for $a,b \in \calX$.  There exists $L_{simp} = L_{simp}(R, \calX)>0$ so that $\Delta:\calQ \to \calQ'$ is an $(L_{simp},L_{simp})$-quasi-isometry.

\begin{itemize}
\item In particular, up to a uniform quasi-isometry of $\calQ$, we can assume that the collapsed intervals $\hT_U$ are all simplicial intervals with cluster and marked points at vertices. 
\end{itemize}
\end{corollary}

See \cite[Subsection 10.1]{Dur_infcube} for more discussion.  We note that the above corollary will not be used in the proof of the Main Theorem \ref{thmi:main}, because we will carefully construct stable intervals to be simplicial.  However, it is necessary for taking general intervals in each hyperbolic spaces and performing the above constructions to form a hierarchical family of intervals (see Definition \ref{defn:HFI}, which is necessary for our proof of the distance formula below (Corollary \ref{cor:DF lower bound}).

\subsection{From $\Omega:\calQ \to \prod_{U \in \calU} T_U$ to $\hO:\calQ \to H$} \label{subsec:upgrade to HHS}

We return to our main task, namely defining a quasi-isometry $\hO:\calQ \to H$.  So far, we have defined a map $\Omega:\calQ \to \prod_{U \in \calU} T_U$, which was the bulk of the work.  Our next step is to upgrade this to a map $\hO:\calQ \to H$.  Since this part of the argument mostly uses standard HHS techniques, we will just outline it, with appropriate citations to \cite{Dur_infcube}.

For each $\hx \in \calQ$, we first take our tuple $\Omega(\hx) = (x_U) \subset \prod_{\calU} T_U$ and push it forward to a tuple $(y_U) \in \prod_{\calU} \calC(U)$ by $\phi_U(x_U) = y_U$, where $\phi_U:T_U \to \calC(U)$ are the maps provided by the $R$-interval system structure (Definition \ref{defn:interval axioms}).  These combine to define a map $\Phi:\prod_{U \in \calU} T_U \to \prod_{U \in \calU} \calC(U)$.  By Proposition \ref{prop:Q-consistent}, there is some $\alpha_{\omega} = \alpha_{\omega}(\mathfrak S, R,r_1,r_2)>0$ so that the tuple $\Omega(\hx) \in \prod_{U \in \calU}T_U$ is $\alpha_{\omega}$-consistent.  Using just the $R$-system axioms and standard HHS techniques, \cite[Proposition 6.17]{Dur_infcube} implies that $(y_U) \in \prod_{U \in \calU} \calC(U)$ is consistent in the sense of Definition \ref{defn:consistency} for some consistency constant depending only on $\mathfrak S, \alpha_{\omega}$. 

Since $d_V(a,b) < K$ for all $V \in \mathfrak S - \calU$, it follows that by choosing $y_V \in \pi_V(a)$ for all $V \in \mathfrak S - \calU$, then the complete tuple $(y_U) \in \prod_{U \in \mathfrak S} \calC(U)$ is $\beta$-consistent for $\beta = \beta(\mathfrak S, R, K)>0$ (as $\alpha_{\omega}$ only depended on $\mathfrak S$ and $R$).

Now with a complete consistent tuple in hand, we can apply the Realization Theorem \ref{thm:realization} to define a map $\mathrm{Real}_{\beta}:\prod_{U \in \mathfrak S} \calC(U) \to \calX$ which takes each $\beta$-consistent tuple and produces a point $x \in \calX$ whose projections coarsely coincide with the tuple, with error depending on $\beta, \mathfrak S$.

Finally, \cite[Lemma 6.6]{Dur_infcube}, which depends on the more general fact \cite[Proposition 6.3]{HHS_II}, provides a $(J,J)$-coarse retraction map $\mathrm{ret}_H:\calX \to H$, where $H = \hull_{\calX}(a,b)$ where $J = J(\mathfrak S)>0$.  This map is just closest-point projection to $\hull_U(a,b)$ in each $\calC(U)$ plus Theorem \ref{thm:realization}. 

Combining these, we get a map $\hO:\calQ \to H$ given by $\hO = \mathrm{ret}_H \circ \mathrm{Real}_{\beta} \circ \Phi \circ \Omega$.  Our main goal is to show that $\hPsi \circ \hO:\calQ \to \calQ$ is coarsely the identity, where $\hPsi:H \to \calQ$ is the map defined in Proposition \ref{prop:hPsi defined}.

Before analyzing this total composition, we want to look at an intermediate step.  For convenience, if $\alpha>0$, we let $\calZ_{\alpha}$ denote the collection of $\alpha$-consistent tuples in $\prod_{U \in \calU} T_U$.  Define
$$\Phi_{\alpha}:\calZ_{\alpha} \to H$$ by $\mathrm{ret}_H \circ \mathrm{Real}_{\beta} \circ \Phi$.

The following proposition says that this map behaves coordinate-wise as one might expect upon composition with the coordinate maps $\psi_U:\calX \to T_U$, which we recall are defined by $\psi_U = \phi^{-1}_U\circ p_U \circ \pi_U$, where $p_U:\calC(U) \to \phi_U(T_U)$ is closest-point projection.  

\begin{proposition}\label{prop:interval coordinate control}
There exists $D_1 = D_1(K, \mathfrak S)>0$ so that if $x = (x_V) \in \calZ_{\alpha_{\omega}}$, then we have
\begin{enumerate}
\item $d_{T_U}(x_U, \psi_U \circ \Phi_{\alpha_{\omega}}(x))<D_1$ for all $U \in \calU$, and
\item $\#\{U \in \calU| d_{T_U}(x_U, \psi_U \circ \Phi_{\alpha_{\omega}}(x)) > 2\alpha_{\omega} + 3R\} < D_1$.
\end{enumerate}
\begin{itemize}
\item In particular, by choosing $L>2\alpha_{\omega} + 3R$, we get
$$\sum_{U \in \calU} [d_{T_U}(x_U, \psi_U \circ \Phi_{\alpha_{\omega}}(x)]_L < D_1^2.$$
\end{itemize}
\end{proposition}

\begin{proof}[Sketch]
This proposition is a special case of \cite[Proposition 11.4]{Dur_infcube}, and the proof is almost entirely the same.  The proof of item (1) is a definition chase.  It is basically saying that consistency in the interval setting is roughly consistency in the hierarchical setting, and that flowing from the former to the latter and back only moves coordinates a bounded amount.

On the other hand, the proof of (2)---which says that only boundedly-many coordinates move more than a bounded amount---is more involved, but also fairly standard.  It uses a variation on the Passing-Up property from Lemma \ref{lem:passing-up}, and the BGI property from Definition \ref{defn:interval axioms}.  Set $\psi_U \circ \Phi_{\alpha_{\omega}}(x) = y = (y_U)$.  The idea here is that if there are too many coordinates where $d_{T_U}(x_U, y_U) > 2\alpha_{\omega} + 3R$, then we can force the existence of some coordinate $T_W$ where $d_{T_W}(x_W,y_W)>D_1$, which is impossible by (1).
\end{proof}

\subsection{The map $\hO:\calQ \to H$ and coarse surjectivity of $\hPsi$}

We are ready to analyze the map $\hO = \Phi_{\alpha_{\omega}} \circ \Omega:\calQ \to H$.    Since $\hPsi(H) \subset \calQ$ by Proposition \ref{prop:hPsi defined}, we see that $\hPsi \circ \hO:\calQ \to \calQ$.  We prove that this map is coarsely the identity:

\begin{proposition}\label{prop:coarse surjectivity}
There exists $M_0 = M_0(R, \mathfrak S)>0$ so that for any $\hx \in \calQ$, we have $d_{\calQ}(\hx, \hPsi \circ \hO(\hx)) < M_0$. 
\end{proposition}

\begin{proof}[Sketch]
This proposition is a special case of the proof contained in \cite[Subsection 11.2]{Dur_infcube}.  Since the details of the proof are very similar to the original, we will again only provide a sketch.

Fix $\hx \in \calQ$ and let $\hy \in \hPsi \circ \hO(\hx)$. We will argue coordinate-wise since that is how the $\ell^1$ distance on $\calQ$ is defined.  The difference from the usual HHS style argument here is that $d_{\calQ}$ is an $\ell^1$ sum not involving a threshold.  To get around that issue, we will make an argument using Proposition \ref{prop:interval coordinate control} and Tree Trimming \ref{thm:tree trimming} as follows.  We show that all coordinate distances $d_{\hT_U}(\hx_U, \hy_U)$ are bounded in terms of $\mathfrak S$ and $K$, with all but boundedly-many (in terms of $\mathfrak S$) being bounded in terms of $\mathfrak S$.  We can then collapse the subintervals between $\hx_U$ and $\hy_U$ in the latter domains while only incurring a bounded error globally, resulting in images for $\hx,\hy$ which are bounded distance in the quotient, showing that $\hx,\hy$ were bounded distance in $\calQ$.

Let $\Omega(\hx) = (x_U)$ and $\Psi \circ \hO(\hx) = (y_U)$, and observe that since $q\circ \Omega(\hx) = \hx$, the tuples $(x_U), (y_U)$ are exactly the type in the statement of Proposition \ref{prop:interval coordinate control}.  Moreover, since $q \circ \Omega(\hy) = \hy$ by construction, while $q$ is $1$-Lipschitz by item (4) of Lemma \ref{lem:collapsed interval control}, it will suffice to control the coordinate-wise distances of the $d_{T_U}(x_U,y_U)$ in order to utilize the above argument.

Let $\calV = \{U \in \calU| \hx_U \neq \hy_U\}$.  One proceeds inductively on the $\nest_{\calU}$-level.  Let $\calV_{min} \subset \calV$ denote the $\nest_{\calU}$-minimal domains in $\calV$.

By item (8) of Lemma \ref{lem:collapsed interval control}, for all but $A$-many $U \in \calV_{min}$ with $A = A(\mathfrak S)>0$, we have $\hx_U,\hy_U \in \{\ha_U,\hb_U\}$, and hence for all but at most $A$ domains in $\calV_{min}$, we have that $\hx_U,\hy_U$ are the (opposite) endpoints of $\hT_U$.  By item (6) of that lemma, $d_{T_U}(a_U,b_U) > d_{\hT_U}(\ha_U,\hb_U) \succ K$, while $K$ can be taken to be as much larger than $2\alpha_{\omega} + 3R$ as we need, as $\alpha_{\omega}$ and $R$ depend only on $\mathfrak S$.  Since Proposition \ref{prop:interval coordinate control} bounds the number of domains $U \in \calU$ where $d_{T_U}(x_U,y_U)> 2\alpha_{\omega} + 3R$, we can conclude that $\#\calV_{min} < D_1 + A$.

Now let $\calV' \subset \calV - \calV_{min}$ denote the set of non-minimal domains $V$ where $\hx_V \neq \hy_V$ and $d_{T_V}(x_V,y_V) < 2\alpha_{\omega} + 3R$, which we note only excludes boundedly-many domains by item (2) of Proposition \ref{prop:interval coordinate control}.

Since the maps $q_V:T_V \to \hT_V$ are $1$-Lipschitz (item (4) of Lemma \ref{lem:collapsed interval control}), we have $d_{\hT_V}(\hx_V,\hy_V) < 2\alpha_{\omega} + 3R$.  The idea now is to use the Tree Trimming Theorem \ref{thm:tree trimming} to collapse the geodesics $[\hx_V,\hy_V]$ between $\hx_V,\hy_V$ in $\hT_V$ to points.  Each interval $\hT_V$ contains at most one such geodesic, and each such geodesic has a length bounded independently of $K$.  Hence the set of such intervals across both $\nest$-minimal and non-minimal domains with this distance bound is an admissible trimming set in the sense required for Theorem \ref{thm:tree trimming}.

If we take $\hT'_V$ to be the interval obtained from $\hT_V$ post-trimming and $\calQ'$ the corresponding $0$-consistent set, then the images $\hx', \hy'$ of $\hx,\hy$ now satisfy $\hx'_V = \hy'_V$ for all but boundedly-many $V$, and the remaining distances are all bounded by some function of $D_1$ by item (1) of Proposition \ref{prop:interval coordinate control}.  In particular, $d_{\calQ'}(\hx,\hy)$ is bounded.  But the map $\Delta:\calQ \to \calQ'$ is a uniform quasi-isometry, so that means that $d_{\calQ}(\hx,\hy)$ is also bounded, which is what we wanted.  We leave checking the details to the reader.
\end{proof}

\section{Building a cube complex from a hierarchical family of intervals}\label{sec:cubulation}

The goal of this section is to explain how to take a special collection of intervals satisfying certain HHS-like properties, and to produce from them a CAT(0) cube complex.  The details of this construction motivate the stabilization techniques we develop later in this article.

\subsection{Hierarchical families of intervals}

In this subsection, we axiomatize the properties we need from a family of intervals in order to build a cube complex.  The purpose of this axiomatization is to allow us to modify our construction in various ways while preserving the axioms, as any setup which satisfies the axioms produces a cube complex.  The properties themselves are motivated by the HHS axioms and the properties in Lemma \ref{lem:collapsed interval control}.

\begin{definition}[Hierarchical family of intervals] \label{defn:HFI}
A \emph{hierarchical family of intervals (HFI)} is the following collection of objects and properties:

\begin{enumerate}
\item A finite index set $\calU$ of \emph{domains} with \emph{relations} $\nest, \pitchfork, \perp$ so that
\begin{itemize}
    \item $\nest$ is anti-reflexive and anti-symmetric, and gives a partial order with a bound on the length of chains;
    \item $\pitchfork$ is symmetric and anti-reflexive;
    \item There is a unique $\nest$-maximal element $S$;
    \item Any pair of domains $U,V \in \calU$ satisfies exactly one of the relations. 
\end{itemize}
\item There is a bound on the size of any pairwise $\perp$-comparable subset of $\calU$. \label{item:finite dim}
\item To each $U \in \calU$ there is an associated finite simplicial interval $\hT_U$. \label{item:simplicial}
\item A pair of labels $F = \{x,y\}$ which determined the endpoints, called \emph{marked points} $\hx_U, \hy_U$, of the interval $\hT_U$ for each $U \in \calU$.
\item A family of \emph{relative projections} determining:
\begin{itemize}
    \item For each $U, V \in \calU$ with $V \nest U$ or $V \pitchfork U$, there is a vertex $\hd^V_U \in \hT^{(0)}_U$ called a \emph{cluster point}.
 
\end{itemize}
    \item There exist projection maps $\hd^U_V: \hT_U \to \hT_V$ when $V \nest U$, satisfying the following properties:\label{item:BGI HFI}
    \begin{enumerate}
    	\item (Consistency) For each $f \in F$, we have $\hd^U_V(\hf_U) = \hf_V$.
        \item (BGI) If $C \subset \hT_U - \hd^V_U$ is a component, then $\hd^U_V(C)$ coincides with $\hf_V$ for $f \in F$ with $\hf_U \in C$. 
    \end{enumerate}

\item If $U,V,W \in \calU$ with $V,W \nest U$, $V \pitchfork W$, and $\hd^V_U \neq \hd^W_U$, then $\hd^U_V(\hd^W_U) = \hd^W_V \in \hT_V$. \label{item:delta control HFI}
\end{enumerate}
\end{definition}

\begin{remark}
In \cite[Definition 12.1]{Dur_infcube}, the definition of a \emph{hierarchical family of trees} allows for the trees to have more than $2$ leaves.  This is only needed for modeling tuples of $3$ or more points.
\end{remark}

The following is an immediate consequence of our work in Section \ref{sec:interval systems}, especially Lemma \ref{lem:collapsed interval control}:

\begin{lemma}\label{lem:HHS to HFI}
Let $R>0$, $r_1 > R_1=R_1(R,\mathfrak S)>0$ as in Proposition \ref{prop:hPsi defined} and $r_2 > R_2=R_2(R,\mathfrak S)>0$ as in Lemma \ref{lem:collapsed interval control}.  For any $(r_1,r_2)$-thickening of an $R$-interval system for $a,b \in \calX$, the associated family of collapsed intervals and their collapsed projections form a hierarchical family of intervals.
\end{lemma}

As above in Definition \ref{defn:Q defined}, we can define a $0$-consistent subset of the product $\prod_U \hT_U$.  Our next goal is the following theorem:

\begin{theorem}\label{thm:Q CCC}
The $0$-consistent subset of an HFI is a CAT(0) cube complex.  Moreover, there is a bijective correspondence between its hyperplanes and the edges of the factors. 
\end{theorem}

\subsection{Median graphs and dual cube complexes}

The proof of Theorem \ref{thm:Q CCC} that we give here will use the notion of median spaces.  The connection between median graphs and CAT(0) cube complexes goes back to Roller's \cite{Roller} reframing of Sageev's construction \cite{Sageev:cubulation}, see also \cite{CN05, Nica}.  Notably, Chepoi \cite[Theorem 6.1]{Che00} showed that every median graph is the $1$-skeleton of a unique CAT(0) cube complex.  Bowditch explains these connections in detail in \cite[Section 11]{bowditch2022median}.

Let $X$ be a space and $\mm:X^3\to X$ a ternary operator.  We call $(X,\mm)$ a \emph{median algebra} if for all $a,b,c,d,e \in X$ we have
$$\mm(a,b,c) = \mm(b,a,c) = \mm(c,a,b), \hspace{.1in} \mm(a,a,b) = a, \hspace{.1in} \mm(a,b,\mm(c,d,e)) = \mm(\mm(a,b,c),\mm(a,b,d),e).$$

A graph $\Gamma$ is \emph{median} if all choices of $v_1,v_2,v_3 \in \Gamma^{(0)}$, there exists a unique point lying on some geodesic from $v_i,v_j$ for all $i,j$; then $\mm(v_1,v_2,v_3)$ is that unique point.

Given $a,b \in X$, the (median) \emph{interval} between $a,b$ is $I(a,b) = \{x \in X| \hspace{.01in} \mm(a,x,b) = x\}$.  We say a median algebra $(X, \mm)$ is \emph{discrete} when $I(a,b)$ is finite for every $a,b \in X$.  A \emph{subalgebra} is a subset $S \subset X$ which is closed under $\mm$.  In particular, locally finite median graphs are discrete median algebras.

Every median algebra $(X,\mm)$ determines a graph called its \emph{adjacency graph} $\Gamma(X)$ as follows.  We say that $a,b \in X$ are \emph{adjacent} when $I(a,b) = \emptyset$.  The vertices of $\Gamma(X)$ are the points of $X$, with $a,b \in X$ connected by an edge in $\Gamma(X)$ exactly when $a,b$ are adjacent.  This graph is median exactly when $(X,\mm)$ is a discrete median algebra \cite[Lemma 16.1.2]{bowditch2022median}.

Every CAT(0) cube complex $X$ admits a canonical median structure as follows.  For any $a,b \in X^{(0)}$, let $J(a,b)$ denote the set of $\ell^1$-geodesics between $a,b$.  Then given $a,b,c \in X^{(0)}$, we have $\mm(a,b,c) = J(a,b) \cap J(a,c) \cap J(b,c)$ is a single point, and $(X, \mm)$ is a median structure with $X^{(1)}$ a median graph, with notably $J(a,b) = I(a,b)$ for all $a, b \in X^{(0)}$.  Moreover, it is immediate that $X^{(1)} = \Gamma(X^{(0)})$, since median adjacent vertices in $X^{(0)}$ are exactly those spanning an edge in $X^{(1)}$.

\begin{proof}[Proof of Theorem \ref{thm:Q CCC}]
   First, observe that $\calY = \prod_{U \in \calU} \hT_U$ is a product of simplicial trees and hence a CAT(0) cube complex.
   
   Next, observe that $\calQ \subset \calY$ is a cubical subcomplex as follows.  For any $\hx \in \calQ$, let $C \subset \calY$ be the minimal cube containing $\hx$ in its interior.  Presuming $C$ is not a vertex, minimality says that $\hx_U$ is in the interior of an edge in each of the defining coordinate intervals of $C$, whose domain labels are pairwise orthogonal by item (9) of Lemma \ref{lem:collapsed interval control}.  This means that we can freely alter those coordinates of $\hx$ to be any choice of points in those edges, meaning $C \subset \calQ$.  Hence $\calQ$ is a cubical subcomplex of $\calY$.

   It is enough then to show that $\calQ^{(0)} \subset \calY$ is a discrete median subalgebra.  For once that is done, it will follow that $\Gamma(\calQ)$ is a median graph, but evidently $\Gamma(\calQ) = \calQ^{(1)}$, as $\calQ$ is a cubical subcomplex of $\calY$.  Since every median graph is the $1$-skeleton of a unique CAT(0) cube complex \cite[Theorem 6.1]{Che00}, it follows that $\calQ$ must be that cube complex, completing the proof.

   The claim boils down to three exercises that we leave for the reader:
   \begin{itemize}
       \item Any (finite) product of median spaces admits the structure of a median space, where the median on the product is the tuple of medians in the coordinates.
       \item For any $\ha,\hb,\hc \in \calQ^{(0)}$, we have $\left(\mm_U(\ha_U,\hb_U,\hc_U)\right)_U \in \calQ$.  This can be proven using a basic consistency argument (see \cite[Lemma 15.12]{Dur_infcube}).
       \item With this median structure, $a,b \in \calQ^{(0)}$ are median adjacent exactly when they span an edge in $\calQ^{(1)}$.
   \end{itemize}

   Finally, observe that the hyperplanes of $\calY$ are of the form $\hh \times \prod_{V\neq U} \hT_V$ for $\hh$ some hyperplane in $\hT_U$.  Since $\calQ$ is a median subalgebra of $\calY$, the hyperplanes of $\calQ$ are just hyperplanes of $\calY$ intersected with $\calQ$, so the moreover statement follows.  This completes the proof.
\end{proof}

\subsection{More tree trimming and hyperplane deletions}\label{subsec:tt cube}

In Subsection \ref{subsec:tree trimming}, we introduced our tree trimming techniques, which allowed us to take a family of collapsed intervals and collapse a further family of subintervals while maintaining control over their associated $0$-consistent subsets.  In this subsection, we need a further refinement applied to HFIs and their associated cube complexes.

To set up the statement, let $\{\hT_U\}_{U \in \calU}$ be a hierarchical family of intervals, as in Definition \ref{defn:HFI}.  Fix a constant $B>0$, which will control the size and number of subintervals we will collapse across all of the intervals.  In our applications, we will be able to control the size of $B$, but notably not with respect to the largeness constant $K$, unlike in the statement of Theorem \ref{thm:tree trimming}.

\begin{definition}[Admissible simplicial trimming setup]\label{defn:simplicial trimming setup}
Given $B>0$, a $B$-\emph{admissible simplicial trimming setup} is a collection of \emph{trimming intervals} $\calB$ in the various $\hT_{U,e}$ satisfying the following:
\begin{enumerate}
\item For each $D \in \calB$ with $D \subset E \in \pi_0(\hT_{U,e})$, we have \label{item:trimming integer}
\begin{enumerate}
\item $D$ is a subinterval of $E$ with endpoints at vertices, and
\item the length $l(D)$ of $D$ is a positive integer whose length is bounded above by $B$.
\end{enumerate}
\item The number of subintervals in $\calB$ is bounded by $B$, that is $\#\calB < B$.
\end{enumerate}
\end{definition}

\begin{remark}
We note that the conditions in Definition \ref{defn:simplicial trimming setup} guarantee that each quotient interval $\hT_{U,0}$ is a simplicial interval of the usual type we consider.  The requirement that each segment $D \in \calB$ has endpoints at vertices guarantees that when we collapse such a segment, then it only affects the interval-hyperplanes it contains, preserving those in its complement.
\end{remark}

We use the same associated notation as in Subsection \ref{subsec:tree trimming}.  Furthermore, we note that if $\hT^e_U \subset \calB$, then $\diam (\hT_U) < B^2$ and $\hT_{U,0}$ is a point.  Let $\calU_0 \subset \calU$ denote the subset of domains for which $\hT_{U,0}$ is not a point.  By ignoring the domains in $\calU - \calU_0$, we can consider the set of collapsed intervals $\{\hT_{U,0}\}_{U \in \calU_0}$, where we remove any labels $\hd^V_U$ from $\hT_U$ from any domain $V \nest U$ with $V \in \calU - \calU_0$.

The following theorem expands on Theorem \ref{thm:tree trimming} in the HFI setting:

\begin{theorem}\label{thm:simplicial tree trimming}
Let $B>0$ and $\{\hT_U\}_{U \in \calU}$ be an HFI.  For any $B$-admissible simplicial trimming setup $\calB$, the following hold:

\begin{enumerate}
\item The collections $\{\hT_{U,0}\}_{U \in \calU}$ and $\{\hT_{U,0}\}_{U \in \calU_0}$ are HFIs.  We let $\calQ', \calQ_0$ denote their corresponding $0$-consistent subsets.
\item We have $\Delta(\calQ) = \calQ'$, where $\Delta:\prod_U \hT_U \to \prod_U \hT'_U$ is the global collapsing map.
\item There exists a map $\Xi:\calQ' \to \calQ$ with $\Delta \circ \Xi = id_{\calQ'}$.
\item The map $\Delta|_{\calQ}: \calQ \to \calQ'$ is a $(1, B^2)$-quasi-isometry.
\item At the level of cube complexes, the map $\Delta:\calQ \to \calQ'$ is a hyperplane deletion map deleting at most $B^2$-many hyperplanes.
\end{enumerate}

\end{theorem}

\begin{proof}
Item (1) is straightforward, since we are just collapsing simplicial subintervals.  Items (2)--(3) follow from Theorem \ref{thm:tree trimming}, where in the argument we just replace the HFI axioms (Definition \ref{defn:HFI}) with the appropriate properties from Lemma \ref{lem:collapsed interval control}.  The fact that $\Delta$ deletes at most $B^2$-many hyperplanes from $\prod_U \hT_U$, so item (5) follows from the ``moreover'' part of Theorem \ref{thm:Q CCC}.  Finally, item (4) follows immediately from item (5).  This completes the proof.

\end{proof}

\subsection{Hierarchy paths and the distance formula}\label{subsec:HP and DF}

With Theorem \ref{thm:Q CCC} in hand, we can make some observations about the cubical structure of $\calQ$ and how it relates to its hierarchical nature.  In particular, we get a proof of the (more difficult) lower bound of the Distance Formula \ref{thm:DF} and the existence of hierarchy paths.

The following lemma is \cite[Lemma 14.7]{Dur_infcube}, and we have included its proof here for completeness.

\begin{lemma}\label{lem:NR hp in Q}
Let $\hx,\hy \in \calQ$.  Let $\gamma$ be any combinatorial geodesic in $\calQ$ from $\hx$ to $\hy$.  Then for all $U \in \calU$, the projection $\hpi_U(\gamma) \subset \hT_U$ is an unparameterized geodesic between $\hx_U$ and $\hy_U$.
\end{lemma} 

\begin{proof}
Since combinatorial geodesics do not cross a hyperplane twice and the restriction quotient $\hpi_U:\calQ \to \hT_U$ is $1$-Lipschitz, it follows that $\hpi_U(\gamma)$ does not backtrack either.  This completes the proof.

\end{proof}

This lemma plays two roles: it's useful for proving that an isomorphism of HFIs induces an isomorphism of their cubical duals (Proposition \ref{prop:HFI isomorphism}), and, in the next proposition, we use it directly to show that combinatorial geodesics in a cubical model get sent to hierarchy paths (Definition \ref{defn:hp}) in the ambient HHS.  That is the content of the next proposition, which is simplified version of \cite[Proposition 14.8]{Dur_infcube}:

\begin{proposition}\label{prop:cubical hp}
Let $a,b \in \calX$, $K \geq K_0$ where $K_0 = K_0(\mathfrak S)>0$ is from Theorem \ref{thm:DF}, $\calU = \Rel_K(a,b)$, and let $H = \hull_{\calX}(a,b)$.  Let $\{T_U\}_{U \in \calU}$ be any $R$-interval system for $a,b$ and $\{\hT_U\}_{U \in \calU}$ the HFI associated to any $(r_1,r_2)$-thickening for $r_1 \geq R_1, r_2 \geq R_2$  with $r_1$ a positive integer.  Let $\calQ$ be the corresponding $0$-consistent set and $\hO:\calQ \to H$ the quasi-isometry constructed in Section \ref{sec:hPsi qi}.
\begin{itemize}
\item If $\ha, \hb$ are the points $\calQ$ corresponding to $a, b$ and $\gamma$ is a combinatorial geodesic between them, then $\hO(\gamma) \subset H$ is an $(L_{hp},L_{hp})$-hierarchy path between $a,b$ for $L_{hp} = L_{hp}(R,r_1,r_2,K, \mathfrak S)>0$.
\end{itemize} 
\end{proposition}

\begin{proof}
The following is a fairly complete sketch of the proof, which is simpler than the more general version from \cite[Proposition 14.8]{Dur_infcube}.

Since the map $\hO:\calQ \to H$ is effectively defined coordinate-wise, its enough to consider the coordinate-wise maps $B_U:\hT_U \to T_U$ as described in Subsection \ref{subsec:hO defined}, as the map $\phi_U:T_U\to \calC(U)$ is an $(R,R)$-quasi-isometric embedding by assumption.  For this, we can let $\gamma_U$ denote the coordinate-wise projection of $\gamma$ to each $\hT_U$, and we want to show that $B_U(\gamma_U)$ is an unparameterized quasi-geodesic in $T_U$.

By construction, $B_U(\gamma_U) \subset [a_U,b_U] = T_U$, so it suffices to show that this path never backtracks an unbounded amount.  Proceeding by induction in the $\nest_{\calU}$-level in $\calU$, if $U \in \calU$ is $\nest$-minimal, then the amount of collapsing in $q_U:T_U \to \hT_U$ is bounded by $2r_1$, and we are done by Lemma \ref{lem:NR hp in Q}.  For non-minimal domains $U \in \calU$, note that by Lemma \ref{lem:NR hp in Q}, backtracking in $T_U$ is only possible if $\gamma_U(t), \gamma_U(s)$ are at the same collapsed point corresponding to some cluster $C$.  If there are times $s < t$ so that $d_{T_U}(B_U(\gamma(s), b_U) - d_{T_U}(B_U(\gamma(s)),b_U) \gg 100r_2 + 100r_1$, then there is some $\nest_{\calU}$-minimal domain $V \nest U$ with $\hd^V_U = q_U(C)$ for which $B_U(\gamma_U(t))$ is on the side of $a_U$ while $B_U(\gamma_U(s))$ is on the side of $b_U$, with both $(50r_2 + 50r_1)$-away from $\delta^V_U$.  But then consistency and the BGI properties from Lemma \ref{lem:interval control} implies that $B_V(\gamma_V(t)), a_V$ and $B_V(\gamma_V(s)),b_V$ are $R$-close, and so $\gamma_V(t) = \ha_V$ and $\gamma_V(s) = \hb_V$, which contradicts the base case of this induction.  This completes the proof.
\end{proof}

\begin{remark}
     Since the map $\hO: \calQ \to \calX$ is uniformly quasi-median \cite[Section 15]{Dur_infcube}, one could alternatively use the fact that median paths in HHSes are hierarchy paths \cite[Lemma 1.37]{HHS_quasi}.
\end{remark}

As a consequence of Propositions \ref{prop:cubical hp} and \ref{prop:lower bound}, we obtain the lower bound of the Distance Formula \ref{thm:DF}:

\begin{corollary}\label{cor:DF lower bound}
There exists $K_0 = K_0(\mathfrak S)>0$ such that for all $K>K_0$ and any $x, y \in \calX$, we have
$$\sum_{U \in \mathfrak S} \left[d_U(x,y)\right]_K \prec d_{\calX}(x,y).$$
\end{corollary}

\subsection{Isomorphisms of HFIs and their associated cube complexes}

In this final subsection, we will discuss when two HFIs to determine isomorphic cube complexes.  This statement, while unsurprising, motivates what happens in the subsequent sections, hence its placement here.

Let $\{\hT_U\}_{U \in \calU}$ and $\{\hT'_U\}_{U \in \calU}$ be two HFIs over a common index set $\calU$ and common relations, i.e. $U,V \in \calU$ have the same relation in both HFIs.  Any family of maps $i_U:\hT_U \to \hT'_U$ induces a global map $I:\prod_{U \in \calU} \hT_U \to \prod_{U \in \calU} \hT'_U$, where $I((\hx_U)) = (i_U(\hx_U))$.

The following proposition, which is an interval version of \cite[Proposition 14.12]{Dur_infcube}, gives sufficient conditions for the induced map to be an isometry $I:\calQ \to \calQ'$ which induces a cubical isomorphism between their dual cube complexes.

\begin{proposition}\label{prop:HFI isomorphism}
Let $\{\hT_U\}_{U \in \calU}$ and $\{\hT'_U\}_{U \in \calU}$ be two HFIs over a common index set $\calU$ with marked points $F = \{a,b\}$, and let $\calQ, \calQ'$ be their corresponding $0$-consistent sets.  Suppose that for each $U \in \calU$, there is a simplicial isometry $i_U:\hT_U \to \hT'_U$ which also satisfies:
\begin{enumerate}
	\item We have $i_U(\ha_U) = \ha'_U \in \hT'_U$ and $i_U(\hb_U) = \hb'_U \in \hT'_U$.
	\item For all $V  \in \calU$ with $V \nest U$ or $V \pitchfork U$, we have $i_U(\hd^V_U) = \left(\hd^V_U\right)'$.
	\item For all $V \in \calU$ with $U \nest V$ and $\hx_V \in \hT_V$, we have $\left(\hd^V_U\right)'(i_V(\hx_V)) = i_U(\hd^V_U(\hx_V))$.
\end{enumerate}
Then the corresponding map induces a cubical isomorphism $I:\calQ \to \calQ'$.
\end{proposition}

\begin{proof}[Sketch of proof]

The proof of this version below is the same as the original, which is, in turn, an exercise in the construction from this section plus Lemma \ref{lem:NR hp in Q}.  The main idea is that items in the statement guarantee that the global map $I$ induced from the coordinate-wise isometries preserves $0$-consistency and distances, i.e. we have $I:\calQ \to \calQ'$ is an isometric embedding, with surjectivity following from a consistency argument.  Since the interval-wise maps respect the simplicial structure, we have that $I$ is also a cubical isomorphism via Lemma \ref{lem:halfspaces_bijection}.  We leave the details to the reader.

\end{proof}

\section{Controlling domains} \label{sec:controlling domains}

In this section, we turn toward the proof of our main technical result, namely the Stable Cubical Interval Theorem \ref{thm:stable cubes}, which is a stabilized version of the cubical model construction for a pair of points $a,b$ in any HHS $\calX$.  The goal is to show that if we perturb $a,b$ a small amount to a pair of points $a',b'$ with $d_{\calX}(a,a') \leq 1$ and $d_{\calX}(b,b') \leq 1$, then we can produce cubical models $\calQ$ for $a,b$ and $\calQ'$ for $a',b'$ which are isomorphic up to deleting boundedly-many hyperplanes.  Doing so requires carefully controlling the main inputs into the cubulation machine---namely the families of intervals we construct from the projections of the points.

The first step in the process involves explaining how to gain control over the set of relevant domains (Definition \ref{defn:relevant}) for a pair of points $F = \{a,b\} \subset \calX$ and its perturbation $F' = \{a',b'\}$  as above.  In the next sections, we explain how to take this projection data to build a family of stable intervals that we can plug into the cubulation machine.

The main bit of work for this first part has already been done for us by Bestvina--Bromberg--Fujiwara \cite{BBF} and Bestvina--Bromberg--Fujiwara--Sisto \cite{BBFS}.  The work we did with Minsky and Sisto in \cite{DMS_bary} mostly adapts their setup to the HHS setting, and makes some straight-forward but important observations.

In \cite[Subsections 2.3 and 2.4]{DMS_bary}, we were dealing with the situation where $F,F'$ were finite sets of points with cardinality possibly larger than $2$.  Nonetheless, all of the arguments there reduce to considering pairs of points, and so the proofs for our current setting are essentially the same.  As such, we will mainly explain the setup and key statements, only giving minimal sketches of proofs.

\subsection{Colorability}

The hierarchical setting in which \cite{BBFS} works is for the following class, which was first defined in \cite[Definition 2.8]{DMS_bary} and later used in \cite{HP_proj,Petyt_quasicube}:

\begin{definition}[Colorable HHS]\label{defn:colorable}
We say that an HHS
$(\calX, \mathfrak S)$ is \emph{colorable} if there exists a decomposition of
$\mathfrak S$ into finitely many families $\mathfrak S_i$, so that each $\mathfrak S_i$ is
pairwise-$\pitchfork$.

\begin{itemize}
\item We refer to the $\mathfrak S_i$ as \emph{BBF families}.
\end{itemize}
\end{definition}

\begin{remark}
The notion of colorability was inspired by work of Bestvina-Bromberg-Fujiwara \cite{BBF}, who proved that the curve graph is finitely-colorable.  As a consequence of this, the set of subsurfaces of a given finite-type surface $S$ can be decomposed into finitely-many families of pairwise-interlocking subsurfaces, which the above definition directly generalizes.  See \cite[Subsection 2.3]{DMS_bary} for a detailed discussion of this connection.
\end{remark}

\subsection{Stable projections}

Our next goal is to state a theorem which equips any colorable HHS with a set of projections which satisfy certain key stability properties.  We need some definitions first.

Given a pair of points $F = \{x,y\} \subset \calX$ and $K>0$, we let $\calU(F) = \Rel_K(F)$.  If $V \in \calU(F)$, then we let $\calU^V(F) = \{W \in \calU(F) | W \sqsubsetneq V\}$.

The following definition \cite[Definition 2.13]{DMS_bary} describes the domains we want to control when perturbing $F$:

\begin{definition}[Involved domains]\label{defn:involved}
Let $F = \{x,y\}, F' = \{x', y'\} \subset \calX$ be so that $d_{\calX}(x,x') \leq 1 $ and $d_{\calX}(y,y') \leq 1$.  For $K>0$, we say that $V \in \calU(F) \cup \calU(F')$ is \emph{involved in the transition} between $F$ and $F'$ if one of the following holds:

\begin{enumerate}
\item $\pi_V(F) \neq \pi_V(F')$.
\item $\calU^V(F) \neq \calU^V(F')$.
\end{enumerate}
\end{definition}

The following, which combines the statements of \cite[Theorem 2.9 and Proposition 2.14]{DMS_bary}, states that, for any colorable HHS $\calX$, we can replace its given projections with coarsely equivalent projections with which we can control involved domains up to a bounded error:

\begin{theorem}\label{thm:controlling domains}
Let $(\calX, \mathfrak S)$ be a colorable HHS with standard projections $\hat{\pi}_{-}, \hat{\rho}^{-}_{-}$.  Then there exists $\theta>0$ and new projections $\pi_{-}, \rho^{-}_{-}$ with the same domains and ranges, respectively, such that the following hold:

\begin{enumerate}
\item The new projections $\pi_{-}, \rho^{-}_{-}$ coarsely coincide with the old ones $\hat{\pi}_{-}, \hat{\rho}^{-}_{-}$.  More precisely:
\begin{enumerate}
\item If $X,Y$ lie in different $\mathfrak S_j$ and $\hat{\rho}^X_Y$ is defined, then $\rho^X_Y = \hat{\rho}^X_Y$.
\item If $X,Y \in \mathfrak S_j$, then $d^{Haus}_{Y}(\rho^X_Y, \hat{\rho}^X_Y)<\theta.$
\item If $x \in \calX$ and $Y \in \mathfrak S$, then $d^{Haus}_Y(\pi_Y(x),\hat{\pi}_Y(x))<\theta$.
\end{enumerate}
\item There exists $K_0 = K_0(\calX)>0$ and $N_0(\calX)>0$ so that if $K>K_0$, then the following holds.  Let $x,y,x',y' \in \calX$ with $d_{\calX}(x,x') \leq 1 $ and $d_{\calX}(y,y') \leq 1$, and let $\calU = \Rel_K(\{x,y\})$ and $\calU' = \Rel_K(\{x',y'\})$ be the relevant sets for the new projections $\pi_-$.  Then:
\begin{enumerate}
\item $\#(\calU \triangle \calU') < N_0$, and if $V \in \calU \triangle \calU'$, then $d_V(x,y) \prec K$ and $d_V(x',y') \prec K$. \label{item:bounded symdiff}
\item  There are at most $N_0$ domains $V \in \calU(F) \cup \calU(F')$ involved in the transition between $F$ and $F'$. \label{item:bounded involved}
\end{enumerate}
\end{enumerate}
\end{theorem}

\begin{proof}[Sketch of proof]
The existence of new projections satisfying the conditions of item (1) is essentially an immediate consequence of the construction in \cite[Theorem 4.1]{BBFS}.  Item (2), which regards stability of the relevant sets $\calU, \calU'$, requires a bit more work, but the key ideas are as follows:

Using \cite[Proposition 2.8]{HHS_II} (which generalizes \cite{CLM12, BKMM:qirigid, BBF}), there is a partial order $\prec$ on $\Rel(x,y)$, which becomes a total order on $\Rel^i(x,y) = \Rel_K(x,y) \cap \mathfrak S_i$ for each $i$.  In particular, we can think of each $\Rel^i(x,y)$ as an interval $X_1 \prec \cdots \prec X_n$.  One of the additional properties that the new projections $\pi_{-}, \rho^{-}_{-}$ satisfy (see \cite[Theorem 2.9]{DMS_bary}) is that if $X \prec Y$ along both $\Rel(x,y)$ and $\Rel(x',y')$, then $\pi_Y(x) = \pi_Y(x') = \rho^X_Y$ and $\pi_X(y) = \pi_X(y') = \rho^Y_X$.  Using this, one can show that $\Rel(x,y)$ and $\Rel(x',y')$ can only differ at their first or last elements in the order $\prec$.  Hence $\#\left(\Rel^i(x,y) \triangle \Rel^i(x',y')\right) \leq 4$; see \cite[Lemma 2.11 and Proposition 2.12]{DMS_bary}.  Since there are only boundedly-many BBF families $\mathfrak S_i$, item \eqref{item:bounded symdiff} follows easily.

A similar argument can now be used to control the number of involved domains of type (1) of Definition \ref{defn:involved}, see the proof of \cite[Proposition 2.14]{DMS_bary}.  For involved domains of type (2), this is an immediate consequence of item \eqref{item:bounded symdiff} and the Covering Lemma \ref{lem:covering}, which bounds the number of domains $W \in \calU$ into which a given $V \in \calU$ can nest.
\end{proof}

\section{Stable Intervals}\label{sec:stable intervals}

With the new HHS projections provided by Theorem \ref{thm:controlling domains}, we now turn toward converting stabilized projection data into stable families of intervals, in partivuvlar stabilizing the input into the cubical model construction from Section \ref{sec:interval systems}.

Our ultimate goal is to produce cubical models for pairs of nearby points that differ by a controlled number of hyperplane deletions (Theorem \ref{thm:stable cubes}). Since the hyperplanes in the cubical models come from edges of the simplicial intervals in our HFIs (Theorem \ref{thm:Q CCC}), it will suffice to build HFIs where the intervals only differ by boundedly-many bounded-size collections of edges.  That is the content of the Stable Intervals Theorem \ref{thm:stable intervals}.

We first explain the build-up to this theorem and defer the explicit statement.  We will work in the setting of a hyperbolic space with a hierarchical-like configuration.  The construction is a somewhat long and delicate exercise in hyperbolic geometry and we will mainly provide a detailed sketch.

\subsection{Basic setup and assumptions}\label{subsec:stable setup}

For the rest of this section, let us assume that $\calZ$ is a $\delta$-hyperbolic geodesic space.  The first bit of notation we want to set is a choice of geodesic between any pair of points.  That is, for any $x,y \in \calZ$, we fix (once and for all) a geodesic $\lambda(x,y)$.  More generally, if $C,C' \subset \calZ$ are finite subsets, then we let $\lambda(C,C')$ denote a choice of geodesic between them whose length satisfies $|\lambda(C,C')| = d_{\calZ}(C,C')$, and we require symmetry, namely $\lambda(C,C') = \lambda(C',C)$.

It will also be necessary to perform a slight modification on the geodesic function $\lambda$, so that its resulting length is an integer.  To fix this, we let $\lambda_0(C,C')$ denote the restriction of $\lambda(C,C'):[0, d_{\calZ}(C,C')]\to \calZ$ to the interval $[0, \lfloor{d_{\calZ}(C,C')}\rfloor]$.  The effect of this is to trim $\lambda(A,B)$ so that it has positive integer length and their lengths satisfy $l(\lambda(C,C')) -l(\lambda_0(C,C')) < 1$.  While not necessary, one could make this symmetric by trimming both the sides by half of $l(\lambda(C,C')) -l(\lambda_0(C,C'))$.  In the course of the arguments, however, one can usually work with $\lambda$ instead of $\lambda_0$.

Now, for the rest of this section, we fix a pair of points $\{a,b\} \subset \calZ$.  Furthermore, let $\calY \subset \calZ$ be a finite (but possibly arbitrarily large) set of points of $\calZ$ and let $\epsilon>0$ so that $d_{\calZ}(\lambda(a,b), y) < \epsilon/2$ for all $y \in \calY$.  We call such an arrangement $(a,b;\calY)$ an $\epsilon$-\emph{setup} in $\calZ$.

In the hierarchical setting, $\calZ$ will be one of the hyperbolic spaces $\calC(U)$ for $U \in \calU$, and the points $\calY$ will be the points in $\rho^V_U$ for $V \nest U \in \calU$.  The proximity condition of the points in $\calY$ to $\lambda(a,b)$ comes from BGIA \ref{ax:BGIA}.

\begin{notation}[$O(\epsilon)$-notation]
In this section, we will often want to get controlled bounds over various distance quantities, where the bounds will depend on our choice of $\epsilon$ and the hyperbolicity constant $\delta$ for $\calZ$.  When a number $A$ is bounded above as a function of $\epsilon, \delta$, we will simply write 
$$A < O(\epsilon),$$
since $\delta$ is given to us, while we need some flexibility to choose $\epsilon$.
\end{notation}

The following lemma, which is essentially a consequence of the Morse Lemma for hyperbolic spaces, gives us control over how far $\lambda(C,C')$ can get from $\lambda(a,b)$ when $C, C' \subset \calY$, and are thus $\epsilon$-close to $\lambda(a,b)$:

\begin{lemma}[Fellow-traveling, $\delta$-hyperbolic]\label{lem:Morse close}
Let $\calZ$ be $\delta$-hyperbolic and geodesic, and $\gamma$ a geodesic between $a,b \in \calZ$.  For every $\epsilon>0$ there exists $\epsilon'=O(\epsilon)>0$ so that the following holds:

\begin{itemize}
\item Suppose $\gamma'$ is a geodesic connecting points $x,y \in \calN_{\epsilon}(\gamma)$, and that $x',y'$ are closest points to $x,y$ on $\gamma$, respectively.  Then any geodesic $[x,y]$ satisfies $d_{\calZ}^{Haus}([x,y], [x',y']) < \epsilon'$.
\end{itemize}
\end{lemma}

As a consequence, if $C, C' \subset \calY$ are finite subsets, then $\lambda(C,C') \subset \calN_{\epsilon'}(\lambda(a,b))$.  Finally, it will be convenient for us to assume that $\epsilon' > \epsilon$.

\subsection{Clusters and the cluster graph} \label{subsec:clusters}

Our first goal is to group the points of $\calY$ into \emph{clusters} of nearby sets.  To that end, let $E = E(\epsilon)> \epsilon$ be sufficiently large, to be determined below.

Let $\calC_E$ denote the graph whose vertices are points of $\calY \cup \{a,b\}$, with two points $x,y \in \calY \cup \{a,b\}$ connected by an edge exactly when $d_{\calZ}(x,y) < E$.  The connected components of $\calC_E$ are called $E$-\emph{clusters}, or simply \emph{clusters} when $E$ is fixed.

For any $E$-cluster $C$, let $s(C) = \hull_{\lambda(a,b)}(p_{\lambda(a,b)}(C))$ denote the hull (in $\lambda(a,b)$) of the closest point projection of $C$ to $\lambda(a,b)$.  The set $s(C)$ is called the \emph{shadow} of $C$ on $\lambda(a,b)$.

The following lemma controls how shadows of clusters are arranged along $\lambda(a,b)$. Its proof is straightforward (see \cite[Lemma 3.6]{DMS_bary}), so we leave it to the reader to check (:

\begin{lemma} \label{lem:cluster arrangement}
If $E > 5\epsilon$ and $C \neq C'$ are distinct $E$-clusters, then $s(C) \cap s(C') = \emptyset$.
\end{lemma}

In what follows, we will need to adjust the cluster constant $E$ a number of times, but we will usually omit the constant and refer to them simply as clusters.

We now can define our notion of cluster separation using Lemma \ref{lem:Morse close}.  The idea here is that clusters are $\epsilon$-close to $\lambda(a,b)$, and hence geodesics connecting them stay $\epsilon'$-close to $\lambda(a,b)$ by Lemma \ref{lem:Morse close}.

\begin{definition}
Let $C_1, C_2, C_3$ be clusters.  Given $\epsilon'>0$, we say $C_2$ $\epsilon'$-\emph{separates} $C_1$ from $C_3$ if there exists a minimal length geodesic connecting $C_1$ to $C_3$ which passes within $2\epsilon'$ of $C_2$.

\begin{itemize}
\item Given $E>0$, the \emph{$E$-cluster separation graph} for $(a,b; \calY)$, denoted $\calG_E = \calG_{E,\epsilon'}(a,b;\calY)$ is defined as follows: The vertices of $\calG_E$ are clusters, and two vertices are connected by an edge if and only if they are not $2\epsilon'$-separated by another cluster.
\end{itemize}
\end{definition}

The following is essentially a consequence of Lemma \ref{lem:cluster arrangement}.  It states that the cluster graph is an interval with the order of vertices corresponding to the order of the cluster shadows along $\lambda(a,b)$.  We omit its proof, which uses only hyperbolic geometry and Lemma \ref{lem:cluster arrangement}.

\begin{proposition}\label{prop:cluster sep graph}
Let $\calZ$ be $\delta$-hyperbolic, $\lambda(a,b)$ a geodesic connecting $a,b \in \calZ$, and $\calY \subset \calZ$ finite with $\calY \subset \calN_{\epsilon}(\lambda(a,b))$.  Also let $\epsilon' = \epsilon'(\epsilon, \delta)>0$ as in Lemma \ref{lem:Morse close}, with also $\epsilon' > \epsilon$.

There exists $E = E(\epsilon, \delta)>0$ so that the $E$-cluster separation graph $\calG_{E,\epsilon'}(a,b;\calY)$ is an interval whose order coincides with the order of cluster shadows along $\lambda(a,b)$, with the endpoint clusters containing $a,b$, respectively. 
\end{proposition}

\subsection{Stable intervals defined}

We are now ready to describe the stable intervals.

By Proposition \ref{prop:cluster sep graph}, the cluster separation graph is an interval.  Our stable intervals admit a decomposition into two families of intervals, namely of \emph{edge} components and \emph{cluster} components.  An edge component connects adjacent clusters at their nearest points, while a cluster component connects the endpoints of  whatever edge components are connected to a given cluster.

\begin{definition}[Stable interval]\label{defn:stable interval}
Let $\calZ$ be $\delta$-hyperbolic, $a,b \in \calZ$ with $\lambda(a,b)$ a geodesic connecting them, and $\calY \subset \calZ$ finite with $\calY \subset \calN_{\epsilon/2}(\lambda(a,b))$.  The \emph{stable interval} for this setup is the union of the following two families of geodesic segments:

\begin{enumerate}
\item (Edge components) For each adjacent pair of $C, C'$ of clusters in $\calG = \calG_{E,\epsilon'}(a,b;\calY)$, the \emph{edge component} between them is $\lambda_0(C,C')$.
\begin{itemize}
\item The edge components comprise the \emph{edge forest}, which we denote by $T_e$.
\end{itemize}
\item (Cluster components) For each cluster $C$, we let $C_1,C_2$ be the clusters adjacent to $C$ in $\calG$.  Let $r(C) = C \cap (\lambda(C_1,C) \cup \lambda(C,C_2) \cup \{a,b\})$, where $\lambda(C_i,C) = \emptyset$ if $C$ is an endpoint cluster, for the appropriate $i \in \{1,2\}$.  The cluster $C$ then determines the \emph{cluster component} $\mu(C) = \lambda(r(C))$.
\begin{itemize}
\item The cluster components comprise the \emph{cluster forest}, which we denote by $T_c$.
\end{itemize} 
\end{enumerate}

\begin{itemize}
\item The \emph{stable interval} is  $T(a,b; \calY) = T_e \sqcup T_c$, by which we mean the abstract union of these two forests, where we attach $\lambda_0(C_1,C_2)$ to the corresponding endpoints of $\mu(C_1), \mu(C_2)$ for adjacent $C_1, C_2$ as determined by $\calG$.
\item We let $\phi:T(a,b;\calY) \to \calZ$ denote the map induced from the fact that its pieces are geodesics in $\calZ$. 
\end{itemize}
\end{definition}

\begin{remark}
    We note that stable intervals are abstractly intervals, but not necessarily as subsets of $\calZ$.  In particular, the points $a,b$ need not be the ``endpoints'' of the stable interval $T$, either abstractly or in $\calZ$.  However, their images in the collapsed stable interval will be the endpoints, see Subsection \ref{subsec:collapsed stable intervals}.
\end{remark}

The following lemma is the key to controlling the structure of $T(a,b;\calY)$.  We omit its proof since it again uses only hyperbolic geometry, though see Figure \ref{fig:cluster close} for a schematic:

\begin{lemma}\label{lem:cluster close}
Let $C_1, C_2$ be distinct clusters along $\lambda(a,b)$.  Suppose that $p_1 \in C_1$ and $p_2 \in C_2$ are closest points, while $x_1 \in s(C_1)$ and $x_2 \in s(C_2)$ are closest points.  Then $d_{\calZ}(x_i, p_i) \leq 7\epsilon$ for $i=1,2$.

As a consequence, the following hold:
\begin{enumerate}
\item For distinct clusters $C_1,C_2$, we have $|d(C_1,C_2) - d(s(C_1),s(C_2))| < 14\epsilon.$
\item For any $p \in T_e$, we have that $d_{\calZ}(p, \calY \cup \{a,b\}) \leq d_{\calZ}(p, \partial T_e) + O(\epsilon).$
\item For any components $V,W$ of $T = T_e \cup T_c$, we have $\diam_{\calZ}(s(V)\cap s(W)) < O(\epsilon)$.
\end{enumerate}
\end{lemma}

\begin{figure}
    \centering
    \includegraphics[width=.75\textwidth]{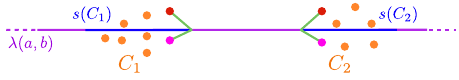}
    \caption{Proof of Lemma \ref{lem:cluster close}: Clusters can have multiple pairs of closest points.  In this example, the red and pink points on $C_1$ are equidistant from their matching point on $C_2$, while the red and pink cluster points in $C_1$ are both within $7\epsilon$ of the endpoint of the shadow $s(C_1)$ nearest $s(C_2)$, and the same for the roles of $C_1,C_2$ switched.  In subsequent arguments, this will allow us to conclude that connecting points between clusters are close even after we alter the clusters.}
    \label{fig:cluster close}
\end{figure}

\begin{figure}
    \centering
    \includegraphics[width=1\textwidth]{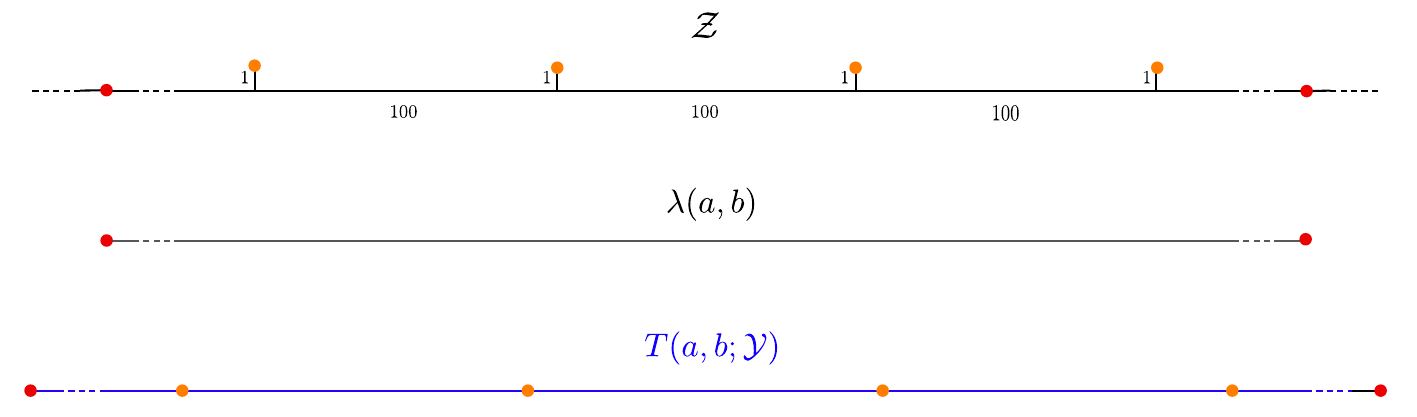}
    \caption{This figure, which is essentially \cite[Figure 11]{DMS_bary}, is an example of two phenomena.  First, the edge components of a stable interval can have bounded overlap in the space $\calZ$.  This can lead to a bounded multiplicative error in the quasi-isometric embedding $\phi:T(a,b;\calY) \to \calZ$; see item (3) of Proposition \ref{prop:stable interval structure}.  In the figure, the multiplicative constant is at least $\frac{102}{100}$.  Second, the image of $T(a,b;\calY)$ in $\calZ$ need not itself be an interval, as the clusters can form leaves on the image.}
    \label{fig:pathologies}
\end{figure}

With the above lemma in hand, we can prove the following structural proposition:                                                                                                                                                                                                                      
 
 \begin{proposition}\label{prop:stable interval structure}
 
There exists $L_0 = L_0(\delta, \epsilon)>0$ so that the following hold:

\begin{enumerate}
\item $T(a,b;\calY)$ is an interval with endpoints contained in $\calY \cup \{a,b\}$ whose images in $\calZ$ are $O(\epsilon)$ of $a,b \in \calZ$.
\item For every cluster $C$,we have that $d_{\calZ}^{Haus}(\mu(C), C) < L_0$ and $\phi(T_C) \subset \calN_{L_0}(\calY \cup \{a,b\})$. 
\item The map $\phi:T(a,b;\calY) \to \calZ$ is an $(L_0,L_0)$-quasi-isometric embedding with image $L_0$-Hausdorff close to $\lambda(a,b)$.
\item Each edge component of $T$ has positive integer length.
\end{enumerate}
 \end{proposition}
 
 \begin{proof}
 
 For (1), observe that $T(a,b;\calZ)$ is an interval by construction and Proposition \ref{prop:cluster sep graph}, and the end clusters contain $a,b$, respectively, also by that proposition.  Since each cluster $C \subset \calN_{E}(\lambda(a,b))$, if $a \in C$, then the outer-most endpoint of $\mu(C)$ must be within $E$ of $a$, as required.
 
 For (2), observe that each cluster $C$ is within Hausdorff distance $O(\epsilon)$ of $s(C)$, while Lemma \ref{lem:cluster close} implies that $\mu(C)$ is within Hausdorff distance $O(\epsilon)$ of $s(C)$.  The conclusion follows.
 
 Item (3) is essentially a consequence of Lemma \ref{lem:cluster close} and the construction as follows.  By Lemma \ref{lem:Morse close}, the image under the map $\phi:T \to \calZ$ of each component of $T_e$ and $T_c$ is within $O(\epsilon)$ of its shadow on $\lambda(a,b)$.  In other words, for any component $V$ of $T_e$ or $T_c$, the shadow map induces a $(1,O(\epsilon))$-quasi-isometry $s:\phi(V) \to s(\phi(V))$.  On the other hand, Lemma \ref{lem:cluster close} says every pair of components of $T$ have shadows on $\lambda(a,b)$ with $O(\epsilon)$-bounded overlap.  It follows that $\phi:T \to \calZ$ is an $(O(\epsilon), O(\epsilon))$-quasi-isometric embedding, as required.
 
 Finally, item (4) follows because we have used the function $\lambda_0$ as our connectors in Definition \ref{defn:stable interval}, which performs a small surgery at one of the ends of $\lambda$ so that the result has integer length.  This completes the proof.
\end{proof}

\subsection{Stable decompositions}\label{subsec:stable decomp}

Having defined stable intervals above, we are almost ready to state the Stable Intervals Theorem \ref{thm:stable intervals}.  The motivation for the definitions in this section comes from our desire to plug things into our cubical model machinery.  We describe how this works in Subsection \ref{subsec:perturbed to stable} below, and we point the reader to it for motivation for what follows.

A stable decomposition involves decomposing a pair of stable intervals for a pair of admissible setups into subintervals, with some maps which identify and organize the various pieces.  The purpose of the next definition is to give shorthand for requiring that the various maps involved respect how the various pieces are oriented towards the endpoints of the ambient intervals.

\begin{definition}[Order-preserving maps]\label{defn:order preserving}
Let $T_1, T_2$ be two intervals with the endpoints of $T_i$ labeled by $a_i,b_i$.  Suppose that $T_i = T_{i,1} \sqcup T_{i,2}$ are subdivisions of $T_i$ into finite collections of subintervals with $\# T_{1,1} = \#T_{2,1}$ (and hence $\#T_{1,2} = \#T_{2,2}$).

\begin{itemize}
\item We say an isometry $i_{E,E'}:E \to E'$ between components $E \subset T_{1,1}$ and $E' \subset T_{2,1}$ is \emph{order preserving} if $i_{E,E'}$ sends the endpoint of $E$ closest to $a_1$ to the endpoint of $E'$ closest to $a_2$.
\item We say a bijection $\alpha:\pi_0(T_{1,1}) \to \pi_0(T_{2,1})$ is \emph{order-preserving} if it coincides with the bijection obtained by recording the orders of appearance of the components of $T_{i,1}$ along $T_i$ from $a_i$ to $b_i$.
\end{itemize}
\end{definition}

\begin{remark}\label{rem:order preserving}
Using the above notation, note that an order-preserving bijection $\alpha:\pi_0(T_{1,1}) \to \pi_0(T_{2,1})$ automatically induces an order-preserving bijection $\beta:\pi_0(T_{1,2}) \to \pi_0(T_{2,2})$.  Moreover, in this paper, we will only be working with intervals between a pair of points, which are usually denoted by (some form of) $a,b$.  We will always take order-preserving to mean as one moves from the $a$-side to the $b$-side of the given intervals.
\end{remark}

We also set one last bit of notation: Given a stable tree $T = T_e \cup T_c$ for an $\epsilon$-setup $(\calY; \{a,b\})$, if $y \in \calY \cup \{a,b\}$, we let $C_y$ denote the cluster containing $y$, and $\mu(C_y)$ the corresponding component of $T^c$.

The following definition contains the stability properties we want:

\begin{definition}[Stable decomposition] \label{defn:stable decomp}
Let $\calZ$ be $\delta$-hyperbolic and geodesic.  
\begin{enumerate}
\item Given $N>0$, two $\epsilon$-setups $(a,b; \calY)$ and $(a',b'; \calY')$ are $(N,\epsilon)$-\emph{admissible} if
\begin{enumerate}
\item $d_{\calZ}(a,a'), d_{\calZ}(b,b') \leq \epsilon$,
\item $\calY, \calY' \subset \calN_{\epsilon/2}(\lambda(a,b)) \cap \calN_{\epsilon/2}(\lambda(a',b'))$,
\item $|\calY \triangle \calY'| < N$.
\end{enumerate}
\item Given an $\epsilon$-setup $(a,b;\calY)$, an \emph{edge decomposition} of its stable interval $T = T_e \cup T_c$ is a collection of subintervals $T_s \subset T_e$.
\item Given $L>0$, $\calY_0 \subset \calY \cap \calY'$, and two $(N,\epsilon)$-admissible setups $(a,b;\calY)$ and $(a',b'; \calY')$, we say that two edge decompositions $T_s \subset T_e$ and $T'_s \subset T'_e$ are $\calY_0$-\emph{stably} $L$-\emph{compatible} if
\begin{enumerate}
\item Each component of $T_s$ and $T'_s$ has positive integer length with endpoints at vertices of $T_e$ and $T'_e$, respectively. \label{item:integer length}
\item There is an order-preserving bijection $\alpha:\pi_0(T_s) \to \pi_0(T'_s)$ between the sets of stable components. \label{item:stable bijection}
\item For each \emph{stable pair} $(E, E')$ identified by $\alpha$, there exists an order-preserving isometry $i_{E,E'}:E \to E'$. \label{item:stable pairs}
\item For all but at most $L$ pairs of stable components $(E,E')$, we have $\phi(E) = \phi(E')$ are identical with $\phi(x) = \phi'(i_{E,E'}(x))$ for all $x \in E$. \label{item:identical pairs}
\item For the (at most) $L$-many remaining stable pairs $(E,E')$, we have $d_{\calZ}(\phi(x), \phi'(i_{E,E'}(x)))< L$ for all $x \in E$. \label{item:close pairs}
\item The complements $T_{e} - T_s$ and $T'_{e} - T'_s$ consist of at most $L$-many \emph{unstable components} of diameter at most $L$.\label{item:unstable components}
\item The induced order-preserving bijection $\beta:\pi_0(T-T_s) \to \pi_0(T' - T'_s)$ satisfies: \label{item:adjacency}
\begin{enumerate}
\item (Endpoints) Let $D_a$ denote the component of $T-T_s$ containing $\mu(C_a)$, and define $D_b,D_{a'},D_{b'}$ similarly.  Then $\beta(D_a) = D_{a'}$ and $\beta(D_b) = D_{b'}$. \label{item:endpoint condition}
\item (Identifying clusters) For any $y \in \calY_0$, let $D_y, D'_y$ denote the components of $T-T_s$ and $T' - T'_s$ containing $\mu(C_y), \mu(C'_y)$, respectively.  Then $\beta(D_y) = D'_y$. \label{item:cluster identify}
\end{enumerate}
\end{enumerate}
\end{enumerate}
\end{definition}

\begin{remark}
The above definition is a refinement of the properties contained in \cite[Stable Tree Theorem 3.2]{DMS_bary}, where we have set up these stable decompositions to plug nicely into the cubical model machinery described earlier in the paper. How this process works is described below in more detail in Subsection \ref{subsec:perturbed to stable}, but for now we motivate some parts of the definition.  Roughly speaking, the points $\calY$ and $\calY'$ stand in for the relative projection data on each interval, with each component of $T - T_s$ and $T' - T'_s$ being labeled by some such point.  Eventually, we will collapse each of the components of $T - T_s$ and $T' - T'_s$ to points.  Item \eqref{item:adjacency} says that the common labels $\calY \cap \calY'$ are identified, as well as the endpoint pairs $a,a'$ and $b,b'$.  The rest of the statement now says that the components of $T_s$ and $T'_s$ are in bijection with each other, where identified components are isometric, and the order-preserving properties explain how to glue these components to obtain an isometry between collapsed intervals.  Finally, the proximity properties in items \eqref{item:identical pairs} and \eqref{item:close pairs} are necessary for controlling the maps back into the ambient HHS; see the proof of Theorem \ref{thm:stable cubes}.  A more general version can be found in \cite[Definition 10.18]{DMS_FJ}. 
\end{remark}

\subsection{The stable interval theorem}

With Definition \ref{defn:stable decomp} in hand, we can now state the main result of this section, Theorem \ref{thm:stable intervals} below.  It says that the stable intervals for a pair of admissible $\epsilon$-setups admit compatible stable decompositions---that is, decompositions into subintervals which are in bijective correspondence with idenitified pairs uniformly close, up to ignored a bounded collection of bounded-length subintervals.  We first state this technical result, and then prove a corollary which gives the main upshot for our purposes of building stable cubical models.  

For a bit of setup, recall the notion of a thickening of an interval from Subsection \ref{subsec:thickenings}.  This was a way of taking an interval with a decomposition $T = A \cup B$ into collections of segments, and expanding and combining one collection of the segments.  Our stable intervals come with such a decomposition $T = T_e \cup T_c$, and we will always mean a thickening of a stable interval $T$ to be a thickening of $T$ along the components of the cluster forest $T_c$.  Finally, we denote the components of the resulting thickening by $T = \bT_{e} \cup \bT_c$, and we note that $\bT_e \subset T_e$ and $T_c \subset \bT_c$. 

\begin{theorem}[Stable intervals]\label{thm:stable intervals}
Let $\calZ$ be $\delta$-hyperbolic and geodesic.  For any $\epsilon, N>0$ and positive integers $r_1,r_2>0$, there exist $L_1 = L_1(\delta, \epsilon, N)>0$ and $L_2 = L_2(\delta, \epsilon, N,r_1,r_2)>0$ so that the following holds.  Suppose $(a,b;\calY)$ and $(a',b';\calY')$ are $(N,\epsilon)$-admissible $\epsilon$-setups and let $T = T_e \cup T_c$ and $T' = T'_e \cup T'_c$ denote their stable intervals.  Then
\begin{enumerate}
\item There exist $(\calY \cap \calY')$-stable $L_1$-compatible decompositions $T_s \subset T_e$ and $T'_s \subset T'_e$, and
\item There exist $(\calY \cap \calY')$-stable $L_2$-compatible edge decompositions $\bT_{s} \subset \bT_{e}$ and $\bT'_{s} \subset \bT'_{e}$ of the $(r_1,r_2)$-thickenings of $T, T'$ along $T_c,T'_c$.
\begin{itemize}
\item Moreover, we have $\bT_s \subset T_s$ and $\bT_s \subset T'_s$.
\end{itemize}
\end{enumerate}

\end{theorem}

\begin{remark}
In the HHS setting, we will be able to control each of the constants $\delta, \epsilon, N, r_1,r_2$ in terms of the ambient HHS structure.  We have written the statement with two conclusions because the first conclusion is a mild reformulation of the original theorem \cite[Theorem 3.2]{DMS_bary}, while the second is what we actually need for the cubulation machine discussed in this paper.  Notably, the second statement will be an easy consequence of the first.
\end{remark}

Before we move onto the proof, we observe the following corollary, whose statement and proof motivate the statement of Theorem \ref{thm:stable intervals}.  In particular, this corollary will allow us to verify the interval-wise condition in Proposition \ref{prop:HFI isomorphism} and give us the cubical isomorphism we need for the Stable Cubulations Theorem \ref{thm:stable cubes}; see Subsection \ref{subsec:perturbed to stable}.

\begin{corollary}\label{cor:interval isometry}
Let $(a,b;\calY)$ and $(a',b'; \calY')$ be $(N,\epsilon)$-admissible $\epsilon$-setups with $(\calY \cap \calY')$-stable $L$-compatible decompositions $T_s \subset T_e$ and $T'_s \subset T'_e$.   Let $\Delta:T \to \hT$ and $\Delta':T' \to \hT'$ denote the quotients obtained by collapsing each component of $T-T_s$ and $T'-T'_s$ to a point.  The following hold:
\begin{enumerate}
\item $\hT$ and $\hT'$ are simplicial intervals where each collapsed component of $T-T_s$ and $T'-T'_s$ is a vertex, and
\item There exists an isometry $\Phi:\hT \to \hT'$ which is induced from the bijections $\alpha:\pi_0(T_s) \to \pi_0(T'_s)$ and $\beta:\pi_0(T-T_s) \to \pi_0(T'-T'_s)$, and the isometries of pairs of stable components $i_{E,\alpha(E)}:E \to \alpha(E)$ for $E \subset T_s$.   Moreover, we have
\begin{enumerate}
\item $\Phi(\Delta(a)) = \Delta'(a')$ and $\Phi(\Delta(b)) = \Delta'(b')$,
\item If $y \in \calY \cap \calY'$ and $D_y \in \pi_0(T - T_s)$ and $D'_y \in \pi_0(T - T_s)$ contain $\mu(C_y)$ and $\mu(C'_y)$, respectively, then $\Phi(\Delta(D_y)) = \Delta'(D'_y)$. 
\end{enumerate}
\end{enumerate}
\end{corollary}

\begin{proof}
First, observe that $\hT$ and $\hT'$ are intervals because $T,T'$ are and each component of $T-T_s$ and $T'-T'_s$ is a subinterval.  They are simplicial intervals with the given description because Theorem \ref{thm:stable intervals} (specifically item (4a) of Definition \ref{defn:stable decomp}) provides that each component of $T_s,T'_s$ has positive integer length, and so the distance between the points corresponding to pairs of collapsed components of $T-T_s$ and $T'-T'_s$ is a positive integer.

For the second conclusion, we can define $\Phi:\hT\to \hT'$ as follows: For each vertex $v \in \hT$ corresponding to a component $C \in \pi_0(T-T_s)$, we define $\Phi(v) = v'$, where $v'$ is the vertex of $\hT'$ corresponding to $\beta(C) \in \pi_0(T'-T'_s)$.  If now $x \in \hT$ is not such a vertex, then $x \in \Delta(E)$ for some component $E \in \pi_0(T_s)$.  So we define $\Phi(x) = \Delta'(i_{E, \alpha(E)}(x))$.  The fact that these piece-wise maps glue coherently is a direct consequence of the fact that $\alpha$ is order-preserving.  The subitems of (2) are direct consequences of item \eqref{item:adjacency} of Definition \ref{defn:stable decomp}.  This completes the proof.
\end{proof}

\begin{remark}
Notably, the above corollary holds whenever we have a stable decomposition, so in particular when this is a stable decomposition arising from the thickened versions of our stable intervals.
\end{remark}

\subsection{Detailed sketch of the proof} \label{subsec:detailed sketch}

In lieu of a full proof (see Appendix \ref{app:SI}), we provide a detailed sketch of Theorem \ref{thm:stable intervals} with figures illustrating the key technical arguments, which are easier to understand at a fairly close level than to write down carefully.  Our proof strategy here follows more closely the proof of \cite[Theorem 10.23]{DMS_FJ} than the original \cite[Theorem 3.1]{DMS_bary}.

The goal is to show that the stable intervals (Definition \ref{defn:stable interval}) for pairs of admissible $\epsilon$-setups admit stable decompositions in the sense of Definition \ref{defn:stable decomp}.  Roughly, these are decompositions of the intervals into collections of subintervals, which are either identical (as in item (d)), or the same length and uniformly close (as in item (e)), or are one of boundedly-many unstable components we can ignore (as in item (f)).  We need to show how to constraints on the input data assumed in item (1) translates into control over the components of the stable intervals.

To begin the sketch, first observe that the two main assumptions of Theorem \ref{thm:stable intervals} on the setups $(a,b;\calY)$ and $(a',b';\calY')$ are that
\begin{itemize}
\item The sets of cluster points $\calY, \calY'$ are close to both $\lambda(a,b)$ and $\lambda(a',b')$ and have bounded symmetric difference, and 
\item The endpoints pairs $(a,b)$ and $(a',b')$ are within $\epsilon$ of each other.
\end{itemize}

This suggests the following iterated proof strategy: 
\begin{enumerate}
\item First, show that we can build a stable decomposition when we add a single nearby cluster point, while keeping the endpoints the same.
\item Next, show that we can iteratively combine stable decompositions to obtain a stable decomposition.
\item Finally, show that we can build a stable decomposition when we only wiggle a single endpoint and keep everything else (including the cluster sets) the same.

\end{enumerate}

Steps (1) and (3) can be thought of as \emph{moves}, where we only change one piece of data at a time, and we will see that type (3) moves are basically type (1) moves.  For step (2), we will want to iteratively perform type (1) moves by adding the extra cluster points of $\calY' - \calY$ to the setup $(a,b;\calY)$, thereby shifting from $(a,b;\calY)$ to $(a,b;\calY \cup \calY')$.  Then we make two type (3) moves to shift $(a,b;\calY \cup \calY') \to (a',b;\calY \cup \calY') \to (a',b';\calY \cup \calY')$.  Finally, we make another sequence of iterated type (1) moves to shift from $(a',b';\calY \cup \calY')$ down to $(a',b';\calY')$.  Since this total sequence of moves is bounded by $N+2$, we will be done.

For step (1), there are two main cases---the \emph{split cluster} case, where we add a point $z$ to $\calY$ far from all $\calY$-clusters; and the \emph{affected cluster} case, where $z$ is close to some existing $\calY$-cluster.  See Figures \ref{fig:split cluster} and \ref{fig:affected cluster}, respectively, for relevant schematics.  The philosophy here is simple: adding a cluster point only affects things locally, thus one only needs to make bounded local modifications.

Step (2) of the proof now involves iteratively applying step (1), and thus one must show that one can coherently combine two pairs of stable decompositions across a common setup.  Figure \ref{fig:iterated stable decomp} shows a schematic for the argument, in which one has three stable trees $T_1,T_2,T_3$ with associated stable decompositions, so that those for $T_1,T_2$ and $T_2,T_3$ are compatible.  The idea is to intersect the two decompositions on $T_2$ and pull them back to $T_1,T_3$.

The associated stability constants increase a controlled amount with each iteration, but we are assuming that we only need to perform boundedly-many such iterations in Theorem \ref{thm:stable intervals}.

\smallskip

Finally, step (3) is an application of steps (1) and (2).  To get the idea, we can consider the following restricted setup, which we highlight to focus on what happens when we perturb an endpoint.  Let $a,b,a' \in \calZ$ and $\calY \subset \calZ$ be finite with $\calY \subset \calN_{\epsilon/2}(\lambda(a,b)) \cap \calN_{\epsilon/2}(a',b)$.  The following collection of setups, each of which differs by adding or deleting one cluster point, or by switching the role that $a,a'$ or $b,b'$ play between cluster or endpoint:
\vspace{-.09in}
\begin{multicols}{3}
\begin{enumerate}
\item $(a,b;\calY)$,
\item $(a,b;\calY \cup \{a\})$,
\item $(a,b;\calY \cup \{a, a'\})$,
\item $(a',b;\calY \cup \{a,a'\})$,
\item $(a',b;\calY \cup \{a'\})$,
\item $(a',b;\calY)$.
\end{enumerate}
\end{multicols}

Hence we can use steps (1),(2), (3) plus iteration to change the roles of endpoints once we have added all of the necessary cluster points.  See Appendix \ref{app:SI} for the full argument.

\begin{figure}[b]
    \centering
    \includegraphics[width=1\textwidth]{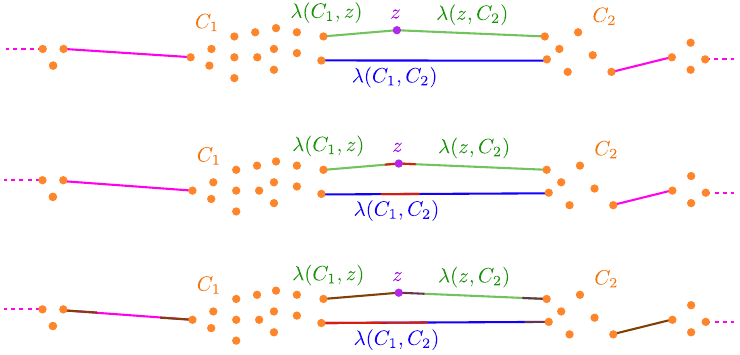}
    \caption{Step (1) of the proof of Theorem \ref{thm:stable intervals}, split cluster case:  In the first figure, the two edge components of the two stable intervals are shown.  The point $z$ forms an individual cluster and therefore must be inserted as a vertex in the graph $\calG$, between clusters $C_1, C_2$.  Every cluster to the left of $C_1$ is unchanged, as are all of the connecting segments (in pink), and same for those to the right of $C_2$.  The second figure shows the stable decomposition of the stable intervals.  Here, we need to include the shadow of the cluster $z$ on $\lambda(C_1,C_2)$ as an unstable component.  This cuts $\lambda(C_1,C_2)$ into two segments, which are paired with $\lambda(C_1,z)$ and $\lambda(z,C_2)$, though these need to be slightly truncated so guarantee they have the same lengths; the unstable segments accomplishing this are red.  The third figure shows the stable decomposition for their $(r_1,r_2)$-thickenings, which are subsets of the edge components $T_e \subset T$ and $T'_e \subset T$, are indicated in brown.  Note that this can result in collapsing an entire edge component if its diameter is less than $r_2-2r_1$.  These brown segments refine the stable decompositions $T_s \subset T_e$ and $T'_s \subset T'_e$.}    \label{fig:split cluster}
\end{figure}

\newpage
\begin{figure}[!htb]
    \centering
    \includegraphics[width=1\textwidth]{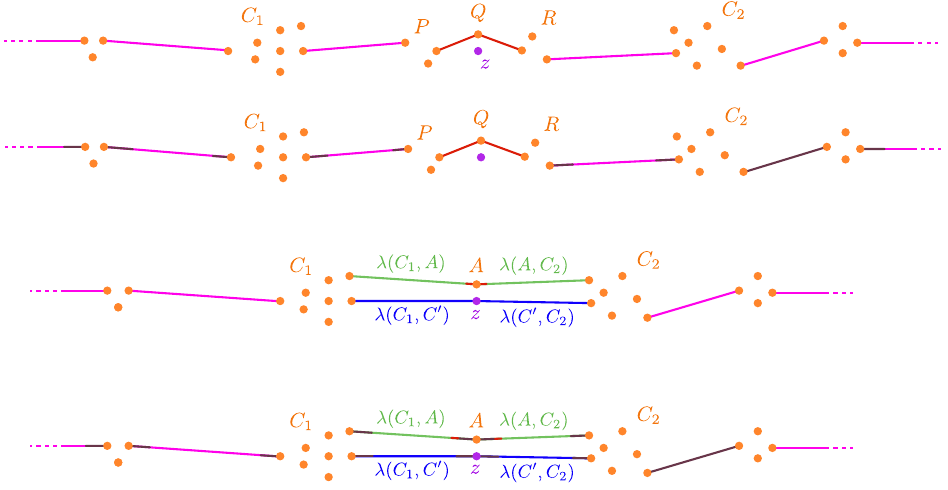}
    \caption{Step (1) of the proof of Theorem \ref{thm:stable intervals}, affected cluster case.  Two phenomena are shown.  In the top figure, there are three affected clusters, $P,Q,R$, for $(a,b;\calY)$ which combine with $z$ into a new cluster $C'$ for $(a,b;\calY \cup \{z\})$.  By definition, $P,Q,R$ must be $O(\epsilon)$ apart, and actually are adjacent in $\calG$.  The clusters $C_1,C_2$ appear in both setups, and the subgraph of $\calG$ to the left of $C_1$ coincides with the same subgraph for $\calG'$, and similarly for $C_2$.  So the corresponding (pink) components form (stable) components of $T_s \subset T_e$ and $T'_s \subset T'_e$.  The components of $T_e$ connecting $P,Q,R$ are declared unstable and, in this example, the components $\lambda(C_1,P)$ and $\lambda(R,C_2)$ of $T_e$ coincide with the components $\lambda(C_1,C')$ and $\lambda(C',C_2)$ of $T'_e$, and hence are stable.  The second figure shows the stable decomposition for the $(r_1,r_2)$-thickenings, which are just the thickenings of the components of $T_s = T'_s$, and hence are also exactly the same.  In the second example, the same discussion applies to the outside parts.  For the middle of the graph, $z$ combines with the cluster $A$, which results in pairs of segments $\lambda(C_1,A)$ and $\lambda(C_1, C')$, and $\lambda(A,C_2)$ and $\lambda(C',C_2)$, which have $O(\epsilon)$-close endpoints and are thus $O(\epsilon)$-Hausdorff close, with the same length up to error $O(\epsilon)$.  Small surgeries provide the required stable components for the stable intervals $T, T'$, shown in the third figure.  The fourth figure shows the stable decomposition for the $(R,r)$-thickenings of this example.}
    \label{fig:affected cluster}
\end{figure}

\begin{figure}[t]
    \centering
    \includegraphics[width=1\textwidth]{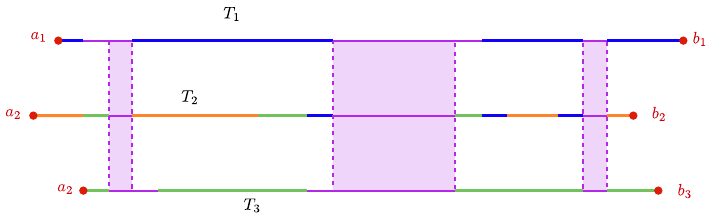}
    \caption{The proof of Theorem \ref{thm:stable intervals}, iterating stable decompositions: The idea for the iteration argument is to intersect the stable decompositions $T^1_{s,2} \cap T^3_{s,2} \subset T_2$ for $T_1,T_3$, respectively, and then use the isometries which identify the stable pairs for $T_1, T_2$ to give a refined stable decomposition for $(T_1,T_2)$, which can be combined with a refinement for $(T_2,T_3)$ via simple composition.  In the schematic (which is \emph{not} happening in $\calZ$), we have lined up the stable intervals $T_1,T_2,T_3$ to indicate how to refine the pairwise stable decompositions for $(T_1,T_2)$ and $(T_2,T_3)$ for one for $(T_1,T_3)$.  The orange segments correspond to cluster components of the intervals, while the blue and green segments correspond to the unstable components for the pairs $(T_1,T_2)$ and $(T_2,T_3)$, respectively.  The complement of the orange, blue, and green segments on $T_2$ form the refined stable decomposition on $T_2$, and the purple squares indicate how these segments induce refined stable decompositions on $T_1$ and $T_3$.}
    \label{fig:iterated stable decomp}
\end{figure}

\newpage
\section{Stable cubical models} \label{sec:stable cubes}

In this section, we explain how to combine the cubulation machinery described in Sections \ref{sec:interval systems}--\ref{sec:cubulation} with the Controlling Domains Theorem \ref{thm:controlling domains} established in Section \ref{sec:controlling domains} and the Stable Interval Theorem \ref{thm:stable intervals} from Section \ref{sec:stable intervals}.  Over the first two subsections, we prove that our stable intervals can be plugged into our interval system machinery to produce an HFI for any pair of points $a,b \in \calX$ in an HHS.

The heart of the section is then explaining how to relate two such HFIs for a pair of points $a,b$ and a perturbed pair $a',b'$.  We explain this below in detail in Subsection \ref{subsec:perturbed to stable}.

\subsection{Stable intervals form an interval system}

Our first goal is to explain how to take the family of stable intervals associated to a pair of points $a, b\in \calX$ and produce a hierarchical family of intervals.

First, we need to induce a collection of relative projections on the stable intervals which turn them into an $R$-interval system (Definition \ref{defn:interval axioms}) for a uniform $R = R(\mathfrak S)>0$.

Let $U \in \calU$.  Recall that the BGIA \ref{ax:BGIA} says that if $d_V(a,b)$ is sufficiently large for $V \in \mathfrak S$, then $\rho^V_U \subset \calN_{\ES}(\lambda(a,b))$ when $V \nest U$ and, even better, $\rho^V_U \subset \calN_{\ES}(a \cup b)$ when $V \pitchfork U$.  If we set $\calY_U= \{\rho^V_U|V \nest U \in \calU\}$, then $(\pi_U(a), \pi_U(b); \calY_U)$ is an $\epsilon$-setup in $\calC(U)$ for $\epsilon = 2\ES$ in the sense of Subsection \ref{subsec:stable setup}.  In particular, our stable interval construction (Proposition \ref{prop:stable interval structure}) applies.

For each $U \in \calU$, let $T_U = T_{U,e} \cup T_{U,c}$ denote the stable interval for the setup $(a,b;\calY_U)$ in $\calC(U)$.  Recall that each component of $T_{U,c}$ coincides with $\mu(C)$ for some cluster $C$ of points in $\calY_U \cup \{a,b\}$ (Subsection \ref{subsec:clusters}).  Importantly, item (2) of Proposition \ref{prop:stable interval structure} says that $\phi_U(\mu(C))$ is $O(\epsilon)$ (and hence $O(\mathfrak S)$) Hausdorff close to $C$ itself.  This property is what makes the following definition work:

\begin{definition}[Stable interval projections]\label{defn:stable interval projections}
For each $U \in \calU$, we define the following \emph{stable interval projections}:
\begin{itemize}
\item For each $V \in \calU$ with $V \nest U$, let $C$ be the cluster containing $\rho^V_U$.  Define 
$$\delta^V_U = \phi_U^{-1}(\hull_{\mu(C)}(p_{\phi_U(\mu(C))}(\rho^V_U))) \subset \mu(C),$$where $p_U:\calC(U) \to \phi_U(\mu(C))$ is the closest point projection to the geodesic segment $\phi_U(\mu(C)) = \mu(C)$.
\item For each $V \in \calU$ with $V \pitchfork U$, let $f \in \{a,b\}$ denote the point (provided by Definition \ref{defn:consistency}) for which $d_U(\pi_U(f), \rho^V_U)<\epsilon$, and let $C$ denote the cluster containing $f$.  Again we define $\delta^V_U = \phi_U^{-1}(\hull_{\mu(C)}(p_{\phi_U(\mu(C)}(\rho^V_U))) \subset \mu(C)$.
\item If $U \nest V$, then we define the \emph{interval projection map} $\delta^V_U:T_V \to T_U$ by $\delta^V_U = \phi^{-1}_U \circ p_{T_U} \circ \rho^V_U \circ \phi_V.$
\end{itemize}
\end{definition}

The next lemma says that these projections turn the family of stable intervals into an interval system.  The proof, which we leave to the reader, involves verifying the conditions of Definition \ref{defn:interval axioms} (and hence Lemma \ref{lem:interval control}, and one does this via Lemma \ref{lem:rho control} and Proposition \ref{prop:stable interval structure}.

\begin{lemma}\label{lem:stable interval system}
There exists $R_s = R_s(\mathfrak S)>0$ so that the stable intervals $\{T_U\}_{U \in \calU}$ endowed with the stable interval projections forms an $R_s$-interval system for $a,b$.
\end{lemma}

\subsection{Thickenings of stable intervals, their collapses, and HFIs}\label{subsec:collapsed stable intervals}

In this subsection, we explain why collapsed thickened interval systems of stable intevals form HFIs.

To begin, in Subsection \ref{subsec:thickenings}, we took an interval system $\{T_U\}_{U \in \calU}$ along with its relative projection data $\delta^V_U$, and then thickened the intervals along the components $B$ of a decomposition $T_U = A \cup B$, where $B$ consists of these $\delta^V_U$ along with the marked points $a_U,b_U$.  Then in Subsection \ref{subsec:collapsing}, we collapsed along this thickening which induced collapsed relative projections on the collapsed intervals.  On the other hand, the thickenings of stable intervals we considered in the Stable Intervals Theorem \ref{thm:stable intervals} were thickenings along the cluster components.

The following lemma, whose proof we leave to the reader, says that these thickenings coincide, and hence that the corresponding collapsed trees are identical:

\begin{lemma}\label{lem:same thickenings}
For any $U \in \calU$, let $T_U = A \cup B$ be the decomposition where $B$ is the union of all $\delta^V_U$ for $V \nest U$ along with $a_U,b_U$.  Let $T_U = T_{U,e} \cup T_{U,c}$ denote the decomposition provided by Proposition \ref{prop:stable interval structure}.  There exists $R_3 = R_3(\mathfrak S)>0$ so that for any $r_1>R_3$ and $r_2>0$, the $(r_1,r_2)$-thickenings of $T_U$ along $B$ and along $T_{U,c}$ are identical.
\end{lemma}

The upshot of Lemmas \ref{lem:stable interval system} and \ref{lem:same thickenings} is that we can plug our stable intervals directly into our interval system machinery.  The next step is to explain how to produce an HFI (Definition \ref{defn:HFI}) from this system using the thickenings and collapsing procedures discussed in Subsections \ref{subsec:thickenings} and \ref{subsec:collapsing}.

Continuing with our notation, Lemma \ref{lem:stable interval system} says that the stable intervals $\{T_U\}_{U \in \calU}$ endowed with the interval projections via Definition \ref{defn:stable interval projections} form an $R_s$-interval system for $a,b\in \calX$, where $R_s = R_s(\mathfrak S)>0$.  Since $R_s= R_s(\mathfrak S)>0$, Proposition \ref{prop:hPsi defined} and Lemma \ref{lem:collapsed interval control} provide constants $R_1 = R_1(\mathfrak S)>0$ and $R_2(\mathfrak S)>0$ so that for any $r_1\geq R_1$ and $r_2 \geq R_2$, the conclusions of Lemma \ref{lem:collapsed interval control} hold for the family of collapsed intervals $\{\hT_U\}_{U \in \calU}$ obtained from the $(r_1,r_2)$-thickening of the system $\{T_U\}_{U \in \calU}$.  We also want to assume that $r_1>R_3 = R_3(\mathfrak S)>0$ so that Lemma \ref{lem:same thickenings} holds.

The next proposition states that as long as $r_1$ is a positive integer, the collapsed intervals $\hT_U$ are all simplicial intervals, where the collapsed points lie at vertices.  The proof is a straightforward application of the ideas in the preceding sections, so we leave it to the reader.

\begin{proposition}\label{prop:stable HFI}
As long as $r_1\geq \max\{R_1,R_3\}$ is a positive integer and $r_2\geq R_2$, the family of collapsed stable intervals $\{\hT_U\}_{U \in \calU}$ associated to an $(r_1,r_2)$-thickening of $\{T_U\}_{U \in \calU}$ along $T_{U,c}$ and endowed with their collapsed projections forms a hierarchical family of intervals.
\end{proposition}

We note the following immediate corollary of Proposition \ref{prop:stable HFI} and Theorem \ref{thm:Q CCC}:

\begin{corollary}\label{cor:stable HFI cube}
Let $a,b \in \calX$, $K > K_0$ for $K_0$ as in Theorem \ref{thm:DF} and $\calU = \Rel_K(a,b)$.  Let $\{\hT_U\}_{U \in \calU}$ be the associated family of stable intervals.  Then the hierarchical hull $H = \hull_{\calX}(a,b)$ is quasi-isometric to the dual cube complex of the associated HFI, with constants depending only on $\mathfrak S$ and $K$.
\end{corollary}

\subsection{From a perturbed pair to a family of stable decompositions} \label{subsec:perturbed to stable}

In this subsection, we explain how the basic setup of this paper can be plugged into the machinery we have built so far.

For the rest of this section, assume that we have a colorable HHS $(\calX, \mathfrak S)$ endowed with the stable projections from Theorem \ref{thm:controlling domains}.  Let $K > K_0=K_0(\mathfrak S)>0$ be as in item (2) of that theorem.

Also for the rest of this section, fix $a,b,a',b' \in \calX$ so that $d_{\calX}(a,a'), d_{\calX}(b,b') \leq 1$.  We refer to the points $a',b'$ as a \emph{perturbed pair} from $a,b$.  Let $\calU = \Rel_{K}(a,b)$ and $\calU' = \Rel_{K}(a',b')$.  Let $\{T_U\}_{U \in \calU}$ and $\{T'_U\}_{U \in \calU'}$ denote the families of stable intervals produced as above. 

The first thing to observe is that item (2)(a) of Controlling Domains Theorem \ref{thm:controlling domains} provides that 
\begin{enumerate}
\item $|\calU \triangle \calU'|<N_0 = N_0(\mathfrak S)$, and
\item Each domain $V \in \calU \triangle \calU'$ satisfies $d_V(a,b) \leq K'$ and $d_V(a',b') \leq K'$, where $|K-K'|$ is bounded in terms of $\mathfrak S$.
\end{enumerate}

Item (2) plus Lemma \ref{lem:collapsed interval control} says that each of the associated collapsed intervals $\hT_V$ or $\hT'_V$ for $V\in \calU \triangle \calU'$ has bounded diameter.  For the rest of the domains, namely the domains $U \in \calU_0 = \calU \cap \calU'$, we want to obtain compatible stable decompositions between the intervals $T_U$ and $T'_U$, which we do as follows.

First, observe that if $V \nest U \in \calU_0$, then $d_V(a,b)\geq K>\ES$ and $d_V(a',b') \geq K> \ES$, so the BGIA \ref{ax:BGIA} implies that $\rho^V_U \subset \calN_{\ES}(\lambda_U(a,b)) \cap \calN_{\ES}(\lambda_U(a',b'))$, where $\lambda_U$ is the function which assigns a geodesic in $\calC(U)$ to any pair of points from Subsection \ref{subsec:stable setup}.  Hence taking $\epsilon = 2\ES$ in our stable tree setup suffices to make $(\pi_U(a),\pi_U(b);\calY_U)$ and $(\pi_U(a'),\pi_U(b');\calY'_U)$ into $\epsilon$-setups in $\calC(U)$ (item (1) of Definition \ref{defn:stable decomp}).

Moreover, for each $U \in \calU_0$, Theorem \ref{thm:controlling domains} provides that 
\begin{itemize}
\item[(a)] The $\rho$-sets for the domains in $\calU$ and $\calU'$ which nest into $U$---namely $\calY_U$ and $\calY'_U$ respectively---satisfy $|\calY_U \triangle \calY'_U| < N_0 = N_0(\mathfrak S)$, and
\item[(b)] If $\calU^*_0 \subset \calU_0$ is the set of domains $U \in \calU_0$ for which $\calY_U \neq \calY'_U$ or $\pi_U(a) \neq \pi_U(a')$ or $\pi_U(b) \neq \pi_U(b')$, then $\#\calU^*_0 < N_0$.
\end{itemize}

Hence, by item (a), these $\epsilon$-setups are $(N_0, \epsilon)$-admissible in the sense of item (2) of Definition \ref{defn:stable decomp} for all domains in $\calU_0$, while item (b) says that for domains in $\calU_0 - \calU^*_0$, these $\epsilon$-setups exactly coincide.  Thus item (2) of the Stable Intervals Theorem \ref{thm:stable intervals} implies that the $(r_1,r_2)$-thickenings of their stable intervals $T_U, T'_U$ admit $L_s$-compatible stable decompositions, where $L_s=L_s(\mathfrak S)>0$ by choosing $r_1 = \max\{R_1,R_3\}$ and $r_2 = R_2$, since the constants $R_i = R_i(\mathfrak S)>0$ depend only on $\mathfrak S$.  In particular, we have:

\begin{proposition}\label{prop:stable interval decomp}
For each $U \in \calU_0$, the $(r_1,r_2)$-thickenings of the stable intervals $T_U = T_{U,e} \cup T_{U,c}$ and $T'_U = T'_{U,e} \cup T'_{U,c}$ along $T_{U,c}, T'_{U,c}$ respectively admit $L_s$-compatible stable decompositions, where $L_s = L_s(\mathfrak S)>0$.
\begin{itemize}
\item Moreover, if $U \in \calU_0 - \calU^*_0$, then $T_U = T'_U$ and these thickenings coincide exactly.
\end{itemize}
\end{proposition}

For each $U \in \calU_0$, let $\Delta_U:\hT_U \to \hT_{U,0}$ and $\Delta'_U:\hT'_U \to \hT'_{U,0}$ denote the maps which collapse the unstable components these decompositions to points; we note that these maps are the identity when $U \in \calU_0- \calU^*_0$.  By Definition \ref{defn:stable decomp} of a stable decomposition, the subintervals of the $\hT_U$ and $\hT'_U$ being collapsed for $U \in \calU_0$ along with the whole intervals $\hT_V, \hT'_V$ (whichever is defined) for $V \in \calU \triangle \calU'$ collectively form $(L_s^2+2K'N_0)$-admissible simplicial trimming setups for $\{\hT_U\}_{U \in \calU}$ and $\{\hT'_U\}_{U \in \calU'}$, in the sense of Definition \ref{defn:simplicial trimming setup}.  This is because the set $(\calU \cup \calU') - \calU_0$ has at most $N_0$-many domains by (1) above, each of which has diameter bounded by $K'$, while the unstable pieces for the domains in $\calU_0$ consist of at most $L_2$-many pieces of diameter at most $L_2$.

Set $L_2^2 + 2K'N_0 = S_0$, which we note depends only on $\mathfrak S$ and our choice of $K$.  We now obtain the following from Theorem \ref{thm:simplicial tree trimming} and Proposition \ref{prop:stable interval decomp}:

\begin{proposition} \label{prop:stable collapse}
The following hold:
\begin{enumerate}
\item The collections $\{\hT_{U,0}\}_{U \in \calU_0}$ and $\{\hT'_{U,0}\}_{U \in \calU_0}$ are HFIs.  Let $\calQ_0, \calQ'_0$ denote their corresponding $0$-consistent subsets.
\item The combined maps $\Delta:\calQ \to \calQ_0$ and $\Delta':\calQ' \to \calQ'_0$ satisfy $\Delta(\calQ) = \calQ_0$ and $\Delta'(\calQ') = \calQ'_0$ and are both $(1, S_0^2)$-quasi-isometries.
\item There exists maps $\Xi:\calQ_0 \to \calQ$ and $\Xi':\calQ'_0 \to \calQ'$ with $\Delta \circ \Xi = id_{\calQ_0}$ and $\Delta' \circ \Xi' = id_{\calQ'_0}$.
\item At the level of cube complexes, the maps $\Delta\calQ \to \calQ_0$ and $\Delta':\calQ' \to \calQ'_0$ delete at most $S_0^2$-many hyperplanes.
\end{enumerate}
\end{proposition}

Using this proposition, we can define maps $\hO_0:\calQ_0 \to \calX$ and $\hO'_0:\calQ'_0 \to \calX$ by $\hO_0 = \hO \circ \Xi$ and $\hO'_0 = \hO' \circ \Xi'$, where $\hO:\calQ \to \calX$ and $\hO':\calQ \to \calX$ are the quasi-isometric embeddings defined Subsection \ref{subsec:upgrade to HHS}.

For the final piece of setup, we want to observe that these HFIs $\{\hT_{U,0}\}_{U \in \calU_0}$ and $\{\hT'_{U,0}\}_{U \in \calU_0}$ admit an isomorphism in the sense of Proposition \ref{prop:HFI isomorphism}.  For this, we require the following lemma, whose proof is a straightforward application of the definitions and Corollary \ref{cor:interval isometry}.  

\begin{lemma}\label{lem:interval isometries}
For each $U \in \calU^*_0$, the isometry $i_U:\hT_{U,0} \to \hT'_{U,0}$ provided by Corollary \ref{cor:interval isometry} satisfies the conditions of Proposition \ref{prop:HFI isomorphism}.
\begin{itemize}
\item Thus these component-wise isometries combine to give an isometry $I_0:\calQ_0 \to \calQ'_0$. 
\end{itemize}

\end{lemma}

\subsection{Statement of the stable cubulations theorem}

We are now ready to state the stable cubical intervals theorem:

\begin{theorem}[Stable cubical intervals]\label{thm:stable cubes}
Continuing with our standing notation, there exists $\Sigma = \Sigma(\mathfrak S, K)>0$ so that the following diagram commutes up to error $\Sigma$:
\begin{equation}\label{Phi diagram}
  \begin{tikzcd}
   \calQ \arrow[ddrr,"\hO", bend left=40] \arrow[dr,"\Delta_0 \hspace{.075in}" left] &  \\
    &\calQ_{0} \arrow[dr,"\hO_0 \hspace{.2in}" below]\arrow[dd, "I_0"] \\
    & & \calX\\
    & \calQ'_{0} \arrow[ur,"\hO'_0"] \\
    \calQ'\arrow[uurr,"\hspace{.2in} \vspace{.1in} \hO'" below, bend right=40] \arrow[ur,"\Delta'_0"] & \\
  \end{tikzcd}
  \end{equation}

\begin{itemize}
\item In the above, $I_0:\calQ_0 \to \calQ'_0$ is the isomorphism from Lemma \ref{lem:interval isometries}.
\item In particular, any colorable hierarchically hyperbolic space is $\Sigma$-stably weakly locally quasi-cubical (Definition \ref{defn:stable cubes}). 
\end{itemize}
\end{theorem}

\begin{proof}

It remains to prove that the upper, lower, and middle triangles coarsely commute.  The proof for the upper (and hence the lower) is essentially automatic by construction.  In particular, by item (2) of Proposition \ref{prop:stable collapse}, the map $\Delta_0:\calQ \to \calQ_0$ is a $(1,S^2_0)$-quasi-isometry for $S_0 = S_0(\mathfrak S, K)>0$.  On the other hand, item (3) of that proposition implies that $\Delta_0 \circ \Xi = id_{\calQ_0}$ while $\hO_0:\calQ_0 \to \calX$ is defined by $\hO_0 = \hO \circ \Xi$.  Hence the top triangle commutes up to error $S_0$, and the same holds for the bottom triangle.

For the middle triangle, we want to show that $d_{\calX}(\hO_0(\hx), \hO'_0 \circ I_0(\hx))$ is bounded in terms of $\mathfrak S, K$ for any $\hx \in \calQ_0$.  However, we note that the maps $\hO_0:\calQ_0 \to H$ and $\Delta_0 \circ \hPsi:H \to \calQ_0$ are quasi-inverses (with errors depending on $\mathfrak S, K$), and similarly for $\hO'_0:\calQ'_0 \to H'$ and $\Delta_0 \circ \hPsi:H' \to \calQ'_0$.  Here, $H = \hull_{\calX}(a,b)$ and $H' = \hull_{\calX}(a',b')$.  But since $d_{\calX}(a,a')\leq 1$ and $d_{\calX}(b,b')\leq 1$, it follows from the Distance Formula \ref{thm:DF} and Definition \ref{defn:hier hull} that $d^{Haus}_{\calX}(H,H')$ is bounded in terms of $\mathfrak S$.

Hence it suffices to show that for $x \in H \cap H'$ (so that $\hPsi(x)$ and $\hPsi'(x)$ are both defined), we have
$$d_{\calQ'_0}(\Delta'_0 \circ \hPsi'(x), I_0 \circ \Delta_0 \circ \hPsi(x))< \Sigma,$$
for $\Sigma = \Sigma(\mathfrak S, K)$.  Set $\Delta'_0 \circ \hPsi'(x) = \hx'_0 = (\hx'_{U,0})$ and $I_0 \circ \Delta_0 \circ \hPsi(x) = \hy'_0= (\hy'_{U,0})$, where we note that both $\hx'_0,\hy'_0 \in \calQ'_0$.  

We argue component-wise that $d_{\calQ'_0}(\hx'_0,\hy'_0)$ is bounded.  Recall that the map $\hPsi:H \to \calQ$ is defined component-wise (over the set $\calU$) by $\hpsi_U = q_U \circ \phi_U^{-1} \circ p_U \circ \pi_U$, where $\pi_U:\calX \to \calC(U)$ is the HHS projection, $p_U:\calC(U) \to \phi_U(T_U)$ is closest point projection, $\phi_U:T_U \to \calC(U)$ is the $(L_0,L_0)$-quasi-isometric embedding provided by Proposition \ref{prop:stable interval structure} (here $L_0=L_0(\mathfrak S)$), and $q_U:T_U \to \hT_U$ is the map which collapses a uniform thickening of $T_U$ along the cluster components, with constants depending only on $\mathfrak S$.  The map $\hPsi':H' \to \calQ'$ is defined similarly.

First, we can reduce to involved domains in $\calU_0 = \calU \cap \calU'$.  If $U \in \calU \triangle \calU'$, then $a,b$ and $a',b'$ have bounded diameter projections to $U$ in terms of $K$ and $\mathfrak S$, while also each such domain is involved in the sense of Definition \ref{defn:involved} and hence there are only boundedly-many by Theorem \ref{thm:controlling domains}. Next observe that if $U \in \calU_0$ is not involved, then the stable trees $T_U, T'_U$ for the two setups are exactly the same, and hence have identical thickenings, and thus the stable decompositions the pair are the (identical) collection of edge components for each, and their corresponding collapsed trees $\hT_U$ and $\hT'_U$ are identical.  Hence, for uninvolved $U \in \calU_0$, we have $\hx'_{U,0} = \hy'_{U,0}$.

Let $U \in \calU_0$ be involved.  Since there are boundedly-many (in $\mathfrak S$) such domains by Theorem \ref{thm:controlling domains}, it suffices to bound the diameter $d_{\hT'_{U,0}}(\hx'_{U,0}, \hy'_{U,0})$ in terms of $\mathfrak S$.

  The distance $d_{\hT'_{U,0}}(\hx'_{U,0}, \hy'_{U,0})$ is measured in the edge components of $\hT'_{U,0}$ which separate them, and these edge components are stable components of $T'_U$.  This means that if the distance $d_{\hT'_{U,0}}(\hx'_{U,0}, \hy'_{U,0})$ is very large, then there must be either 
\begin{enumerate}
\item some long stable component $E' \subset T'_{U,s}$ so that some long segment of $\Delta'_U(q'_U(E'))$ separates $\hx'_{U,0}$ from $\hy'_{U,0}$, or
\item if no such long component exists, there must be some large number of such stable components separating them.
\end{enumerate}

In case (1), let $E \subset T_{U,s}$ be the stable component paired with $E' \subset T'_{U,s}$.  Note that $\phi_U(E) \cup \phi'_U(E')$ is uniformly quasi-convex (in $\mathfrak S$), and hence $\pi_U(x)$ has a uniformly bounded projection to it.  However the fact that a long segment of $\Delta'_U(q_U(E')) \subset \hT'_{U,0}$ separates $\hx'_{U,0}$ from $\hy'_{U,0}$ means that some large diameter portion of $\phi_U(E) \cup \phi'_U(E')$ coarsely (in $\mathfrak S$) separates $p_U(\pi_U(x))$ from $p'_U(\pi_U(x))$, which is impossible.

In case (2), we can assume that no such long stable component exists, so that the distance $d_{\hT'_{U,0}}(\hx'_{U,0}, \hy'_{U,0})$ corresponds to some sequence $E'_1, \dots, E'_n \subset T'_{U,s}$ of short stable components of length bounded in terms of $\mathfrak S$.  Let $E_i \subset T_{U,s}$ denote the stable component paired with $E'_i$.  By Proposition \ref{prop:stable interval structure} and the fact that $d_{\calX}(a,a') \leq 1$ and $d_{\calX}(b,b')\leq 1$, we have that $\phi_U(T_U)$ and $\phi'_U(T'_U)$ are within bounded (in $\mathfrak S$) Hausdorff distance in $\calC(U)$.  Since each of the $E'_i$ are uniformly close (in $\mathfrak S$) to $E_i$ and the bijection $\alpha:\pi_0(T_{U,s}) \to \pi_0(T'_{U,s})$ is order-preserving, it follows that the hulls $\hull_{\phi_U(T_U)}(E_1, E_n)$ and $\hull_{\phi'_U(T'_U)}(E'_1,E'_n)$ are within uniform (in $\mathfrak S$) Hausdorff distance of each other, and hence their union is uniformly quasi-convex (in $\mathfrak S$).  Thus, if $\Delta'_U(q'_U(E'_1)), \dots, \Delta'_U(q'_U(E'_n))$ separate $\hx'_{U,0}$ from $\hy'_{U,0}$ in $\hT'_{U,0}$, then the union of these hulls coarsely (in $\mathfrak S$) separates $p_U(\pi_U(x))$ from $p'_U(\pi_U(x))$.

Since each of the segments $E'_i$ is bounded length, the union of these hulls is long if the number of segments $n$ is very large.  Moreover, each pair of segments $E'_i, E'_{i+1}$ is separated by some complementary component $C_i \subset T'_U - T'_{U,s}$.  By construction, all but boundedly-many (in $\mathfrak S$) of these components corresponds to some cluster component of $T'_U$.  Since there are only boundedly-many (in $\mathfrak S$) involved domains, by choosing $n$ to be sufficiently large, there exists some $C_j$ which contains a $\nest_{\calU_0}$-minimal domain $V \nest U$ which is both $\nest_{\calU}$-minimal and $\nest_{\calU'}$-minimal.  But this is impossible, for the BGIA \ref{ax:BGIA} would imply that $d_V(\rho^U_V(p_U(\pi_U(x)), \rho^U_V(p'_U(\pi_U(x)))) \succ K$, and this cannot happen as we can also arrange $\pi_U(x)$ to be as far from $\rho^V_U$ as necessary.  Hence no such long sequence $E'_1, \dots E'_n$ can exist, completing the proof.

\end{proof}

\appendix
\section{Proof of the Tree Trimming Theorem \ref{thm:tree trimming}}\label{app:TT}

Below is a re-statement of Theorem \ref{thm:tree trimming}.  See Subsection \ref{subsec:tree trimming} for the contextual definitions and notation.

\begin{theorem}[Tree Trimming]
There exists $B_{tt} = B_{tt}(\mathfrak S, K)>0$ so that for any $B < B_{tt}$, there exists $L =L(\mathfrak S, B)>0$ so that the following hold:

\begin{enumerate}
\item We have $\Delta(\calQ) \subset \calQ'$.
\item There is a map $\Xi:\calQ' \to \calQ$ so that $\Delta \circ \Xi = id_{\calQ'}$.  In particular, $\Delta(\calQ) = \calQ'$.
\item The map $\Delta:\calQ \to \calQ'$ is an $(L,L)$-quasi-isometry.
\end{enumerate}
\end{theorem}

\begin{proof}
The fact that $\Delta(\calQ) \subset \calQ'$, namely that $\Delta$ preserves $0$-consistency, is immediate from the fact that the $\Delta_U$ just collapse subintervals, so coordinates which were collapsed projections remain collapsed projections, showing item (1).  Items (2) and (3) require more work and we provide a detailed sketch of both.
\medskip

\textbf{\underline{Item (2): Surjectivity of $\Delta:\calQ \to \calQ'$}}: Let $\hy \in \calQ'$.  We will construct a point $\hx \in \calQ$ with $\Delta(\hx) = \hy$ by a coordinate-wise argument, thereby defining a map $\Xi:\calQ' \to \calQ$ with $\Delta \circ \Xi = id_{\calQ'}$.

For each $U \in \calU$, let $C_U = \Delta^{-1}_U(\hy_U)$.  Our goal is to pick, for each $U \in \calU$, a point $\hx_U \in C_U$ and then prove $0$-consistency of the tuple $\hx = (\hx_U)$.  Note that $\Delta(\hx) =\hy$ for such a tuple, by construction.  The proof is by induction on the nesting level in $\calU$, and is similar to the construction of the map $\hO:\calQ \to H$ as in Proposition \ref{prop:coarse surjectivity}.  The main difference is that the collapsing subintervals are not automatically associated to a collection of domains, and this is where the work lies.  The proof is basically the proof of the analogous fact in \cite[Theorem 10.2]{Dur_infcube}, except that our setup is much simpler.  What follows is a very detailed sketch, which leaves some small bits to the reader.

For the base case, let $U \in \calU$ be $\nest_{\calU}$-minimal.   If $\hy_U \in \hT_{U,0}$ is not contained in $\Delta_U(B)$ for some $B \in \calB$, then $C_U\subset \hT_U$ is a point and we set $\hx_U = C_U$.  If $\hy_U \in \Delta_U(B)$ for some $B \in \calB$ and $B \subset \hT_U$ does not contain an endpoint, then let $\hx_U \in C_U$ be any point.  If instead $C_U$ does contain an endpoint, then it is possible that it contains both endpoints, and we need to choose one.

The following claim, which is exactly \cite[Claim 10.3]{Dur_infcube}, does not require $U$ to be $\nest_{\calU}$-minimal:

\begin{claim} \label{claim:choose one}
If $W, Z \in \{V \in \calU| V\pitchfork U, \hspace{.1in} \text{and} \hspace{.1in} \hy_V \neq (\hd^U_V)_0\}$, then $\hd^W_U = \hd^Z_U$.
\end{claim}

With Claim \ref{claim:choose one} in hand, we can define $\hx_U$ for $\nest_{\calU}$-minimal $U$ as follows:

\begin{itemize}
\item If $\calW_U^{\pitchfork} = \{V \in \calU| V \pitchfork U \hspace{.1in} \mathrm{ and } \hspace{.1in} \hy_W \neq (\hd^V_U)_0\} \neq \emptyset$, then we set $\hx_U = \hd^W_U$ for any (and hence all, by the claim) $W \in \calW_U^{\pitchfork}$.
\item If $\calW_U^{\pitchfork} = \emptyset$, then we let $\hx_U$ be any marked point in $C_U$.
\end{itemize}

At this point, we also want to set $l_U$ to be equal to the endpoint $\hx_U$ when $U$ is $\nest_{\calU}$-minimal, where $l_U = \emptyset$ when $C_U$ does not contain an endpoint of $\hT_U$.  These labels will play a similar role as in the cluster honing process in Subsection \ref{subsec:hO defined}.

For the inductive step, suppose that $U$ is not $\nest_{\calU}$-minimal and that we have defined $\hx_V$ for all $V \nest U$ and, in particular, for all $\nest_{\calU}$-minimal $V \in \calU$.  As before, if $C_U$ does not contain a marked or a projection $\hd^V_U$ for some $V \nest U$, then we let $\hx_U \in C_U$ be any point.  If $C_U$ happens to contain one of the endpoints but no $\hd^V_U$ for $V \nest U$, then $C_U$ contains only that one endpoint (since $U$ is not $\nest_{\calU}$-minimal), and we let $\hx_U \in C_U$ be any point

Otherwise, we set $C=C_U$ and let $\calV_U$ denote the set of $\nest_{\calU}$-minimal domains $V \nest U$ so that $\hd^V_U \in C$ (or equivalently, $(\hd^V_U)_0 = \hy_U$).  For each $V \in \calV_U$, set $C^V_U = \hull_C(\hd^V_U \cup \pi_C(\hx_V))$, namely the convex hull in $C$ of $\hd^V_U$ and the projections of the marked point labeled by $l_V$ down to the convex subset $C$.  The following claim, which is exactly \cite[Claim 10.4]{Dur_infcube}, helps us choose a point in such a $C_U$:

\begin{claim}\label{claim:cluster overlap}
For every $V, W \in \calV_U$, we have $C^V_U \cap C^W_U \neq \emptyset$.  Moreover, at least one of $\hd^V_U \in C^W_U$ or $\hd^W_U \in C^V_U$ holds.

\end{claim}

\begin{proof}
We may assume that $\hd^V_U \neq \hd^W_U$, and hence that $V \pitchfork W$.  If $C^V_U \cap C^W_U = \emptyset$, then $\hd^V_U \notin C^W_U$ and $\hd^W_U \notin C^V_U$, and it follows that $l_V \neq l_W$.  The BGI property in (7) of Lemma \ref{lem:collapsed interval control} plus item (4) of Lemma \ref{lem:interval control} imply that  $\hd^W_V \notin \Delta^{-1}_V(\hy_V)$ and $\hd^V_W \notin \Delta^{-1}_W(\hy_W)$, and hence $\hy_V \neq (\hd^W_V)_0$ and $\hy_W \neq (\hd^V_W)_0$, which is a contradiction of $0$-consistency of $\hy$.  This proves the claim.
\end{proof}

Hence the Helly property for trees implies that $\bigcap_{V \in \calV_U} C^V_U \subset \hT_U$ is nonempty.  This intersection, nonetheless, may contain many $\hd^V_U$ for $V \in \calV_U$.  The following claim, which is exactly \cite[Claim 10.6]{Dur_infcube} helps us make a consistent choice:

\begin{claim}\label{claim:choose two}
Suppose $V, W \in \calW_U^{\nest}$, then $\hd^V_U = \hd^W_U$.
\end{claim}

\begin{proof}[Proof of Claim \ref{claim:choose two}]
For a contradiction, suppose that $\hd^W_U \neq \hd^V_U$, so that $V \pitchfork W$.  Then again the BGI property in item (7) of Lemma \ref{lem:collapsed interval control} implies that $\hd^V_W = \hd^U_W(\hd^V_U)$ and $\hd^W_V = \hd^U_V(\hd^W_U)$.  But then $\hy_W \neq (\hd^V_W)_0 = (\hd^U_W)_0((\hd^V_U)_0)$ and $\hy_V \neq (\hd^W_V)_0 = (\hd^U_V)_0((\hd^W_U)_0)$, which contradicts $0$-consistency of $\hy$, completing the proof.
\end{proof}

Finally, this last claim helps us make a consistent choice between $\hd^V_U$ sets and marked points when $C_U$ contains both:

\begin{claim}\label{claim:good choice}
The following holds:
\begin{enumerate}
\item If there exists $V \pitchfork U$ so that $\hy_V \neq (\hd^U_V)'$, then for all $W \in \calV_U$ with $\hd^W_U \neq \hd^V_U$, we have $\hx_W = \hd^V_W$.
\item If there exists $W \in \calV_U$ with $\hy_W \neq (\hd^U_W)'(\hy_U)$, then for all $V \pitchfork U$ with $\hd^W_U \neq \hd^V_U$, we have $\hy_V = (\hd^U_V)'$.
\end{enumerate}
\begin{itemize}
\item In particular, if $V \in \calW_U^{\pitchfork}$ and $W \in \calW_U^{\nest}$, then $\hd^V_U = \hd^W_U$.
\end{itemize}
\end{claim}

\begin{proof}[Proof of Claim \ref{claim:good choice}]
Suppose $V \pitchfork U$ as in item (1) of the claim, and let $W \in \calV_U$ with $\hd^W_U \neq \hd^V_U$.  Then $W \pitchfork V$, and since $W \nest U$, we have $(\hd^W_V)_0 = (\hd^U_V)_0 \neq \hy_V$, and hence $\hx_W = \hd^V_W$, by definition of $\hx_W$.

Now suppose $W \in \calV_U$, as in item (2), and that $V \pitchfork U$ with $\hd^W_U \neq \hd^V_U$. Then $W \pitchfork V$.  By $0$-consistency of $\hy$, we have that $\hy_U = (\hd^W_U)_0$.  But since $\hd^V_U \neq \hd^W_U$, we have $(\hd^V_U)_0 \neq (\hd^W_U)_0$, and so $\hy_V = (\hd^U_V)_0 = (\hd^W_V)_0$ since $W \nest U$.  This completes the proof.
\end{proof}

We are finally ready to define our tuple $\hx$, which we do domain-wise for each $U \in \calU$:

\begin{enumerate}
\item If $\hy_U$ is not a marked or cluster point of $\hT_{U,0}$, then we let $\hx_U \in C_U$ be any point in $\Delta^{-1}_U(\hy_U)$.
\item For any $U$ of type (1): If $V \pitchfork U$ or $U \nest V$, we set $\hx_V = \hd^U_V$, and if $U \nest V$, we set $\hx_V = \hd^U_V(\hx_U)$.
\item If $U$ is $\nest_{\calU}$-minimal and $\hy_U$ is a marked point of $\hT_{U,0}$, then there are three subcases: \label{item:minimal choice}
\begin{enumerate}
\item The set $\calW_U^{\pitchfork} = \{W \in \calU| W \pitchfork U \hspace{.1in}  \textrm{and } \hspace{.1in} (\hd^U_W)' \neq \hy_W\}$ is nonempty, and we set $\hx_U = \hd^W_U$ for any (and hence all by Claim \ref{claim:choose one}) $W \in \calW_U^{\pitchfork}$;
\item $\calW_U^{\nest} = \{V \nest U| (\hd^U_V)'(\hy_U) \neq \hy_V \hspace{.1in} \textrm{and} \hspace{.1in} (\hd^U_W)'(\hy_U) \neq \hy_W\}$ is nonempty, and we set $\hx_U = \hd^V_U$ for any (and hence all by Claim \ref{claim:choose two}) $V \in \calW_U^{\nest}$; or 
\item $\calW_U^{\pitchfork}$ and $\calW_U^{\nest}$ are both empty, and we set $\hx_U$ to be any point of $\hT^e_U$ contained in $C_U$, in particular a point which is not a cluster or marked point. 
\end{enumerate}
\item Proceeding inductively, if $U$ is not in cases (1)--(3) above, then $\bigcap_{V \in \calV_U} C^V_U \neq \emptyset$ and has finite diameter.  There are three subcases: \label{item:general choice}
\begin{enumerate}
\item $\hx_U = \hd^W_U$ for any $W \in \calW_U^{\pitchfork}$ if $ \calW_U^{\pitchfork} \neq \emptyset$, \label{item:choose transverse}
\item $\hx_U = \hd^V_U$ for any $V \in \calW_U^{\nest}$ when $\calW_U^{\nest} \neq \emptyset$, or \label{item:choose nested}
\item Otherwise, both $\calW_U^{\pitchfork}$ and $\calW_U^{\nest}$ are empty, and we set $\hx_U$ to be any point in $\hT^e_U \cap C_U$.
\end{enumerate}

\end{enumerate}

Finally, we prove that $\hx = (\hx_U)$ is $0$-consistent by a case-wise analysis, similar to the proof of Proposition \ref{prop:Q-consistent}.  Showing this will complete the proof of surjectivity of $\Delta: \calQ \to \calQ'$.  The proof closely follows the analogous part of the proof from \cite[Theorem 10.2]{Dur_infcube}.

\vspace{.1in}

\underline{$U \pitchfork V$:} We may assume that $\hx_V \neq \hd^U_V$, for we are done otherwise.  Then $\hy_V \neq (\hd^U_V)_0$, and so $\hy_U = (\hd^V_U)_0$ by $0$-consistency of $\hy$.  It follows that $\hd^V_U \in C_U = \Delta^{-1}_U(\hy_U)$, and necessarily $\hd^V_U  \in \{\ha_V, \hb_V\}$, so we may assume without loss of generality that $\hd^V_U = \ha_V$.  If $U$ is $\nest_{\calU}$-minimal, then our choice of $\hx_U$ in \eqref{item:minimal choice} says that $\hx_U = \hd^V_U$, as required.

Now suppose that $U$ is not $\nest_{\calU}$-minimal.  Set $\calV_U$ to be the set of $\nest_{\calU}$-minimal $W \nest U$ with $\hd^W_U \in C_U$.  Let $W \in \calV_U$ be so that $\hd^W_U \neq \hd^V_U$.  Then $W \pitchfork V$ and $(\hd^W_V)_0 = (\hd^U_V)_0 \neq \hy_V$, and so $\hx_W = \hd^V_W = \ha_W$, with the first equality following from our choice in \eqref{item:minimal choice} and the last equality by $0$-consistency of $a$.  But this says that $a = l_W$, and hence $\ha_U \in C^W_U$ for all $W \in \calV_U$ (including those for which $\hd^W_U = \hd^V_U = \ha_U$).

It follows then that $\ha_U \in \bigcap_{W \in \calV_U}C^W_U$ and so $\hx_U = \ha_U$ by definition of $\hx$ in item \eqref{item:choose transverse} above.

\vspace{.1in}

\underline{$V \nest U$:} We may assume that $\hx_V \neq \hd^U_V(\hx_U)$, for we are done otherwise.  Then $\hy_V \neq (\hd^U_V)_0(\hy_U)$, and $0$-consistency of $\hy$ implies that $\hy_U = (\hd^V_U)'$, and thus $\hd^V_U \in C_U = \Delta^{-1}_U(\hy_U)$.  We claim that for all $W \in \calV_U$, we have $\hd^V_U \in C^W_U$.

To see this, is suffices to consider $W \in \calV_U$ with $\hd^W_U \neq \hd^V_U$.  Then $W \pitchfork V$.  But $\hx_V \neq \hd^U_V(\hx_U)$, so if also $\hd^V_U \notin C^W_U$, then $\hx_V \neq \hd^U_V(\hd^W_U) = \hd^W_V$, while also $\hx_W \neq \hd^U_W(\hd^V_U) = \hd^V_W$, which both use the BGI property in item (7) of Lemma \ref{lem:collapsed interval control}.  But then $\hy_V \neq (\hd^W_V)_0$ and $\hy_W \neq (\hd^V_W)_0$, which contradicts $0$-consistency of $\hy$.

Hence $\hd^V_U \in C^W_U$ for all $W \in \calV_U$.  It follows that $\hd^V_U \in \bigcap_{W \in \calV_U} C^W_U$, and then \eqref{item:choose nested} says that $\hx_U = \hd^V_U$.

\vspace{.1in}

Hence $(\hx)$ is $0$-consistent.  Since $\Delta(\hx) = \hy$ because $\hx_U \in \Delta^{-1}_U(\hy_U)$ for all $U \in \calU$, we are done with the proof of surjectivity of $\Delta: \calQ \to \calQ'$, and this also defines the map $\Xi:\calQ' \to \calQ$ by $\Xi(\hy) = \hx$, as defined above.  This completes the proof of item (2).

\medskip

\textbf{\underline{Item (3): $\Delta:\calQ \to \calQ'$ is a quasi-isometry:}} The proof of this part proceeds similarly to its analogue in \cite[Theorem 10.2]{Dur_infcube}, though with some modest simplifications (see especially \cite[Subsection 9.1]{Dur_infcube}).

Since $\Delta|_{\calQ}$ is surjective and distance non-increasing, it suffices to prove that $d_{\calQ} \prec d_{\calQ'}$, and for this it suffices to bound the distances in each $U \in \calU$.

Let $\hx, \hy \in \calQ$ with $\hx' = \Delta(\hx)$ and $\hy' = \Delta(\hy)$.  We may assume that $\hx_U \neq \hy_U$ for any $U\in \calU$ that we will consider.  We begin by making a useful reduction.

Observe that item (9) of Lemma \ref{lem:collapsed interval control} implies that there are at most boundedly-many domains $U \in \calU$ for which $\hx_U,\hy_U$ are in the same component of $\hT^e_U$, where the bound is controlled by $\mathfrak S$.  For such $U$, we have $d_{\hT_U}(\hx_U, \hy_U) \leq d_{\hT'_U}(\hx'_U, \hy'_U) + B^2$ because at most $B$-many subtrees of diameter $B$ have been deleted from that edge.  Hence we may deal with these boundedly-many domains with an additive constant in the distance estimate comparing $d_{\calQ}(\hx,\hy)$ and $d_{\calQ'}(\hx',\hy')$. Hereafter, we assume that at least one of $\hx_U, \hy_U$ coincides with a cluster, or they are separated by at least one cluster.

\smallskip

There are two cases, (1) $U$ is $\nest_{\calU}$-minimal and (2) otherwise.

\smallskip

\underline{Case (1)}: Then $\hT^e_U$ has one component.  So there are at most $B$-many subtrees of diameter at most $B$ in $\hT^e_U$ being collapsed.  By item (9) of Lemma \ref{lem:collapsed interval control} there are at most boundedly-many (in $\mathfrak S$) such $\nest_{\calU}$-minimal $U$ for which one of $\hx_U, \hy_U$ is not a cluster point.  We may deal with such $U$ by increasing the additive constant in the distance estimate.

The general case is where $\hx_U, \hy_U$ are the distinct endpoint clusters of the interval $\hT_U$.  In this case, 
$$d_{\hT'_U}(\hx'_U,\hy'_U)  \geq d_{\hT_U}(\hx_U,\hy_U) - B^2 \geq K - B^2 - 2r_1,$$ by item (6) of Lemma \ref{lem:collapsed interval control}.   Hence by choosing $K = K(\mathfrak S)>0$ sufficiently large and requiring that $B^2_{tt}-2r_1<K$, we can guarantee that $d_{\hT'_U}(\hx'_U,\hy'_U) >1$ allowing the estimate
$$d_{\hT_U}(\hx_U,\hy_U) \leq d_{\hT'_U}(\hx'_U,\hy'_U) + B^2 \leq (1+B^2)d_{\hT'_U}(\hx'_U,\hy'_U).$$
Thus we can account for these domains via a bounded multiplicative error.

\smallskip
\underline{Case (2)}:  Suppose that $U \in \calU$ is not $\nest_{\calU}$-minimal.  Since $\hx_U \neq \hy_U$, there is some collection of edge components $E_1, \dots, E_k\subset \hT^e_U$ separating $\hx_U$ from $\hy_U$, i.e. intersecting $[\hx_U, \hy_U]_{\hT_U}$.  By assumption, each such component contains at most $B$-many subintervals in $\calB$ of diameter at most $B$ which are collapsed by $\Delta_U:\hT_U \to \hT'_U$.  Moreover, adjacent intervals $E_i, E_{i+1}$ are attached at some collapsed cluster $C_i$.  For each such collapsed cluster $C_i$, choose a $\nest_{\calU}$-minimal domain $V_i \nest U$ with $\hd^{V_i}_U \in C_i$.  Observe that
\begin{eqnarray*}
d_{\hT_U}(C_i,C_{i+1}) &\leq& d_{\hT'_U}(\Delta_U(C_i), \Delta_U(C_{i+1})) + B^2\\
&\leq& d_{\hT'_U}(\Delta_U(C_i), \Delta_U(C_{i+1})) + B^2d_{\hT'_{V_i}}(\hx'_{V_i}, \hy'_{V_i}).
\end{eqnarray*}

Hence we can account for the lost distance along $[\hx_U, \hy_U]$ between each pair of successive collapsed cluster points $C_i, C_{i+1}$ (i.e., the distance lost in the edge $E_i$) by distances in the $\nest_{\calU}$-minimal domains $V_i$.  Similarly, we can use $V_1, V_k$ to account for any lost distance between $\hx_U, C_1$ and $C_n, \hy_U$, respectively.  Since each domain in $\calU$ nests into at most boundedly-many domains in $\calU$ by the Covering Lemma \ref{lem:covering}, the proof of case (2) is complete.  This completes the proof of the theorem.
\end{proof}

\section{Proof of the Stable Intervals Theorem \ref{thm:stable intervals}} \label{app:SI}

We will give a fairly complete accounting of the two main arguments involved in Theorem \ref{thm:stable intervals}.  See Subsection \ref{subsec:detailed sketch} for a detailed sketch and outline.  The two key arguments involve showing that we can stably add one cluster point and then that we can iterate this process.

\subsection{Adding a cluster point}

In this subsection, we explain how to build a stable decomposition when adding a cluster point to the setup.  This is the main bit of work for Theorem \ref{thm:stable intervals}.

\begin{proposition} \label{prop:add cluster}
Let $\calZ$ be a geodesic $\delta$-hyperbolic space, $\epsilon>0$, and $r_1,r_2>0$ be positive integers.  Let $(a,b;\calY)$ be an $\epsilon$-setup and further fix a point $z \in \calN_{\epsilon/2}(\lambda(a,b))$.  Then
\begin{enumerate}
\item $(a,b;\calY \cup \{z\})$ is an $\epsilon$-setup;
\item The two setups $(a,b;\calY)$ and $(a,b;\calY \cup \{z\})$ are $(1,\epsilon)$-admissible;
\item There exists $L_1=L_1(\delta, \epsilon)>0$ so that the stable intervals $T$ and $T'$ for the setups $(a,b;\calY)$ and $(a,b;\calY \cup \{z\})$ admit $\calY$-stable $L_1$-compatible decompositions;
\item There exists $L_2 = L_2(\delta, \epsilon, r_1,r_2)>0$ so that the $(r_1,r_2)$-thickenings of $T,T'$ admit $\calY$-stable $L_2$-compatible decompositions.
\end{enumerate}
\end{proposition}

\begin{proof}
Since the points $a,b$ have not changed, the first two conclusions are automatic from the assumptions.

Choosing $E \gg \epsilon$, Proposition \ref{prop:cluster sep graph} implies that both $\calG = \calG_E(a,b;\calY)$ and $\calG' = \calG_E(a,b;\calY \cup \{z\})$ are intervals.  There are two main cases: (a) $d(z,C)>E$ for all clusters $C$, or (b) otherwise.

We will deal with these cases separately, though the arguments are similar.  Moreover, for each case, we will want to deal both with the original decompositions of the stable tree $T = T_e \cup T_c$, as well as its $(r_1,r_2)$-thickening $T = \bT_{e} \cup \bT_{c}$.  The compatible decompositions for the latter will be refined for the decomposition for the former, with possibly a slight increase in the error depending on $r_1,r_2$.

\medskip

\underline{\textbf{Case (a), $z$ forms a split cluster}}: In this case, $z$ forms its own cluster and, in particular, does not affect membership of any other cluster.  Moreover, it cannot be an endpoint cluster of $\calG'$, as those necessarily contain $a$ and $b$.  Hence there exist two clusters $C_1,C_2$ which are adjacent in $\calG$ but are separated by $z$ in $\calG'$, say with $C_1$ on the side of $a$ and $C_2$ on the side of $b$.  In this sense, we think of $z$ as splitting the clusters $C_1,C_2$.

Lemma \ref{lem:cluster arrangement} implies that $z$ separates two clusters $C,C'$ for $(a,b;\calY)$ if and only if $s(z)$ separates $s(C), s(C')$ along $\lambda(a,b)$.  It follows then that the subgraph of $\calG$ connecting the cluster containing $a$ to $C_1$ is identical to the subgraph of $\calG_E(a,b;\calY \cup \{z\})$, and similarly for the subgraphs connecting $C_2$ to the cluster containing $b$.  This says that for any pair $C,C'$ on the $a$-side of $z$ in $\calG'$, $\lambda_0(C,C')$ is an edge component of both $T$ and $T'$, and similarly for such pairs  on the $b$-side of $z$.  Thus we can include each such component of $E \subset T_e$ and $E \subset T'_e$ in the respective stable decomposition $T_s \subset T_e$ and $T'_s \subset T'_e$, where the identity map is the map described in parts \eqref{item:stable pairs} and \eqref{item:identical pairs} of Definition \ref{defn:stable decomp}.  By construction, each of these pieces has positive integer length, thus verifying property \eqref{item:integer length}.

To construct the rest of the stable decompositions, namely stable pieces of type \eqref{item:close pairs} in Definition \ref{defn:stable decomp}, we need to compare the component $\lambda_0(C_1, C_2) \subset T_e$ to the two components $\lambda_0(C_1,z) \cup \lambda_0(z, C_2) \subset T'_e$.  First, Lemma \ref{lem:cluster close} says that the endpoints of $\lambda(C_1,C_2)$ and $\lambda(C_1,z)$ on $C_1$ are $14\epsilon$-close, and hence the endpoints of $\lambda_0(C_1,C_2)$ and $\lambda_0(C_1,z)$ are within $14\epsilon+2$, and similarly for $C_2$.  It follows from $\delta$-hyperbolicity of $\calZ$ that $\lambda_0(C_1, C_2)$ and $ \lambda_0(C_1,z) \cup \lambda_0(z, C_2)$ are $O(\epsilon)$-Hausdorff close.  Now, unless $C_1,C_2$ are $O(\epsilon)$-close in $\calZ$---in which case we can declare all three of these geodesics to be unstable---then the shadow $s_T(z)$ of $z$ on $T$ will be properly contained in $\lambda_0(C_1,C_2)$ and have $O(\epsilon)$-bounded diameter.    Hence we can decompose these geodesics into $O(\epsilon)$-close stable pieces and unstable pieces labeled by $z$, as described in property \eqref{item:unstable components}, while satisfying the length requirement of \eqref{item:integer length} and proximity requirements of \eqref{item:identical pairs} and \eqref{item:close pairs}; see Figure \ref{fig:split cluster} for a heuristic picture.  For the bijection $\alpha:\pi_0(T_s) \to \pi_0(T'_s)$, since $T_s,T'_s$ have the same number of components, we can take let $\alpha$ denote the natural order-preserving bijection, i.e. the one obtained by moving along the geodesics $[a,b]_{T}$ and $[a,b]_{T'}$.  By construction, this bijection identifies the components we want so that the related conditions \eqref{item:identical pairs} and \eqref{item:close pairs} are satisfied.

Finally, we need to confirm the properties in \eqref{item:adjacency} for the order-preserving bijection $\beta:\pi_0(T - T_s) \to \pi_0 (T'-T'_s)$ induced by $\alpha$.  Since $\beta$ is order-preserving (with the order being ``from $a$ to $b$'', as in Definition \ref{defn:order preserving}), \eqref{item:endpoint condition}, i.e. that the components of $T-T_s$ and $T' - T'_s$ associated to $a$ and $b$ are identified, is automatically satisfied.  For \eqref{item:cluster identify}, note that $\calY \cap (\calY \cup \{z\}) = \calY$, and, by construction again, if $y \in \calY$, then $\beta$ identifies the component $D_y$ of $T- T_s$ corresponding to $y$ with $D'_y$, the component of $T'-T'_s$ corresponding to $y$.  In this case, this is because all such components are exactly the same.   This proves part (1) of the statement.

For part (2) of the statement, we want to consider how the $(r_1,r_2)$-thickenings of $T,T'$ along $T_c,T'_c$ respectively affect the decompositions $T_s \subset T_e$ and $T'_s \subset T'_e$.  First, observe that $\bT_e \subset T_e$ and $\bT'_e \subset T'_e$ by definition.  In particular, consider a component $D \subset T_e$, where $D = \lambda_0(C,C')$ for adjacent clusters $C,C'$.  Then the $(r_1,r_2)$-thickening of $D$ either truncates $D$ to $D_{thick} = D - (\mu_{r_1}(C) \cup \mu_{r_1}(C'))$, or possibly removes $D$ entirely from $\bT_e$ if $d_T(\mu(C), \mu(C'))<r_2-2r_1$.  A similar statement holds for components of $T'_e$.

By the above, most components of $T_s$ and $T'_s$ are identical components of $T_e, T'_e$, and hence their thickenings are exactly the same by the above.  For the remaining components, namely $\lambda_0(C_1,z), \lambda_0(z,C_2)$ of $T'_e$ and $\lambda_0(C_1,C_2)$ of $T_e$, each is either truncated or removed as above.  The remaining two components of the stable decomposition $T_s$ correspond to $\lambda_0(C_1,z)$ and $\lambda_0(z,C_2)$, and so one can construct the corresponding components of $\bT_s$ by applying the same procedure to the $(r_1,r_2)$-thickening of $\lambda_0(C_1,z)$ and $\lambda_0(z,C_2)$.  Note that this can inflate the proximity constants, but only in terms of $r_1,r_2$.  The rest of the proof goes through unchanged; see Figure \ref{fig:split cluster}.

\medskip

\underline{\textbf{Case (b), $z$ affects other clusters}}: Suppose now instead that there exists some cluster $C$ for $(a,b;\calY)$ so that $d(z,C)<E$; we say $C$ is an  \emph{affected cluster}.

\begin{claim}\label{claim:bounded affect}
If $3E - 56\epsilon > 2E$, then there are at most three affected clusters, and they are adjacent in $\calG$.
\end{claim}

\begin{proof}
Suppose that $C_1,\dots, C_4$ are four clusters appearing in that order along $\lambda(a,b)$.  By item (1) of Lemma \ref{lem:cluster close}, since $d(C_i, C_{i+1})\geq E$, we have $d(s(C_i), s(C_{i+1}))> E - 14\epsilon$, for $i = 1,2,3$.

Since the shadows are disjoint segments of the geodesic $\lambda(a,b)$, we have that
$$d(s(C_1), s(C_4)) \geq d(s(C_1), s(C_2)) + d(s(C_2), s(C_3))+d(s(C_3), s(C_4)) \geq 3E -42\epsilon.$$
Hence if $3E-56\epsilon \geq 2E$, then item (1) of Lemma \ref{lem:cluster close} says that 
$$d(C_1,C_4) \geq d(s(C_1),s(C_4)) - 14\epsilon \geq 3E-56\epsilon > 2E.$$

This is impossible if $z$ is $E$-close to both $C_1,C_4$.  Hence there are at most three affected clusters and they are adjacent vertices in $\calG$.
\end{proof}

Let $C'$ denote the cluster formed by $z$ and the affected clusters from $(a,b;\calY)$.  Note that the membership of all other $(a,b;\calY)$-clusters are unchanged by the addition of $z$.  Let $C_1$ and $C_2$ be the clusters adjacent to the affected clusters in $\calG$ on the side of $a$ and $b$, respectively.  As in the split cluster case above, Lemma \ref{lem:cluster arrangement} implies that $C_1$ separates all clusters on the $a$-side of $\calG$ from $C'$, and similarly for $C_2$ and clusters on the $b$-side.  Hence the subgraph of $\calG$ connecting the cluster containing $a$ to $C_1$ is the subgraph of $\calG'$ connecting the (same) cluster containing $a$ to $C_1$, and similarly for the $b$-side of $\calG$ and $\calG'$.  As such, all of the corresponding components of $T_e$ and $T'_e$ are exactly the same, and have positive integer length by construction, verifying properties \eqref{item:integer length}, \eqref{item:stable pairs}, and \eqref{item:identical pairs} of Definition \ref{defn:stable decomp} for these components.

The forests $T_e$ and $T'_e$ have some remaining components which we need to decompose.  In $T_e$, there are the at-most $2$ components of $T_e$ which connect the affected clusters.  By assumption, these segments are at most $O(\epsilon)$ long, so we can declare them to be unstable components.

We now explain how to handle the two remaining components.  Let $A$ be the affected cluster adjacent to $C_1$ in $\calG$.  It is enough to show that the geodesic segments $\lambda(C_1,A)$ and $\lambda(C_1,C')$ are $O(\epsilon)$-Hausdorff close, since then the lengths of $\lambda_0(C_1,A)$ and $\lambda_0(C_1,C')$ are the same up to error $O(\epsilon)$, and we can declare a small end segment of each unstable, so that their (stable) complements have the same (positive integer) length and are $O(\epsilon)$-close.  But this endpoint proximity statement is again a consequence of Lemma \ref{lem:cluster close}, since the endpoints of $\lambda(C_1,A)$ and $\lambda(C_1,C')$ are within $7\epsilon$ of the endpoint of $s(C_1)$ on $\lambda(a,b)$, while $s(C')$ contains $s(A)$ and its endpoint is at most $(E+\epsilon)$-away.  Finally, the bijection $\beta:\pi_0(T-T_s) \to \pi_0(T'-T'_s)$ required for \eqref{item:adjacency} is defined similar to the split cluster case (a), with the proof of condition \eqref{item:endpoint condition} being basically the same.  For condition \eqref{item:cluster identify}, observe again the $\calY \cap (\calY \cup \{z\}) = \calY$.  Assuming that $y \in \calY$ is contained in an affected cluster in the $(\calY; \{a,b\})$ setup---since otherwise, it is contained in an identical cluster in the $(\calY \cup \{z\};\{a,b\})$ setup---then it corresponds to the single component which gets identified with the component of $T'- T'_s$ corresponding to the single cluster containing both $y$ and $z$.  This completes the proof of (1) in the affected cluster case.  See Figure \ref{fig:affected cluster} for a heuristic picture.

The argument for (2) in the affected cluster case is essentially identical to the split cluster version: All but a few components of $T_s \subset T_e$ and $T'_s \subset T'_e$ are identical and thus have identical $(r_1,r_2)$-thickenings.  And the remaining components of $T_s, T'_s$ are constructed by pairing the corresponding components of $T_e, T'_e$ adjacent to or connecting the affected clusters, and then either declaring them unstable or modifying them pairwise to achieve the length bound.  As above, this can be accomplished post-thickening; see Figure \ref{fig:affected cluster}.  This completes the proof of the theorem.

\end{proof}

\subsection{Iterating stable decompositions}

Next we move to proving our iteration statement (Proposition \ref{prop:stable iteration}).  The main work is the base case of combining two pairs of stable decompositions across a common setup.

For this, suppose we have three setups $(a_i,b_i;\calY_i)$ for $i=1,2,3$, satisfying that $\calY_i \subset \calN_{\epsilon/2}(\lambda(a_j,b_k))$ for all $i,j,k$, and that $\calY_0 \subset \calY_i$ for each $i$.  Suppose that the stable intervals associated to $(a_1,b_1;\calY_1)$ and $(a_2,b_2;\calY_2)$ admits $\calY_0$-stable $L$-compatible decompositions, and the same for $(a_2,b_2;\calY_2)$ and $(a_3,b_3;\calY_3)$.  Our goal is to induce new stable decompositions on $T_1,T_2$ and $T_2,T_3$ which encode the stable decompositions from $T_2,T_3$ and $T_1,T_2$, respectively.  The following describes the properties we want for these induced decompositions:

\begin{definition}[Refined stable decompositions]\label{defn:refined decomp}
Let $(a,b;\calY)$ and $(a',b';\calY')$ denote two admissible $\epsilon$-setups and let $\calY_0 = \calY \cap \calY'$.  Suppose we have two compatible $\calY_0$-stable decompositions $T_s \subset T$ and $T'_s \subset T'$.  We say another pair of stable decompositions $T_{s,0} \subset T$ and $T'_{s,0} \subset T'$ \emph{refine} $T_s \subset T$ and $T'_s \subset T'$ if the following hold:
\begin{enumerate}
\item We have $T_{s,0} \subset T_s$ and $T'_{s,0} \subset T'_s$, and
\item Let $\beta:\pi_0(T - T_s) \to \pi_0(T'-T'_s)$ and $\beta_0:\pi_0(T-T_{s,0}) \to \pi_0(T' - T'_{s,0})$ denote the associated order-preserving bijections.  Suppose $C_1 \in \pi_0(T - T_s)$  and let $C_{1,0} \in \pi_0(T-T_{s,0})$ denote the component containing $C_1$ (which exists because $T_{s,0} \subset T_s$).  Then $\beta(C_1) \subset \beta_0(C_{1,0})$.
\end{enumerate}
\end{definition}

\begin{remark}\label{rem:purpose of refined}
Condition (2) in particular is used to verify the properties in \eqref{item:adjacency} of Definition \ref{defn:stable decomp} during our iteration procedure.  For item \eqref{item:endpoint condition}, since $\beta$ identifies the component of $T-T_s$ containing the cluster component associated to $a$ to the component of $T' - T'_s$ containing the cluster component associated to $a'$, the refinement condition (2) forces the same to hold for $\beta_0$.  A similar observation works to verify item \eqref{item:cluster identify}.
\end{remark}

Following our notation from before the above definition, let $\phi_i:T_i \to \calZ$ denote the maps into $\calZ$.  Moreover, assume that we have
\begin{itemize}
\item a $\calY_0$-stable decomposition $T^1_{s,2} \subset T_{e,2}$ which is $L$-compatible with $T^2_{s,1} \subset T_{e,1} $ for the setups $(a_1,b_1;\calY_1)$ and $(a_2,b_2;\calY_2)$, and
\item a $\calY_0$-stable decomposition $T^3_{s,2} \subset T_{e,2}$ which is $L$-compatible with $T^2_{s,3} \subset T_{e,3}$ for the setups $(a_2,b_2;\calY_2)$ and $(a_3,b_3;\calY_3)$.
\end{itemize}

\begin{proposition}\label{prop:stable iteration, pair}
There exists $\calY_0$-stable $4L^2$-compatible decompositions $T^{2,3}_{s,1} \subset T_{e,1}$ and $T^{2,1}_{s,3} \subset T_{e,3}$ which refine $T^2_{s,1}$ and $T^2_{s,3}$. 
\end{proposition}

\begin{proof}
Note that the stable decompositions $T^1_{s,2}$ and $T^3_{s,2}$ both live on $T_2$.  Set $T^{1,3}_{s,2} = T^1_{s,2} \cap T^3_{s,2}$.  Observe that $T_{e,2} - T^{1,3}_{s,2} = (T_{e,2} - T^1_{s,2}) \cup (T_{e,2} - T^3_{s,2})$ has at most $2L$-many components each of diameter at most $4L^2$; note that the larger diameter bound accounts for unstable components to combine, but there are at most $2L$-many of them.  Moreover, by item \eqref{item:simplicial} of Definition \ref{defn:stable decomp}, the endpoints of each component of $T^1_{s,2}$ and $T^3_{s,2}$ are at vertices of the components of $T_{e,2}$ which contain them, and hence this property holds for each component of $T^{1,3}_{s,2}$

We first produce refined decompositions of $T^{2,3}_{s,1} \subset T^2_{s,1}$ and $T^{1,3}_{s,2} \subset T^1_{s,2}$, and the same argument produces the refined decompositions $T^{1,3}_{s,2}  \subset T^3_{s,2}$ and $T^{2,1}_{s,3} \subset T^2_{s,3}$.  We will then see the properties of Definition \ref{defn:refined decomp} allow us to combine these, using $T^{1,3}_{s,2}$ as a bridge to connect them.

We can induce a decomposition $T^{2,3}_{s,1}$ on $T_1$ by taking the collection of all $i^{-1}_{E,E'}(E' \cap T^{1,3}_{s,2})$ over all stable pairs $(E,E')$ of $T_1,T_2$ as identified by the bijection $\alpha_{1,2}:\pi_0(T_{s,1}) \to \pi_0(T_{s,2})$.  Observe that any component $V \subset T^{2,3}_{s,1}$ is contained in some component $V \subset E \subset T_{s,1} \subset T_{e,1}$, and the map $i_{E,E'}|_V:V \to i_{E,E'}(V)$ is an isometry, where $E' = \alpha_{1,2}(E)$, and vice versa.  Thus the components of $T^{2,3}_{s,1}$ are in bijective correspondence with the components of $T^{1,3}_{s,2}$, with identified components being isometric.  Moreover,  the natural order-preserving bijection $\alpha^3_{1,2}:\pi_0(T^{2,3}_{s,1}) \to \pi_0(T^{1,3}_{s,2})$ (as in Definition \ref{defn:order preserving}) coincides with the bijective correspondence induced from $\alpha_{1,2}$ and the various order-preserving isometries $i_{E,E'}$.  We also note that the various restrictions $i_{E,E'}|V$ are order-preserving since $i_{E,E'}$ is.  Next observe that the images $\phi_1(V)$ and $\phi_2(i_{E,E'}(V))$ are either exactly the same or are $L$-Hausdorff close in $\calZ$ as items \eqref{item:identical pairs} and \eqref{item:close pairs} of Definition \ref{defn:stable decomp}.  Thus the decompositions $T^{2,3}_{s,1} \subset T_1$ and $T^{1,3}_{s,2} \subset T_2$ satisfy all the properties of Definition \ref{defn:stable decomp} except possibly \eqref{item:adjacency}.  We also want to show that they give a refinement as in Definition \ref{defn:refined decomp}.

To prove \eqref{item:adjacency}, it will suffice to show the refinement property (Definition \ref{defn:refined decomp}), as explained in Remark \ref{rem:purpose of refined}.  Property (1), namely that $T^{2,3}_{s,1} \subset T^2_{s,1}$ and $T^{1,3}_{s,2} \subset T^1_{s,2}$, of Definition \ref{defn:refined decomp} holds by construction.  For property (2), we argument one component at a time.  In particular, let $\beta^3_{1,2}:\pi_0(T_1 - T^{2,3}_{s,1}) \to \pi_0(T_2 - T^{1,3}_{s,2})$ be the natural order-preserving bijections (which exists via Remark \ref{rem:order preserving} because $\alpha^3_{1,2}$ does).

First, since the (unstable) components of $T_1 - T^2_{s,1}$ only get expanded when passing to $T_1 - T^{2,3}_{s,1}$, it follows that the first component $C_0 \subset T_1 - T^2_{s,1}$ one encounters when moving from $a_1$ to $b_1$ along the geodesic $[a_1,b_1]_{T_1}$ is contained in the first such component of $C_{0,3} \subset T_1 - T^{2,3}_{s,1}$, and similarly for the setup on $T_2$.  Hence these expanded clusters are identified by $\beta^3_{1,2}$, which we note also confirms item \eqref{item:endpoint condition} of Definition \ref{defn:stable decomp}.

We now proceed inductively along $T_1$ and $T_2$.  Let $C_1, C'_1$ denote the second pairs of components of $T_1 - T^2_{s,1}$ and $T_2 - T^1_{s,2}$ encountered after $C_0, C'_0$.  Let $E, E'$ denote the respective components of $T^2_{s,1}$ and $T^1_{s,2}$ connecting $C_0,C'_0$ to $C_1,C'_1$.  The components of $(T_1 - T^{2,3}_{s,1}) \cap E$ and $(T_2 - T^{1,3}_{s,2}) \cap E'$ are identified by the order-preserving isometry $i_{E,E'}$ in an order-preserving way.  In particular, the two sets of components are in bijective correspondence.  As a consequence, if $C^3_1 \subset T_1 - T^{2,3}_{s,1}$ is  the component containing $C_1$, then $C'_1 \subset \beta^{3}_{1,2}(C^3_1)$, as required.  One can repeat this procedure on subsequent components of $T_1 - T^2_{s,1}$ and their paired components of $T_2 - T^1_{s,2}$ to see that $T^{2,3}_{s,1}$ and $T^{1,3}_{s,2}$ refine $T^{2}_{s,1}$ and $T^1_{s,2}$, respectively.

\medskip

\textbf{\underline{Combining the decompositions}}: We now want to see that $T^{2,3}_{s,1}$ and $T^{2,1}_{s,3}$ are $\calY_0$-stable $4L^2$-compatible decompositions for $(a_1,b_1;\calY_1)$ and $(a_3,b_3;\calY_3)$.  For this, define $\alpha_{1,3}:\pi_0(T^{2,3}_{s,1}) \to \pi_0(T^{2,1}_{s,3})$ to be $\alpha_{1,3} = \alpha^{1}_{2,3} \circ \alpha^3_{1,2}$, and similarly set $\beta_{1,3} = \beta^1_{2,3} \circ \beta^3_{1,2}$.  These compositions are order-preserving bijections because their constituent maps are, and moreover we can similarly induce order-preserving isometries on the components of $T^{2,3}_{s,1}$ and $T^{2,1}_{s,3}$ identified by $\alpha_{1,3}$, which satisfy the proximity conditions of \eqref{item:identical pairs} and \eqref{item:close pairs} of Definition \ref{defn:stable decomp} with a constant of $2L$.  In particular, it remains to show the properties in \eqref{item:adjacency}.

For this, we use the refinement property (2) of Definition \ref{defn:refined decomp}.  This is basically automatic.  For (1), the refinement property implies that $\beta^3_{1,2}$ identifies the component of $T_1 - T^{2,3}_{s,1}$ associated to $a_1$ to the component of $T_2 - T^{1,3}_{s,2}$ associated to $a_2$, and similarly for $\beta^1_{2,3}$.  Hence $\beta_{1,3}$ identifies the components associated to $a_1, a_3$ and $b_1,b_3$.  Similarly, if $y \in \calY_0 \subset \calY_1 \cap \calY_2 \cap \calY_3$, then the refinement property implies that $\beta^3_{1,2}$ identifies the component of $T_1 - T^{2,3}_{s,1}$ associated to $y$ to the component of $T_2 - T^{1,3}_{s,2}$ associated to $y$, and similarly for $\beta^1_{2,3}$.  Hence $\beta_{1,3}$ associates components associated to $y$ in both of $T_1 - T^{2,3}_{s,1}$ and $T_3 - T^{2,1}_{s,3}$.  This completes the proof of the proposition.

\end{proof}

The following proposition is now a simple iterative application of Proposition \ref{prop:stable iteration, pair}:

\begin{proposition}\label{prop:stable iteration}
For every $n\geq 2$, there exists $L_n = L_n(L,n)>0$ so that the following holds.  Suppose $(a_1,b_1;\calY_1), \dots, (a_n,b_n;\calY_n)$ is a sequence of pairwise admissible $\epsilon$-setups and suppose that $\calY_0 \subset \bigcap_i \calY_i$.  If each pair of consecutive setups admits $\calY_0$-stable $L$-compatible stable decompositions, then $(a_1,b_1;\calY_1)$ and $(a_n,b_n;\calY_n)$ admit $\calY_0$-stable $L_n$-compatible stable decompositions. 
\end{proposition}

\subsubsection{Completing the proof} To prove Theorem \ref{thm:stable intervals}, we have two remaining tasks:
\begin{enumerate}
\item We need to explain how to build a stable decomposition when perturbing an endpoint but keeping the rest of the setup the same, and
\item Show how to stick this into an iterative setup.
\end{enumerate}

\textbf{\underline{Perturbing an endpoint}:} This is an iterated application of  Proposition \ref{prop:add cluster}.

Consider the following restricted setup, which we highlight to focus on what happens when we perturb an endpoint.  Let $a,b,a' \in \calZ$ and $\calY \subset \calZ$ be finite with $\calY \subset \calN_{\epsilon/2}(\lambda(a,b)) \cap \calN_{\epsilon/2}(a',b)$.  Consider the following collection of setups, each of which differs by adding or deleting one cluster point, or by switching the role that $a,a'$ or $b,b'$ play between cluster or endpoint:
\vspace{-.09in}
\begin{multicols}{3}
\begin{enumerate}
\item $(a,b;\calY)$,
\item $(a,b;\calY \cup \{a\})$,
\item $(a,b;\calY \cup \{a, a'\})$,
\item $(a',b;\calY \cup \{a,a'\})$,
\item $(a',b;\calY \cup \{a'\})$,
\item $(a',b;\calY)$.
\end{enumerate}
\end{multicols}

First observe that $\calY$ is contained in each of the above sets of cluster points, so we want to see how to build $\calY$-stable decompositions for each, which implies the existence of $\calY_0$-stable decompositions for any $\calY_0 \subset \calY$.  Next, observe that in the construction of the stable interval $T(a,b;\calY)$, the roles of points of the endpoints $a,b$ and the cluster points $\calY$ are equivalent in cluster formation process.  The only special role that endpoints play in the Definition \ref{defn:stable interval} of $T(a,b;\calY)$ is in the segments defined \emph{within} clusters, namely the $\mu(C)$ components.  In particular, the edge components (i.e., the components of $T_e$) treat endpoints and cluster points equivalently.  As the desired stable decompositions from Theorem \ref{thm:stable intervals} are subsets of the edge components, we are free to treat the points $a,a',b,b'$ as cluster points in applying Proposition \ref{prop:add cluster} to obtain stable decompositions for each of the steps $(1) \to (2) \to (3)$ and $(4) \to (5) \to (6)$.  The only exception is the step $(3) \to (4)$, where we are switching the roles of $a$ and $a'$.  Here Proposition \ref{prop:add cluster} does not directly apply, but, by the above discussion, these two setups have the exact same clusters and thus the exact same edge components.  Hence we get the desired bijections (with their proximity constraints) by identifying the edge components as stable components---which, again, are all exactly the same---and taking the bijection between the complementary components to be the obvious map.

In other words, provided an explanation for how to iteratively combine a sequence of stable decompositions, we can apply this process to move from setup (1) above to setup (6).

\medskip

\textbf{\underline{Putting it all together}}: In the more general setting of the theorem, the above sequence only appears in the middle of a longer sequence we need to consider.  In particular, let $\calY' - \calY = \{x_1, \dots, x_n\}$ and $\calY - \calY' = \{z_1, \dots, z_m\}$, where $n,m < N$ by assumption.  Set $\calY_i = \calY \cup \{x_1, \dots, x_i\}$ and $\calY'_i = \calY' \cup \{z_1, \dots, z_i\}$, and let $\calY_0 = \calY \cap \calY'$ and $\calY_{all} = \calY \cup \calY'$.  Then, by the assumptions in our setup, we have two sequences of pairwise admissible $\epsilon$-setups
$$(a,b;\calY) \to (a,b;\calY_1) \to \cdots \to (a,b;\calY_{all}) \to (a,b;\calY_{all} \cup \{a'\}) \to (a,b;\calY_{all} \cup \{a',b'\})$$
and
$$(a',b';\calY') \to \cdots \to (a',b'; \calY_n) \to (a',b';\calY_{all}) \to (a',b';\calY_{all} \cup \{a\}) \to (a',b';\calY_{all} \cup \{a,b\}),$$
which we note are connected in the middle by the above discussion, where we need to switch the roles of both $a,a'$ and $b,b'$.  In particular, the stable trees for the setups $(a,b;\calY_{all} \cup \{a',b'\})$ and $(a',b';\calY_{all} \cup \{a,b\})$ have identical edge components.

By Proposition \ref{prop:stable iteration}, we get $\calY_0$-stable $L_{n+2}$-compatible stable decompositions for $(a,b;\calY)$ and $(a,b;\calY_{all} \cup \{a',b'\})$, and $\calY_0$-stable $L_{m+2}$-compatible stable decompositions for $(a',b';\calY)$ and $(a',b';\calY_{all} \cup \{a,b\})$.  Since $(a,b;\calY_{all} \cup \{a',b'\})$ and $(a',b';\calY_{all} \cup \{a,b\})$ determine identical edge components, we can use an identical argument to the proof of Proposition \ref{prop:stable iteration, pair} to induce refined $\calY_0$-stable $(L_n \cdot L_m)$-compatible stable decompositions for $(a,b;\calY)$ and $(a',b';\calY)$.  We leave the details to the reader.  This completes the proof. \qed

\bibliography{references}{}
\bibliographystyle{alpha}

 \end{document}